\documentclass[11pt]{amsart}
\usepackage{hyperref}
\usepackage{color}
\definecolor{mylinkcolor}{rgb}{0.5,0.0,0.0}
\definecolor{myurlcolor}{rgb}{0.0,0.0,0.75}
\hypersetup{colorlinks=true,urlcolor=myurlcolor,citecolor=myurlcolor,linkcolor=mylinkcolor,linktoc=page,breaklinks=true}
\usepackage{fullpage}
\usepackage{amsmath,mathtools,bbm}
\usepackage{booktabs,array}
\usepackage[shortlabels]{enumitem}
\usepackage{animate}
\usepackage[charter]{mathdesign}
\usepackage{nicefrac}

\newcommand{\Aseq}[1]{\href{http://oeis.org/#1}{#1}}
\newcommand{\animatecaption}[1]{{\small Note: click to animate (requires adobe reader), or visit \url{#1}}}

\newcommand{\ifunc}{\mathbbm{1}} 
\newcommand{\A}{\mathbf{A}}
\newcommand{\I}{\mathbf{I}}
\newcommand{\F}{\mathbf{F}}
\newcommand{\Fp}{\F_p}
\newcommand{\Fpbar}{\overline\F_p}
\newcommand{\Fq}{\F_q}
\newcommand{\Fqbar}{\overline\F_q}
\newcommand{\Z}{\mathbf{Z}}

\newcommand{\Q}{\mathbf{Q}}
\newcommand{\R}{\mathbf{R}}
\newcommand{\C}{\mathbf{C}}
\newcommand{\Qbar}{\overline{\Q}}
\newcommand{\Kbar}{\overline{K}}
\newcommand{\kbar}{\overline{k}}
\newcommand{\Aut}{{\rm Aut}}
\newcommand{\Gal}{{\rm Gal}}
\newcommand{\Jac}{\operatorname{Jac}}
\newcommand{\MT}{\operatorname{MT}}

\newcommand{\Gm}{\mathbf G_m}
\newcommand{\Hg}{\operatorname{Hg}}
\newcommand{\AST}{\operatorname{AST}}
\newcommand{\ST}{\operatorname{ST}}
\newcommand{\GT}{\operatorname{GT}}
\renewcommand{\SS}{\mathbf S}
\newcommand{\GL}{{\rm GL}}
\newcommand{\SL}{{\rm SL}}
\renewcommand{\O}{{\rm O}}
\newcommand{\SO}{{\rm SO}}

\newcommand{\USp}{{\rm USp}}
\newcommand{\GSp}{{\rm GSp}}
\newcommand{\Sp}{{\rm Sp}}

\newcommand{\End}{{\rm End}}
\newcommand{\tr}{\operatorname{tr}}
\newcommand{\disc}{\operatorname{disc}}
\newcommand{\Frob}{\operatorname{Frob}}
\newcommand{\p}{\mathfrak{p}}
\newcommand{\q}{\mathfrak{q}}
\newcommand{\f}{\mathfrak{f}}
\newcommand{\im}{\operatorname{im}}
\renewcommand{\Im}{\operatorname{Im}}
\newcommand{\smallmat}[4]{\left(\begin{smallmatrix}#1&#2\\#3&#4\end{smallmatrix}\right)}
\newcommand{\Exp}{\mathrm{E}}
\newcommand{\Mom}{\mathrm{M}}
\newcommand{\Spec}{\operatorname{Spec}}
\newcommand{\Pic}{\operatorname{Pic}}
\newcommand{\SU}{\operatorname{SU}}
\newcommand{\U}{\operatorname{U}}

\renewcommand{\P}{\mathbf{P}}
\newcommand{\conj}{\operatorname{conj}}
\renewcommand{\Re}{\operatorname{Re}}

\newcommand{\cyc}[1]{{\mathrm{C}_#1}}

\newcommand{\Res}{\operatorname{Res}}

\newtheorem{theorem}{Theorem}[section]
\newtheorem{conjecture}[theorem]{Conjecture}
\newtheorem{corollary}[theorem]{Corollary}
\newtheorem{lemma}[theorem]{Lemma}
\newtheorem{proposition}[theorem]{Proposition}

\theoremstyle{definition}
\newtheorem{definition}[theorem]{Definition}
\newtheorem{algorithm}[theorem]{Algorithm}
\newtheorem{example}[theorem]{Example}
\newtheorem{remark}[theorem]{Remark}
\newtheorem{background}[theorem]{Background}

\newtheorem{exercise}{Exercise}
\numberwithin{exercise}{section}

\title{Sato-Tate Distributions}
\author{{\small Andrew V. Sutherland}}
\thanks{The author was supported by NSF grants DMS-1115455 and DMS-1522526.}

\begin{document}

\begin{abstract}
In this expository article we explore the relationship between Galois representations, motivic $L$-functions, Mumford-Tate groups, and Sato-Tate groups, and we give an explicit formulation of the Sato-Tate conjecture for abelian varieties as an equidistribution statement relative to the Sato-Tate group.  We then discuss the classification of Sato-Tate groups of abelian varieties of dimension $g\le 3$ and compute some of the corresponding trace distributions.  This article is based on a series of lectures presented at the 2016 Arizona Winter School held at the Southwest Center for Arithmetic Geometry.
\end{abstract}

\maketitle

\section{An introduction to Sato-Tate distributions}\label{lec:intro}

Before discussing the Sato-Tate conjecture and Sato-Tate distributions in the context of abelian varieties, let us first consider the more familiar setting of Artin motives (varieties of dimension zero).

\subsection{A first example}\label{sec:wtzero}

Let $f\in \Z[x]$ be a squarefree polynomial of degree $d$.  For each prime $p$, let $f_p\in(\Z/p\Z)[x]\simeq\F_p[x]$ denote the reduction of $f$ modulo $p$, and define
\[
N_f(p)\coloneqq\#\{x\in \Fp: f_p(x)=0\},
\]
which we note is an integer between $0$ and $d$.  We would like to understand how $N_f(p)$ varies with $p$.
The table below shows the values of $N_f(p)$ when $f(x)=x^3-x+1$ for primes $p\le 60$:
\smallskip

\begin{center}
\begin{tabular}{rrrrrrrrrrrrrrrrrrr}
\toprule
$p:$ & 2& 3& 5& 7& 11& 13& 17& 19& 23& 29& 31& 37& 41& 43& 47& 53& 59\\
$N_f(p)$ & 0& 0& 1& 1& 1& 0& 1& 1& 2& 0& 0& 1& 0& 1& 0& 1& 3\\
\bottomrule
\end{tabular}
\end{center}
\smallskip

There does not appear to be any obvious pattern (and we should know not to expect one, because the Galois group of $f$ is nonabelian).
The prime $p=23$ is exceptional because it divides $\disc(f)=-23$, which means that $f_{23}(x)$ has a double root.
As we are interested in the distribution of $N_f(p)$ as~$p$ tends to infinity, we are happy to ignore such primes, which are necessarily finite in number.

This tiny dataset does not tell us much.  Let us now consider primes $p\le B$ for increasing bounds $B$, and compute the proportions
$c_i(B)$ of primes $p\le B$ with $N_f(p)=i$.  We obtain the following statistics:
\smallskip

\begin{center}
\begin{tabular}{lcccc}
$B$ & $c_0(B)$ &  $c_1(B)$ &  $c_2(B)$ &  $c_3(B)$\\\midrule
$10^3$ & 0.323353 & 0.520958 & 0.005988 & 0.155689\\
$10^4$ & 0.331433 & 0.510586 & 0.000814 & 0.157980\\
$10^5$ & 0.333646 & 0.502867 & 0.000104 & 0.163487\\
$10^6$ & 0.333185 & 0.500783 & 0.000013 & 0.166032\\
$10^7$ & 0.333360 & 0.500266 & 0.000002 & 0.166373\\
$10^8$ & 0.333337 & 0.500058 & 0.000000 & 0.166605\\
$10^9$ & 0.333328 & 0.500016 & 0.000000 & 0.166656\\
$10^{12}$ & 0.333333 & 0.500000 & 0.000000 & 0.166666\\\midrule
\end{tabular}
\end{center}
This leads us to conjecture that the following limiting values $c_i$ of $c_i(B)$ as $B\to\infty$ are
\[
c_0=1/3,\quad c_1=1/2,\quad c_2=0,\quad c_3=1/6.
\]

There is of course a natural motivation for this conjecture (which is, in fact, a theorem), one that would allow us to correctly predict the asymptotic ratios $c_i$ without needing to compute any statistics.
Let us fix an algebraic closure $\Qbar$ of~$\Q$.
The absolute Galois group $\Gal(\Qbar/\Q)$ acts on the roots of $f(x)$ by permuting them.
This allows us to define the \emph{Galois representation} (a continuous homomorphism)
\[
\rho_f\colon \Gal(\Qbar/\Q)\to \GL_d(\C),
\]
whose image is a subgroup of the permutation matrices in $\mathrm{O}_d(\C)\subseteq \GL_d(\C)$; here $\mathrm{O}_d$ denotes the orthogonal group (we could replace~$\C$ with any field of characteristic zero).
Note that $\Gal(\Qbar/\Q)$ and $\GL_d(\C)$ are topological groups (the former has the Krull topology), and homomorphisms of topological groups are understood to be continuous.
In order to associate a permutation of the roots of $f(x)$ to a matrix in $\GL_d(\C)$ we need to fix an ordering of the roots; this amounts to choosing a basis for the vector space~$\C^d$, which means that our representation~$\rho_f$ is really defined only up to conjugacy.

The value $\rho_f$ takes on $\sigma\in\Gal(\Qbar/\Q)$ depends only on the restriction of $\sigma$ to the splitting field $L$ of~$f$, so we could restrict our attention to $\Gal(L/\Q)$.
This makes $\rho_f$ an \emph{Artin representation}: a continuous representation $\Gal(\Qbar/\Q)\to \GL_d(\C)$ that factors through a finite quotient (by an open subgroup).  But in the more general settings we wish to consider this may not always be true, and even when it is, we typically will not be given $L$; it is thus more convenient to work with $\Gal(\Qbar/\Q)$.

To facilitate this approach, we associate to each prime~$p$ an \emph{absolute Frobenius element}
\[
\Frob_p\in \Gal(\Qbar/\Q)
\]
that may be defined as follows.
Fix an embedding $\Qbar$ in $\Qbar_p$ and use the valuation ideal $\mathfrak P$ of $\Qbar_p$ (the maximal ideal of its ring of integers) to define a compatible system of primes $\q_L\coloneqq\mathfrak P\cap L$, where $L$ ranges over all finite extensions of $\Q$.
For each prime $\q_L$, let $D_{\q_L}\subseteq \Gal(L/\Q)$, denote its decomposition group, $I_{\q_L}\subseteq D_{\q_L}$ its inertia group, and $\F_{\q_L}\coloneqq\Z_L/\q_L$ its residue field, where $\Z_L$ denotes the ring of integers of~$L$.
Taking the inverse limit of the exact sequences
\[
1 \to I_{\q_L} \to D_{\q_L}\to \Gal(\F_{\q_L}/\Fp)\to 1
\]
over finite extensions $L/\Q$ ordered by inclusion gives an exact sequence of profinite groups
\[
1\to I_p \to D_p\to \Gal(\Fpbar/\Fp)\to 1.
\]
We now define $\Frob_p\in D_p\subseteq  \Gal(\Qbar/\Q)$ by arbitrarily choosing a preimage of the Frobenius automorphism $x\to x^p$ in $\Gal(\Fpbar/\Fp)$ under the map in the exact sequence above.
We actually made two arbitrary choices in our definition of $\Frob_p$, since we also chose an embedding of $\Qbar$ into~$\Qbar_p$.
Our absolute Frobenius element $\Frob_p$ is thus far from canonical, but it exists.
Its key property is that if $L/\Q$ is a finite Galois extension in which~$p$ is unramified, then the conjugacy class $\conj_L(\Frob_p)$ in $\Gal(L/\Q)$ of the restriction of $\Frob_p\colon\Qbar\to\Qbar$ to $L$ is uniquely determined, independent of our choices; note that when $p$ is unramified, $I_p$ is trivial and $D_p\simeq \Gal(\Fpbar/\Fp)$.
Everything we have said applies \emph{mutatis mutandi} if we replace $\Q$ by a number field~$K$: put $\Kbar\coloneqq \Qbar$, replace $p$ by a prime $\p$ of $K$ (a nonzero prime ideal of $\Z_K$), and replace $\Fp$ by the residue field $\F_\p\coloneqq \Z_K/\p$.

We now make the following observation: for any prime $p$ that does not divide $\disc(f)$ we have
\begin{equation}\label{eq:Nfp}
N_f(p) = \tr\rho_f(\Frob_p).
\end{equation}
This follows from the fact that the trace of a permutation matrix counts its fixed points.
Since $p$ is unramified in the splitting field of $f$, the inertia group $I_p\subseteq \Gal(\Qbar/\Q)$ acts trivially on the roots of $f(x)$, and the action of $\Frob_p$ on the roots of $f(x)$ coincides (up to conjugation) with the action of the Frobenius automorphism $x\to x^p$ on the roots of $f_p(x)$, both of which are described by the permutation matrix $\rho_f(\Frob_p)$.
The Chebotarev density theorem implies that we can compute $c_i$ via \eqref{eq:Nfp} by counting matrices in $\rho_f(\Gal(\Qbar/\Q))$ with trace~$i$, and it is enough to determine the trace and cardinality of each conjugacy class.

\begin{theorem}\label{thm:chebotarev}\textsc{Chebotarev Density Theorem}
Let $L/K$ be a finite Galois extension of number fields with Galois group $G\coloneqq\Gal(L/K)$.
For every subset $C$ of $G$ stable under conjugation we have
\[
\lim_{B\to\infty}\frac{\#\{N(\p)\le B: \conj_L(\Frob_\p) \subseteq C\}}{\#\{N(\p)\le B\}} = \frac{\#C}{\#G},
\]
where $\p$ ranges over primes of $K$ and $N(\p)\coloneqq \#\F_\p$ is the cardinality of the residue field $\F_\p\coloneqq \Z_K/\p$.
\end{theorem}
\begin{proof}
See Corollary~\ref{cor:artinL} in Section \ref{lec:equidistribution}.
\end{proof}

\begin{remark}\label{rem:chebotarev}
In Theorem~\ref{thm:chebotarev} the asymptotic ratio on the left depends only on primes of inertia degree 1 (those with prime residue field), since these make up all but a negligible proportion of the primes~$\p$ for which $N(\p)\le B$.  Taking $C=\{1_G\}$ shows that a constant proportion of the primes of $K$ split completely in $L$ and in particular have prime residue fields; this special case is already implied by the Frobenius density theorem, which was proved much earlier (in terms of Dirichlet density).
In our statement of Theorem~\ref{thm:chebotarev} we do not bother to exclude primes of~$K$ that are ramified in~$L$ because no matter what value $\conj_L(\Frob_\p)$ takes on these primes it will not change the limiting ratio.
\end{remark}

In our example with $f(x)=x^3-x+1$, one finds that $G_f:=\rho_f(\Qbar/\Q)$ is isomorphic to $S_3$, the Galois group of the splitting field of $f(x)$.  Its three conjugacy classes are represented by the matrices
\[
\begin{bmatrix}0&1&0\\0&0&1\\1&0&0\end{bmatrix},\qquad
\begin{bmatrix}1&0&0\\0&0&1\\0&1&0\end{bmatrix},\qquad
\begin{bmatrix}1&0&0\\0&1&0\\0&0&1\end{bmatrix},\qquad
\]
with traces 0, 1, 3.  The corresponding conjugacy classes have cardinalities 2, 3, 1, respectively, thus
\[
c_0=1/3,\quad c_1=1/2,\quad c_2=0,\quad c_3=1/6,
\]
as we conjectured.

If we endow the group $G_f$ with the discrete topology it becomes a compact group, and therefore has a \emph{Haar measure} $\mu$ that is uniquely determined once we normalize it so that $\mu(G_f)=1$ (which we always do).
Recall that the Haar measure of a compact group $G$ is a translation-invariant Radon measure (so $\mu(gS)=\mu(Sg)=\mu(S)$ for any measurable set $S$ and $g\in G$), and is unique up to scaling.\footnote{For locally compact groups $G$ one distinguishes left and right Haar measures, but the two coincide when $G$ is compact; see \cite{DS14} for more background on Haar measures.}
For finite groups the Haar measure $\mu$ is just the normalized counting measure.
We can compute the expected value of trace (and many other statistical quantities of interest) by integrating against the Haar measure, which in this case amounts to summing over the finite group $G_f$:
\[
\Exp[\tr] = \int_{G_f}\!\! \tr\,\mu = \frac{1}{\#G_f}\sum_{g\in G_f}\tr(g) = \sum_{i=0}^d c_ii.
\]
The Chebotarev density theorem implies that this is also the average value of $N_f(p)$, that is,
\[
\lim_{B\to\infty}\frac{\sum_{p\le B} N_f(p)}{\sum_{p\le B}1} = \Exp[\tr].
\]
This average is $1$ in our example, because $f(x)$ is irreducible; see Exercise~\ref{ex:Nfpavg}.

The quantities $c_i$ define a probability distribution on the set $\{\tr(g):g\in G_f\}$ of traces that we can also view as a probability distribution on the set $\{N_f(p):p\text{ prime}\}$.
Picking a random prime $p$ in some large interval $[1,B]$ and computing $N_f(p)$ is the same thing as picking a random matrix $g$ in $G_f$ and computing $\tr(g)$.
More precisely, the sequence $(N_f(p))_p$ indexed by primes~$p$ is \emph{equidistributed} with respect to the pushforward of the Haar measure $\mu$ under the trace map.
We discuss the notion of equidistribution more generally in the Section \ref{lec:equidistribution}.

\subsection{Moment sequences}
There is another way to characterize the probability distribution on $\tr(g)$ given by the $c_i$; we can compute its \emph{moment sequence}:
\[
\Mom[\tr] \coloneqq (\Exp[\tr^n])_{n\ge 0},
\]
where
\[
\Exp[\tr^n] = \int_{G_f}\tr^n\mu.
\]
It might seem silly to include the zeroth moment $\Exp[\tr^0]=\Exp[1]=1$, but in Section~\ref{lec:STaxioms} we will see why this convention is useful.
In our example we have the moment sequence
\[
\Mom[\tr] = (1,1,2,5,14,41,\ldots,\tfrac{1}{2}(3^{n-1}+1),\ldots).
\]
The sequence $\Mom[\tr]$ uniquely determines\footnote{Not all moment sequences uniquely determine an underlying probability distribution, but all the moment sequence we shall consider do (because they satisfy \emph{Carleman's condition} \cite[p. 126]{Koo98}, for example).} the distributions of traces and thus captures all the information encoded in the $c_i$.
It may not seem very useful to replace a finite set of rational numbers with an infinite sequence of integers, but when dealing with continuous probability distributions, as we are forced to do as soon as we leave our weight zero setting, moment sequences are a powerful tool.

If we pick another cubic polynomial $f\in \Z[x]$, we will typically obtain the same result as we did in our example; when ordered by height almost all cubic polynomials $f$ have Galois group $G_f\simeq S_3$.
But there are exceptions: if $f$ is not irreducible over $\Q$ then $G_f$ will be isomorphic to a proper subgroup of~$S_3$, and this also occurs when the splitting field of $f$ is a cyclic cubic extension (this happens precisely when $\disc(f)$ is a square in $\Q^\times$; the polynomial $f(x)=x^3-3x-1$ is an example).
Up to conjugacy there are four subgroups of $S_3$, each corresponding to a different distribution of $N_f(p)$:
\medskip

\begin{center}
\begin{tabular}{lcccccl}
$f(x)$ & $G_f$ & $c_0$ & $c_1$ & $c_2$ & $c_3$ & $\Mom[\tr]$\\\midrule
$x^3-x$ & $1$ & 0 & 0 & 0 & 1 & $(1,3,9,27,81,\ldots)$\\
$x^3+x$ & $C_2$ & 0 & 1/2 & 0 & 1/2 & $(1,2,5,14,41,\ldots)$\\
$x^3-3x-1$ & $C_3$ & 2/3 & 0 & 0 & 1/3 & $(1,1,3,19,27,\ldots)$\\
$x^3-x+1$ & $S_3$ & 1/3 & 1/2 & 0 & 1/6 & $(1,1,2,\,\ 5,14,\ldots)$\\\bottomrule
\end{tabular}
\end{center}
\bigskip

One can do the same with polynomials of degree $d>3$.
For irreducible polynomials of degree $d\le 16$ the results are exhaustive: for every transitive subgroup $G$ of $S_d$ the \href{http://galoisdb.math.upb.de/home}{database} of Kl\"uners and Malle \cite{KM01} contains at least one irreducible monic polynomial $f\in \Z[x]$ of degree $d$ with $G_f\simeq G$ (as permutation groups).
It is an open question whether this can be done for all $d$ (even in principle); indeed, the case $d=17$ remains open at this time.  One can also ask this question for non-irreducible polynomials.  Of course the Galois group of any polynomial $f$ is isomorphic to the Galois group of some irreducible polynomial~$g$, but the degree of $g$ might need to be larger than that of~$f$.

\subsection{Zeta functions}
For polynomials $f$ of degree $d=3$ there is a one-to-one correspondence between subgroups of $S_d$ and distributions of $N_f(p)$.
This is not true for $d\ge 4$.
For example, the polynomials $f(x)=x^4-x^3+x^2-x+1$ with $G_f\simeq C_4$ and $g(x)=x^4-x^2+1$ with $G_g\simeq C_2\times C_2$ both have $c_0=3/4$, $c_1=c_2=c_3=0$, and $c_4=1/4$, corresponding to the moment sequence $\Mom[\tr]=(1,1,4,16,64,\ldots)$.

We can distinguish these cases if, in addition to considering the distribution of $N_f(p)$, we also consider the distribution of
\[
N_f(p^r)\coloneqq \#\{x\in \F_{p^r}:f_p(x)=0\}
\]
for integers $r\ge 1$.
In our quartic example we have $N_g(p^2)=4$ for almost all $p$, whereas $N_f(p^2)$ is $4$ or~$2$ depending on whether $p$ is a square modulo 5 or not.
In terms of the matrix group $G_f$ we have 
\begin{equation}\label{eq:Nfpr}
N_f(p^r) =\tr\bigl(\rho_f(\Frob_p)^r\bigr)
\end{equation}
for all primes $p$ that do not divide $\disc(f)$.
To see this, note that the permutation matrix $\rho_f(\Frob_p)^r$ corresponds to the permutation of the roots of $f_p(x)$ given by the $r$th power of the Frobenius automorphism $x\mapsto x^p$.
Its fixed points are precisely the roots of $f_p(x)$ that lie in $\F_{p^r}$; taking the trace counts these roots, and this yields $N_f(p^r)$.

This naturally leads to the definition of the local \emph{zeta function} of $f$ at $p$:
\begin{equation}\label{eq:zetaf}
Z_{f_p}(T)\coloneqq \exp \left(\sum_{r=1}^\infty N_f(p^r)\frac{T^r}{r}\right),
\end{equation}
which can be viewed as a generating function for the sequence $(N_f(p),N_f(p^2),N_f(p^3),\ldots)$.  This particular form of generating function may seem strange when first encountered, but it has some very nice properties.
For example, if $f,g\in \Z[x]$ are squarefree polynomials with no common factor, then their product $fg$ is also square free, and for all $p\nmid \disc(fg)$ we have
\[
Z_{(fg)_p} = Z_{f_pg_p}=Z_{f_p}Z_{g_p}.
\] 

\begin{remark}
The identity \eqref{eq:Nfpr} can be viewed as a special case of the Grothendieck-Lefschetz trace formula.
It allows us to express the zeta function $Z_{f_p}(T)$ as a sum over powers of the traces of the image of $\Frob_p$ under the Galois representation $\rho_f$.
In general one considers the trace of the Frobenius endomorphism acting on \'etale cohomology, but in dimension zero the only relevant cohomology is $H^0$.
\end{remark}

While defined as a power series, in fact $Z_{f_p}(T)$ is a rational function of the form
\[
Z_{f_p}(T) = \frac{1}{L_p(T)},
\]
where $L_p(T)$ is an integer polynomial whose roots lie on the unit circle.
This can be viewed as a consequence of the Weil conjectures in dimension zero,\footnote{Provided one accounts for the fact that $f(x)=0$ does not define an irreducible variety unless $\deg(f)=1$; in this case $N_f(p^r)=1$ and $L_p(T)=1-T$, which is consistent with the usual formulation of the Weil conjectures (see Theorem~\ref{thm:deligne}).} but in fact it follows directly from \eqref{eq:Nfpr}.
Indeed, for any matrix $A\in \GL_d(\C)$ we have the identity
\begin{equation}\label{eq:matid}
\exp\left (\sum_{r=1}^\infty \tr(A^r) \frac{T^r}{r}\right) = \det(1-AT)^{-1},
\end{equation}
which can be proved by expressing the coefficients on both sides as symmetric functions in the eigenvalues of $A$; see Exercise~\ref{ex:matid}.
Applying \eqref{eq:Nfpr} and \eqref{eq:matid} to the definition of $Z_{f_p}(T)$ in \eqref{eq:zetaf} yields
\[
Z_{f_p}(T) = \frac{1}{\det(1-\rho_f(\Frob_p)T)},
\]
thus
\[
L_p(T) = \det(1-\rho_f(\Frob_p)T).
\]

The polynomial $L_p(T)$ is precisely the polynomial that appears in the Euler factor at $p$ of the (partial) \emph{Artin $L$-function} $L(\rho_f,s)$ for the representation $\rho_f$:
\[
L(\rho_f,s)\coloneqq \prod_p L_p(p^{-s})^{-1},
\]
at least for primes $p$ that do not divide $\disc(f)$; for the definition of the Euler factors at ramified primes (and the Gamma factors at archimedean places), see \cite[Ch. 2]{MM11}.\footnote{The alert reader will note that primes dividing the discriminant of $f$ need not ramify in its splitting field; we are happy to ignore these primes as well, just as we may ignore primes of bad reduction for a curve that are good primes for its Jacobian.}
The Euler product for $L(\rho_f,s)$ defines a function that is holomorphic and nonvanishing on $\Re(s)>1$.  We shall not be concerned with the Euler factors at ramified primes, other than to note that they are holomorphic and nonvanishing on $\Re(s)>1$.

\begin{remark}
Every representation $\rho\colon \Gal(\Qbar/\Q)\to\GL_d(\C)$ with finite image gives rise to an Artin $L$-function $L(\rho,s)$, and Artin proved that every decomposition of $\rho$ into sub-representations gives rise to a corresponding factorization of $L(\rho,s)$ into Artin $L$-functions of lower degree.
The representation $\rho_f$  we have defined is determined by the permutation action of $\Gal(\Qbar/\Q)$ on the formal $\C$-vector space with basis elements corresponding to roots of $f$.
The linear subspace spanned by the sum of the basis vectors is fixed by $\Gal(\Qbar/\Q)$, so for $d>1$ we can always decompose $\rho_f$ as the sum of the trivial representation and a representation of dimension $d-1$, in which case $L(\rho_f,s)$ is the product of the Riemann zeta function (the Artin $L$-function of the trivial representation), and an Artin $L$-function of degree $d-1$.
The Artin $L$-functions $L(\rho_f,s)$ we have defined are thus imprimitive for $\deg f > 1$.
\end{remark}

Returning to our interest in equidistribution, the Haar measure $\mu$ on $G_f=\rho_f(\Gal(\Qbar/\Q))$ allows us to determine the distribution of $L$-polynomials $L_p(T)$ that we see as $p$ varies.
Each polynomial $L_p(T)$ is the \emph{reciprocal polynomial} (obtained by reversing the coefficients) of the characteristic polynomial of $\rho_f(\Frob_p)$.
If we fix a polynomial $P(T)$ of degree $d=\deg f$, and pick a prime $p$ at random from some large interval, the probability that $L_p(T)=P(T)$ is equal to the probability that the reciprocal polynomial $T^dP(1/T)$ is the characteristic polynomial of a random element of $G_f$ (this probability will be zero unless $P(T)$ has a particular form; see Exercise~\ref{ex:Lpshape}).

\begin{remark}
For $d\le 5$ the distribution of characteristic polynomials uniquely determines each subgroup of $S_d$ (up to conjugacy).
This is not true for $d\ge 6$, and for $d\ge 8$ one can find non-isomorphic subgroups of $S_d$ with the same distribution of characteristic polynomials; the transitive permutation groups \texttt{8T10} and \texttt{8T11} which arise for $x^8-13x^6+44x^4-17x^2+1$ and $x^8-x^5-2x^4+4x^2+x+1$ (respectively) are an example.
\end{remark}

\subsection{Computing zeta functions in dimension zero}
Let us now briefly address the practical question of efficiently computing the zeta function $Z_{f_p}(T)$, which amounts to computing the polynomial $L_p(T)$.
It suffices to compute the integers $N_f(p^r)$ for $r\le d$, which is equivalent to determining the degrees of the irreducible polynomials appearing in the factorization of $f_p(x)$ in $\Fp[x]$. These determine the cycle type, and therefore the conjugacy class, of the permutation of the roots of $f_p(x)$ induced by the action of the Frobenius automorphism $x\mapsto x^p$, which in turn determines the characteristic polynomial of $\rho_f(\Frob_p)$ and the $L$-polynomial $L_p(T)=\det(1-\rho_f(\Frob_p)T)$; see Exercise~\ref{ex:Lpshape}.
To determine the factorization pattern of $f_p(x)$, one can apply the following algorithm.

\begin{algorithm}\label{alg:facpat}
Given a squarefree polynomial $f\in \Fp[x]$ of degree $d>1$, compute the number $n_i$ of irreducible factors of $f$ in $\Fp[x]$ of degree $i$, for $1\le i\le d$ as follows:
\end{algorithm}

\begin{enumerate}[1.]
\setlength{\itemsep}{2pt}
\item Let $g_1(x)$ be $f(x)$ made monic and put $r_0(x):=x$.
\item For $i$ from $1$ to $d$:
\begin{enumerate}[a.]
\setlength{\itemsep}{2pt}
\item If $i > \deg(g_i)/2$ then for $i\le j\le d$ put $n_j:=1$ if $j=\deg(g_i)$ and $n_j:=0$ otherwise, and then proceed to step~3.
\item Using binary exponentiation in the ring $\Fp[x]/(g_i)$, compute $r_i \coloneqq r_{i-1}^p\bmod g_i$.
\item Compute $h_i(x)\coloneqq\gcd(g_i,r_i(x)-x)=\gcd(g_i(x),x^{p^i}-x)$ using the Euclidean algorithm.
\item Compute $n_i \coloneqq \deg(h_i)/i$ and $g_{i+1} \coloneqq g_i/h_i$ using exact division.
\item If $\deg(g_{i+1}) = 0$ then put $n_j\coloneqq 0$ for $i < j\le d$ and proceed to step 3.
\end{enumerate}
\item Output $n_1,\ldots,n_d$.
\end{enumerate}

\noindent
Algorithm~\ref{alg:facpat} makes repeated use of the fact that the polynomial
\[
x^{p^i}-x=\prod_{{}\ a\,\in\, \F_{p^i}} (x-a)
\]
is equal to the product of all irreducible monic polynomials of degree dividing $i$ in $\Fp[x]$.
By starting with $i=1$ and removing all factors of degree $i$ as we go, we ensure that each $h_i$ is a product of irreducible polynomials of degree~$i$.
Using fast algorithms for integer and polynomial arithmetic and the fast Euclidean algorithm (see \cite[\S8-11]{GG13}, for example), one can show that this algorithm uses $O((d\log p)^{2+o(1)})$ bit operations, a running time that is quasi-quadratic in the $O(d\log p)$ bit-size of its input $f\in \Fp[x]$.\footnote{One can improve this to $O\bigl(d^{1.5+o(1)}(\log p)^{1+o(1)}+d^{1+o(1)}(\log p)^{2+o(1)}\bigr)$ via \cite{KU11}.  In our setting $d$ is fixed and $\log p$ is tending to infinity, so this is not an asymptotic improvement, but it does provide a constant factor improvement for large $d$.}
In practical terms, it is extremely efficient.
For example, the table of $c_i(B)$ values for our example polynomial $f(x)=x^3-x+1$ with $B=10^{12}$ took less than two minutes to create using the \texttt{smalljac} software library \cite{KS08,smalljac}, which includes an efficient implementation of basic finite field arithmetic.
The NTL \cite{NTL} and FLINT \cite{Har10,FLINT} libraries also incorporate variants of this algorithm, as do the computer algebra systems Sage \cite{Sage} and Magma \cite{Magma}.

\begin{remark}
Note that Algorithm~\ref{alg:facpat} does \emph{not} output the factorization of $f(x)$, just the degrees of its irreducible factors.
It can be extended to a probabilistic algorithm that outputs the complete factorization of $f(x)$ (see \cite[Alg. 14.8]{GG13}, for example), with an expected running time that is also quasi-quadratic.
But no deterministic polynomial-time algorithm for factoring polynomials over finite fields is known, not even for $d=2$.
This is a famous open problem.
One approach to solving it is to first prove the generalized Riemann hypothesis (GRH), which would address the case $d=2$ and many others, but it is not even known whether the GRH is sufficient to address all cases.\footnote{If you succeed with even a special case of this first step, the Clay institute will help \href{http://www.claymath.org/millennium-problems/rules-millennium-prizes}{fund} the remaining work.}
\end{remark}

\subsection{Arithmetic schemes}

We now want to generalize our first example.
Let us replace the equation $f(x)=0$ with an \emph{arithmetic scheme}  $X$, a scheme of finite type over $\Z$; the case we have been considering is $X=\Spec A$, where $A=\Z[x]/(f)$.
For each prime $p$ the fiber $X_p$ of $X\to \Spec \Z$ is a scheme of finite type over $\Fp$, and we let $N_X(p)\coloneqq X_p(\Fp)$ count its $\Fp$-points; equivalently, we may define $N_X(p)$ as the number of closed points (maximal ideals) of $X$ whose residue field has cardinality $p$, and similarly define $N_X(q)$ for prime powers $q=p^r$.
The local zeta function of $X$ at $p$ is then defined as
\[
Z_{X_p}(T) \coloneqq \exp\left(\sum_{r=1}^\infty N_X(p^r)\frac{T^r}{r}\right).
\]
These local zeta functions can then be packaged into a single \emph{arithmetic zeta-function}
\[
\zeta_X(s)\coloneqq \prod_p Z_{X_p}(p^{-s}).
\]
In our example with $X=\Spec \Z[x]/(f)$, the zeta function $\zeta_X(s)$ coincides with the Artin $L$-function $L(\rho_f,s)=\prod L_p(s)^{-1}$ up to a finite set of factors at primes $p$ that divide $\disc(f)$.

The definitions above generalize to any number field $K$: replace $\Q$ by~$K$, replace $\Z$ by $\Z_K$, replace $p$ by a prime $\p$ of $K$ (nonzero prime ideal of $\Z_K$), replace $\Fp\simeq \Z/p\Z$ by the residue field $\F_\p:=\Z_K/\p$.
When considering questions of equidistribution we order primes $\p$ by their norm $N(\p)\coloneqq \F_\p$ (we may break ties arbitrarily), so that rather that summing over $p\le B$ we sum over $\p$ for which $N(\p)\le B$.

\subsection{A second example}\label{sec:secondexample}

We now leave the world of Artin motives, which are motives of weight 0, and consider the simplest example in weight 1, an elliptic curve $E/\Q$.
This is the setting in which the Sato--Tate conjecture was originally formulated.
Every elliptic curve $E/\Q$ can be written in the form
\[
E\colon y^2=x^3+Ax+B,
\]
with $A,B\in \Z$.
This equation is understood to define a smooth projective curve in $\P^2$ (homogenize the equation by introducing a third variable $z$), which has a single projective point $P_\infty\coloneqq (0:1:0)$ at infinity that we take as the identity element of the group law on $E$.
Recall that an elliptic curve is not just a curve, it is an abelian variety, and comes equipped with a distinguished rational point corresponding to the identity; by applying a suitable automorphism of~$\P^2$ we can always take this to be the point $P_\infty$.

The group operation on $E$ can be defined via the usual chord-and-tangent law (three points on a line sum to zero), which can be used to derive explicit formulas with coefficients in $\Q$, or in terms of the divisor class group $\Pic^0(E)$ (divisors of degree zero modulo principal divisors), in which every divisor class can be uniquely represented by a divisor of the form $P-P_\infty$, where $P$ is a point on the curve.
This latter view is more useful in that it easily generalizes to curves of genus $g>1$, whereas the chord-and-tangent law does not.
The Abel--Jacobi map $P\mapsto P-P_\infty$ gives a bijection between points on~$E$ and points on $\Jac(E)$ that commutes with the group operation, so the two approaches are equivalent.

For each prime $p$ that does not divide the discriminant $\Delta \coloneqq -16(4A^3+27B^2)$ we can reduce our equation for $E$ modulo~$p$ to obtain an elliptic curve $E_p/\Fp$; in this case we say that $p$ is a \emph{prime of good reduction for~$E$} (or simply a \emph{good prime}).
We should note that the discriminant $\Delta$ is not necessarily minimal; the curve~$E$ may have another model with good reduction at primes that divide $\Delta$ (possibly including $2$), but we are happy to ignore any finite set of primes, including those that divide $\Delta$.\footnote{All elliptic curves over $\Q$ have a global minimal model for which the primes of bad reduction are precisely those that divide the discriminant, but this model is not necessarily of the form $y^2=x^3+Ax+B$.
Over general number fields $K$ global minimal models do not always exist (they do when $K$ has class number one).}

For every prime $p$ of good reduction for $E$ we have
\[
N_E(p) \coloneqq \#E_p(\Fp) = p+1-t_p,
\]
where the integer $t_p$ satisfies the \emph{Hasse-bound} $|t_p|\le 2\sqrt{p}$.
In contrast to our first example, the integers $N_E(p)$ now tend to infinity with $p$: we have $N_E(p) = p+1 +O(\sqrt{p})$.  In order to study how the error term varies with $p$ we want to consider the normalized traces
\[
x_p\coloneqq t_p/\sqrt{p}\in [-2,2].
\]
We are now in a position to conduct the following experiment: given an elliptic curve $E/\Fp$, compute~$x_p$ for all good primes $p\le B$ and see how the $x_p$ are distributed over the real interval $[-2,2]$.

One can see an example for the elliptic curve $E:y^2=x^3+x+1$ in Figure~\ref{fig:g1generic}, which shows a histogram whose $x$-axis spans the interval $[-2,2]$.
This interval is subdivided into approximately $\sqrt{\pi(B)}$ subintervals, each of which contains a bar representing the number of $x_p$ (for $p\le B$) that lie in the subinterval.
The gray line shows the height of the uniform distribution for scale (note that the vertical and horizontal scales are not the same).
For $0\le n\le 10$, the moment statistics
\[
M_n\coloneqq\frac{\sum_{p\le B} x_p^n}{\sum_{p\le B} 1},
\]
are shown below the histogram. They appear to converge to the integers $1,0,1,0,2,0,5,0,14,0,42$, which is the start of sequence \Aseq{A126120} in the Online Encyclopedia of Integer Sequences (OEIS) \cite{OEIS}).

\begin{figure}[h!]
\begin{center}
\animategraphics[width=\textwidth,poster=last]{1.5}{g1_generic_a1_}{0}{30}
\caption{Click image to animate (requires Adobe Reader), or visit this \href{http://math.mit.edu/~drew/g1_D1_a1f.gif}{web page}.}\label{fig:g1generic}
\end{center}
\end{figure}

The Sato--Tate conjecture for elliptic curves over $\Q$ (now a theorem) implies that for almost all $E/\Q$, whenever we run this experiment we will see the asymptotic distribution of Frobenius traces visible in Figure \ref{fig:g1generic}, with moment statistics that converge to the same integer sequence.
In order to make this conjecture precise, let us first explain where the conjectured distribution comes from.
In our first example we had a compact matrix group $G_f$ associated to the scheme $X=\Spec\Z[x]/(f)$ whose Haar measure governed the distribution of $N_f(p)$.
In fact we showed that more is true: there is a direct relationship between characteristic polynomials of elements of $G_f$ and the $L$-polynomials $L_p(T)$ that appear in the local zeta functions $Z_{f_p}(T)$.

The same is true with our elliptic curve example.  In order to identify a candidate group $G_E$ whose Haar measure controls the distribution of normalized Frobenius traces $x_p$ we need to look at the local zeta functions $Z_{E_p}(T)$.
Let us recall what the Weil conjectures \cite{Weil49} (proved by Deligne \cite{Del74,Del80}) tell us about the zeta function of a variety over a finite field.
The case of one-dimensional varieties (curves) was proved by Weil \cite{Weil45}, who also proved an analogous result for abelian varieties \cite{Weil46}.  This covers all the cases we shall consider, but let us state the general result.
Recall that for a compact manifold $X$ over~$\C$, the \emph{Betti number} $b_i$ is the rank of the singular homology group $H_i(X,\Z)$, and the \emph{Euler characteristic} $\chi$ of $X$ is defined by $\chi\coloneqq\sum (-1)^ib_i$.

\begin{theorem}[\textsc{Weil Conjectures}]\label{thm:deligne}
Let $X$ be a geometrically irreducible non-singular projective variety of dimension $n$ defined over a finite field $\Fq$ and define the zeta function
\[
Z_X(T)\coloneqq \exp\left(\sum_{r=1}^\infty N_X(q^r)\frac{T^r}{r}\right),
\]
where $N_X(q^r)\coloneqq \#X(\F_{q^r})$.
The following hold:
\begin{enumerate}[{\rm (i)}]
\setlength{\itemsep}{2pt}
\item \textbf{Rationality}: $Z_X(T)$ is a rational function of the form
\[
Z_X(T) = \frac{P_1(T)\cdots P_{2n-1}(T)}{P_0(T)\cdots P_{2n}(T)},
\]
with $P_i\in 1+ T\Z[T]$.
\item \textbf{Functional Equation}: the roots of $P_i(T)$ are the same as the roots of $T^{\deg P_{2n-i}}P_{2n-i}(1/(q^nT))$.\footnote{Moreover, one has $Z_X(T)=\pm q^{-n\chi/2}T^{-\chi}Z_X(1/(q^nT))$, where $\chi$ is the Euler characteristic of $X$, which is defined as the intersection number of the diagonal with itself in $X\times X$.}
\item \textbf{Riemann Hypothesis}: the complex roots of $P_i(T)$ all have absolute value $q^{-i/2}$.
\item \textbf{Betti Numbers}: if $X$ is the reduction of a non-singular variety $Y$ defined over a number field $K\subseteq \C$, then the degree of $P_i$ is equal to the Betti number $b_i$ of $Y(\C)$.
\end{enumerate}
\end{theorem}

The curve $E_p$ is a curve of genus $g=1$, so we may apply the Weil conjectures in dimension $n=1$, with Betti numbers $b_0=b_2=1$ and $b_1=2g=2$.
This implies that its zeta function can be written as
\begin{equation}\label{eq:zetaE}
Z_{E_p}(T) =\frac{L_p(T)}{(1-T)(1-pT)},
\end{equation}
where $L_p\in \Z[T]$ is a polynomial of the form
\[
L_p(T) = pT^2+c_1T+1,
\]
with $|c_1|\le 2\sqrt{p}$ (by the Riemann Hypothesis).
If we expand both sides of \eqref{eq:zetaE} as power series in $\Z[[T]]$ we obtain
\[
1+N_E(p)T+\cdots = 1+(p+1+c_1)T+\cdots,
\]
so we must have $N_E(p)=p+1+c_1$, and therefore
\[
c_1=N_E(p)-p-1=-t_p.
\]
It follows that the single integer $N_E(p)$ completely determines the zeta function $Z_{E_p}(T)$.

Corresponding to our normalization $x_p=t_p/\sqrt{p}$, we define the \emph{normalized $L$-polynomial}
\[
\bar L_p(T) \coloneqq L_p(T/\sqrt{p}) = T^2+a_1T+1,
\]
where $a_1=c_1/\sqrt{p}=-x_p$ is a real number in the interval $[-2,2]$ and the roots of $\bar L_p(T)$ lie on the unit circle.
In our first example we obtained the group $G_f$ as a subgroup of permutation matrices in $\GL_d(\C)$.
Here we want a subgroup of $\GL_2(\C)$ whose elements have eigenvalues that
\begin{enumerate}[(a)]
\item are inverses (by the functional equation);
\item lie on the unit circle (by the Riemann hypothesis).
\end{enumerate}
Constraint (a) makes it clear that every element of $G_E$ should have determinant $1$, so $G_E\subseteq \SL_2(\C)$.
Constraints (a) and (b) together imply that in fact $G_E\subseteq \SU(2)$.
As in the weight zero case, we expect that $G_E$ should in general be as large as possible, that is, $G_E=\SU(2)$.

We now consider what it means for an elliptic curve to be generic.\footnote{The criterion given here in terms of endomorphism rings suffices for elliptic curves (and curves of genus $g\le 3$ or abelian varieties of dimension $g\le 3$), but in general one wants the Galois image to be as large as possible, which is a strictly stronger condition for $g>3$. This issue is discussed further in Section \ref{lec:STgroups}.}
Recall that the endomorphism ring of an elliptic curve $E$ necessarily contains a subring isomorphic to $\Z$, corresponding to the multiplication-by-$n$ maps $P\mapsto nP$.  Here
\[
nP=P+\cdots+P
\]
denotes repeated addition under the group law, and we take the additive inverse if $n$ is negative.
For elliptic curves over fields of characteristic zero, this typically accounts for all the endomorphisms, but in special cases the endomorphism ring may be larger, in which case it contains elements that are not multiplication-by-$n$ maps but can be viewed as ``multiplication-by-$\alpha$" maps for some $\alpha\in \C$.
One can show that the minimal polynomials of these extra endomorphisms are necessarily quadratic, with negative discriminants, so such an $\alpha$ necessarily lies in an imaginary quadratic field $K$, and in fact $\End(E)\otimes_\Z\Q\simeq K$.
When this happens we say that $E$ has \emph{complex multiplication} (CM) by $K$ (or more precisely, by the order in $\Z_K$ isomorphic to $\End(E)$).

We can now state the Sato-Tate conjecture, as independently formulated in the mid 1960's by Mikio Sato (based on numerical data) and John Tate (as an application of what is now known as the Tate conjecture \cite{Tate63}), and finally proved in the late 2000's by Richard Taylor et al. \cite{BLGG11,BLGHT11,HSBT10}.

\begin{theorem}[\textsc{Sato--Tate conjecture}]\label{thm:ST}
Let $E/\Q$ be an elliptic curve without $CM$.
The sequence of normalized Frobenius traces $x_p$ associated to $E$ is equidistributed with respect to the pushforward of the Haar measure on $\SU(2)$ under the trace map.
In particular, for every subinterval $[a,b]$ of $[-2,2]$ we have
\[
\lim_{B\to \infty} \frac{\#\{p\le B:x_p\in [a,b]\}}{\#\{p\le B\}} = \frac{1}{2\pi}\int_a^b \sqrt{4-t^2}\,dt.
\]
\end{theorem}

We have not defined $x_p$ for primes of bad reduction, but there is no need to do so; this theorem is purely an asymptotic statement.
To see where the expression in the integral comes from, we need to understand the Haar measure on $\SU(2)$ and its pushforward onto the set of conjugacy classes $\conj(\SU(2))$ (in fact we only care about the latter).
A conjugacy class in $\SU(2)$ can be described by an \emph{eigenangle} $\theta\in [0,\pi]$; its eigenvalues are then $e^{\pm i\theta}$ (a conjugate pair on the unit circle, as required).
In terms of eigenangles, the pushforward of the Haar measure to $\conj(\SU(2))$ is given by
\[
\mu = \frac{2}{\pi}\sin^2\theta\, d\theta
\]
(see Exercise~\ref{ex:su2}), and the trace is $t=2\cos\theta$; from this one can deduce the trace measure $\frac{1}{2\pi}\sqrt{4-t^2}dt$ on $[-2,2]$ that appears in Theorem~\ref{thm:ST}.
We can also use the Haar measure to compute the $n$th moment of the trace
\begin{equation}\label{eq:cat}
\Exp[t^n] = \frac{2}{\pi}\int_0^\pi (2\cos\theta)^n\sin^2\theta d\theta = \begin{cases}
0 &\text{if $n$ is odd,}\\
\frac{1}{m+1}\binom{2m}{m} & \text{if $n=2m$ is even},
\end{cases}
\end{equation}
and find that the $2m$th moment is the $m$th Catalan number.\footnote{This gives yet another way to define the Catalan numbers, one that does not appear to be among the 214 listed in \cite{Stan15}.}

\subsection{Exercises}

\begin{exercise}\label{ex:Nfpavg}
Let $f\in \Z[x]$ be a nonconstant squarefree polynomial.  Prove that the average value of $N_f(p)$ over $p\le B$ converges to the number of irreducible factors of $f$ in $\Z[x]$ as $B\to\infty$.
\end{exercise}
\begin{exercise}\label{ex:matid}
Prove that the identity in \eqref{eq:matid} holds for all matrices $A\in \GL_d(\C)$.
\end{exercise}
\begin{exercise}\label{ex:Lpshape}
Let $f_p\in \Fp[x]$ denote a squarefree polynomial of degree $d>0$ and let $L_p(T)$ denote the denominator of the zeta function $Z_{f_p}(T)$.
We know that the roots of $L_p(T)$ lie on the unit circle in the complex plane; show that in fact each is an $n$th root of unity for some $n\le d$.
Then give a one-to-one correspondence between (i) cycle-types of degree-$d$ permutations, (ii) possible factorization patterns of $f_p$ in $\Fp[x]$, and (iii) the possible polynomials $L_p(T)$.
\end{exercise}
\begin{exercise}
Construct a monic squarefree quintic polynomial $f\in\Z[x]$ with no roots in $\Q$ such that $f_p(x)$ has a root in $\Fp$ for every prime $p$.
Compute $c_0,\ldots,c_5$ and $G_f$.
\end{exercise}
\begin{exercise}
Let $X$ be the arithmetic scheme $\Spec\Z[x,y]/(f,g)$, where
\[
f(x,y)\coloneqq y^2-2x^3+2x^2-2x-2,\qquad g(x,y)\coloneqq 4x^2-2xy+y^2-2.
\]
By computing $Z_{X_p}(T)=L_p(T)^{-1}$ for sufficiently many small primes~$p$, construct a list of the polynomials $L_p\in\Z[T]$ that you believe occur infinitely often, and estimate their relative frequencies.
Use this data to derive a candidate for the matrix group $G_X\coloneqq \rho_X(\Gal(\Qbar/\Q)$, where $\rho_X$ is the Galois representation defined by the action of $\Gal(\Qbar/\Q)$ on $X(\Qbar)$.  You may wish to use of computer algebra system such as \href{https://cloud.sagemath.com/}{Sage} \cite{Sage} or \href{http://magma.maths.usyd.edu.au/calc/}{Magma} or \cite{Magma} to facilitate these calculations.
\end{exercise}

\section{Equidistribution, L-functions, and the Sato-Tate conjecture for elliptic curves}\label{lec:equidistribution}

In this section we introduce the notion of equidistribution in compact groups $G$ and relate it to analytic properties of $L$-functions of representations of $G$.
We then explain (following Tate) why the Sato-Tate conjecture for elliptic curves follows from the holomorphicity and non-vanishing of a certain sequence of $L$-functions that one can associate to an elliptic curve over $\Q$ (or any number field).

\subsection{Equidistribution}\label{sec:measure}
We now formally define the notion of equidistribution, following \cite[\S 1A]{Se68}.
For a compact Hausdorff space $X$, we use $C(X)$ to denote the Banach space of complex-valued continuous functions $f\colon X\to \C$ equipped with the sup-norm $\Vert f\Vert\coloneqq \sup_{x\in X}|f(x)|$.
The space $C(X)$ is closed under pointwise addition and multiplication and contains all constant functions; it is thus a commutative $\C$-algebra with unit $\ifunc_X$ (the function $x\mapsto 1$).\footnote{In fact, it is a commutative $C^*$-algebra with complex conjugation as its involution, but we will not make use of this.}
For any $\C$-valued functions $f$ and $g$ (continuous or not), we write $f\le g$ whenever $f$ and $g$ are both $\R$-valued and $f(x)\le g(x)$ for all $x\in X$; in particular, $f\ge 0$ means $\im(f)\subseteq \R_{\ge 0}$.
The subset of $\R$-valued functions in $C(X)$ form a distributive lattice under this order relation.

\begin{definition}\label{def:measure}
A (positive normalized Radon) \emph{measure} on a compact Hausdorff space $X$ is a continuous $\C$-linear map $\mu\colon C(X)\to \C$ that satisfies $\mu(f)\ge 0$ for all $f\ge 0$ and $\mu(\mathbbm 1_X)=1$.
\end{definition}

\begin{example}
For each point $x\in X$ the map $f\mapsto f(x)$ defines the \emph{Dirac measure} $\delta_x$.
\end{example}

The value of $\mu$ on $f\in C(X)$ is often denoted using integral notation
\[
\int_X\! f\mu \coloneqq \mu(f),
\]
and we shall use the two interchangeably.\footnote{Note that this is a definition; with a measure-theoretic approach one avoids the need to develop an integration theory.}

Having defined the measure $\mu$ as a function on $C(X)$, we would like to use it to assign values to (at least some) subsets of $X$.
It is tempting to define the measure of a set $S\subseteq X$ as the measure of its indicator function $\ifunc_S$, but in general the function $\ifunc_S$ will not lie in $C(X)$; this occurs if and only if $S$ is both open and closed (which we note applies to $S=X$).  Instead, for each open set $S\subseteq X$ we define
\[
\mu(S) = \sup\, \bigl\{\mu(f):0\le f \le \ifunc_S,\ f\in C(X) \bigr\}\in [0,1],
\]
and for each closed set $S\subseteq X$ we define
\[
\mu(S)=1-\mu(X-S) \in [0,1].
\]
If $S\subseteq X$ has the property that for every $\epsilon >0$ there exists an open set $U\supseteq S$ of measure $\mu(U)\le \epsilon$, then we define $\mu(S)=0$ and say that $S$ has \emph{measure zero}.
If the boundary $\partial S\coloneqq \overline S-{S^0}$ of a set $S$ has measure zero, then we necessarily have $\mu(S^0)=\mu(\overline S)$ and define $\mu(S)$ to be this common value; such sets are said to be $\mu$-\emph{quarrable}.

For the purpose of studying equidistribution, we shall restrict our attention to $\mu$-quarrable sets $S$.
This typically does not include all measurable sets in the usual sense, by which we mean elements of the Borel $\sigma$-algebra $\Sigma$ of $X$ generated by the open sets under complements and countable unions and intersections (see Exercise \ref{ex:measure}).
However, if we are given a regular Borel measure $\mu$ on $X$ of total mass $1$, by which we mean a countably additive function $\mu\colon\Sigma\to \R_{\ge 0}$ for which $\mu(S)=\inf\,\{\,\mu(U): S\subseteq U,\ U\text{ open}\}$ and $\mu(X)=1$, it is easy to check that defining $\mu(f)\coloneqq \int_X f\mu$ for each $f\in C(X)$ yields a measure under Definition~\ref{def:measure}; see \cite[\S 1]{He77} for details.
This measure is determined by the values $\mu$ takes on $\mu$-quarrable sets \cite{Wulf61}.
In particular, if $X$ is a compact group then its Haar measure induces a measure on $X$ in the sense of Definition~\ref{def:measure}.

\begin{definition}
A sequence $(x_1,x_2,x_3,\ldots)$ in $X$ is said to be \emph{equidistributed with respect to $\mu$}, or simply \emph{$\mu$-equidistributed}, if for every $f\in C(X)$ we have
\[
\mu(f)=\lim_{n\to\infty}\frac{1}{n}\sum_{i=1}^nf(x_i).
\]
\end{definition}

\begin{remark}\label{rem:order}
When we speak of equidistribution, note that we are talking about a \emph{sequence} $(x_i)$ of elements of $X$ in a particular order; it does not make sense to say that a \emph{set} is equidistributed.
For example, suppose we took the set of odd primes and arranged them in the sequence $(5,13,3,17,29,7,\ldots)$ where we list two primes congruent to 1 modulo 4 followed by one prime congruent to 3 modulo 4.
The sequence obtained by reducing this sequence modulo $4$ is not equidistributed with respect to the uniform measure on $(\Z/4\Z)^\times$, even though the sequence of odd primes in their usual order is (by Dirichlet's theorem on primes in arithmetic progressions).
However, local rearrangements that change the index of an element by no more than a bounded amount do not change its equidistribution properties.
This applies, in particular, to sequences indexed by primes of a number field ordered by norm; the equidistribution properties of such a sequence do not depend on how we order primes of the same norm.
\end{remark}

If $(x_i)$ is a sequence in $X$, for each real-valued function $f\in C(X)$ we define the $k$th-\emph{moment} of the sequence $(f(x_i))$ by
\[
\Mom_k[(f(x_i)]\coloneqq  \lim_{n\to\infty}\frac{1}{n} \sum_{i=1}^n f(x_i)^k.
\]
If these limits exist for all $k\ge 0$, we then define the \emph{moment sequence}
\[
\Mom[f(x_i)] \coloneqq (\Mom_0[(f(x_i)], \Mom_1[(f(x_i)], \Mom_2[(f(x_i)],\ldots).
\]
If $(x_i)$ is $\mu$-equidistributed, then $\Mom_k[f(x_i)]=\mu(f^k)$ and the moment sequence
\begin{equation}\label{eq:mumom}
\Mom[f(x_i)] = (\mu(f^0),\mu(f^1),\mu(f^2),\ldots)
\end{equation}
is independent of the sequence $(x_i)$; it depends only on the function $f$ and the measure $\mu$.

\begin{remark}
There is a partial converse that is relevant to some of our applications.
To simplify matters, let us momentarily restrict our attention to real-valued functions; for the purposes of this remark, let $C(X)$ denote the Banach algebra of real-valued functions on $X$ and replace $\C$ with $\R$ in Definition~\ref{def:measure}.
Let $(x_i)$ be a sequence in $X$ and let $f\in C(X)$.
Then $f(X)$ is a compact subset of $\R$, and we may view $(f(x_i))$ as a sequence in $f(X)$.
If the moments $\Mom_k[f(x_i)]$ exist for all $k\ge 0$, then there is a \emph{unique} measure on $f(X)$ with respect to which the sequence $(f(x_i))$ is equidistributed; this follows from the Stone-Weierstrass theorem.
If $\mu$ is a measure on $C(X)$, we define the pushforward measure $\mu_f(g)\coloneqq \mu(g\circ f)$ on $C(f(X)$ and see that the sequence $(f(x_i))$ is $\mu_f$-equidistributed if and only if \eqref{eq:mumom} holds.
This gives a necessary (but in general not sufficient) condition for $(x_i)$ to be $\mu$-equidistributed that can be checked by comparing moment sequences.
If we have a collection of functions $f_j\in C(X)$ such that the pushforward measures $\mu_{f_j}$ uniquely determine $\mu$, we obtain a necessary and sufficient condition involving the moment sequences of the $f_j$ with respect to $\mu$.
One can generalize this remark to complex-valued functions using the theory of $C^*$-algebras.
\end{remark}

More generally, we have the following lemma.

\begin{lemma}\label{lem:se1}
Let $(f_j)$ be a family of functions whose linear combinations are dense in $C(X)$.
If $(x_i)$ is a sequence in $X$ for which the limit $\lim_{n\to\infty}\frac{1}{n}\sum_{i=1}^n f_j(x_i)$ converges for every $f_j$, then there is a unique measure $\mu$ on $X$ for which $(x_i)$ is $\mu$-equidistributed.
\end{lemma}
\begin{proof}
See \cite[Lemma A.1, p. I-19]{Se68}.
\end{proof}

\begin{proposition}\label{prop:mueq}
If $(x_i)$ is a  $\mu$-equidistributed sequence in $X$ and $S$ is a $\mu$-quarrable set in $X$ then
\[
\mu(S) = \lim_{n\to\infty}\frac{\#\{x_i\in S:i \le n\}}{n}.
\]
\end{proposition}
\begin{proof}
See Exercise \ref{ex:mueq}.
\end{proof}

\begin{example}
If $X=[0,1]$ and $\mu$ is the Lebesgue measure then a sequence $(x_i)$ is $\mu$-equidistributed if and only if for every $0\le a < b\le 1$ we have
\[
\lim_{n\to\infty}\frac{\#\{x_i\in [a,b]: i\le n\}}{n} = b-a.
\]
More generally, if $X$ is a compact subset of $\R^n$ and $\mu$ is the normalized Lebesgue measure, then $(x_i)$ is $\mu$-equidistributed if and only if for every $\mu$-quarrable $S\subseteq X$ we have $\lim_{n\to\infty}\frac{1}{n}\#\{x_i\in S:i\le n\}=\mu(S)$.
\end{example}

\subsection{Equidistribution in compact groups}

We now specialize to the case where $X\coloneqq \conj(G)$ is the space of conjugacy classes of a compact group $G$, obtained by taking the quotient of~$G$ as a topological space under the equivalence relation defined by conjugacy; let $\pi\colon G\to X$ denote the quotient map.
We then equip $X$ with the pushforward of the Haar measure $\mu$ on $G$ (normalized so that $\mu(G)=1$), which we also denote~$\mu$.
Explicitly, $\pi$ induces a map of Banach spaces
\begin{align*}
C(X)&\to C(G)\\
f&\mapsto f\circ \pi,
\end{align*}
and the value of $\mu$ on $C(X)$ is defined by
\[
\mu(f) \coloneqq \mu(f\circ\pi).
\]
We say that a sequence $(x_i)$ in $X$ or a sequence $(g_i)$ in $G$ is \emph{equidistributed} if it is $\mu$-equidistributed (when we speak of equidistribution in a compact group without explicitly mentioning a measure, we always mean the Haar measure).

\begin{proposition}
Let $G$ be a compact group with Haar measure $\mu$, and let $X\coloneqq\conj(G)$.
A sequence $(x_i)$ in $X$ is $\mu$-equidistributed if and only if for every irreducible character~$\chi$ of $G$ we have
\[
\lim_{n\to\infty}\frac{1}{n}\sum_{i=1}^n\chi(x_i)=\mu(\chi).
\]
\end{proposition}
\begin{proof}
As explained in \cite[Prop.\ A.2]{Se68}, this follows from Lemma~\ref{lem:se1} and the Peter-Weyl theorem, since the irreducible characters~$\chi$ of $G$ generate a dense subset of $C(X)$.
\end{proof}

\begin{corollary}\label{cor:char}
Let $G$ be a compact group with Haar measure $\mu$, and let $X\coloneqq\conj(G)$.
A sequence $(x_i)$ in $X$ is $\mu$-equidistributed if and only if for every nontrivial irreducible character~$\chi$ of $G$ we have
\[
\lim_{n\to\infty}\frac{1}{n}\sum_{i=1}^n\chi(x_i)=0.
\]
\end{corollary}
\begin{proof}
For the trivial character we have $\mu(1)=\mu(G)=1$, and for any nontrivial irreducible character~$\chi$ we must have $\mu(\chi)=\int_G\chi\mu = \int_G1\cdot\chi\mu = 0$, by Schur orthogonality; the corollary follows.
\end{proof}

To illustrate these results, we now use Corollary~\ref{cor:char} to prove an equidistribution result for elliptic curves over finite fields that will be useful later.
We first recall some basic facts.  Let $E$ be an elliptic curve over a finite field $\Fq$; without loss of generality, assume $E/\Fq$ is given by a projective plane model.
The \emph{Frobenius endomorphism} $\pi_E:E\to E$ is defined by the rational map
\[
(x:y:z)\mapsto (x^q:y^q:z^q).
\]
Like all endomorphisms of elliptic curves, $\pi_E$ has a characteristic polynomial of the form
\[
T^2-(\tr\pi_E)T+\deg \pi_E
\]
that is satisfied by both $\pi_E$ and its dual $\hat\pi_E$, where $\tr\pi_E=\pi_E+\hat\pi_E$ and $q=\deg\pi_E=\pi_E\hat \pi_E$ are both integers.\footnote{By the \emph{dual} of an endomorphism of a polarized abelian variety we mean the Rosati dual (see \cite[\S 13]{Milne:AV}), which for elliptic curves we may identify with the dual isogeny.}
The set $E(\Fq)$ is, by definition, the subset of $E(\Fqbar)$ fixed by $\pi_E$, equivalently, the kernel of the endomorphism $\pi_E-1$.  One can show that $\pi_E-1$ is a separable, and therefore
\[
\#E(\Fq) = \#\ker(\pi_E-1)= \deg(\pi_E-1) = (\pi_E-1)(\hat \pi_E-1)=\hat\pi_E\pi_E+1-(\hat\pi_E+\pi_E) =q+1-\tr\pi_E.
\]
It follows that $t_q\coloneqq q+1- \#E(\Fq)$ is the \emph{trace of Frobenius} $\tr\pi_E$.
As we showed in Section \ref{sec:secondexample} for the case $q=p$, the zeta function of $E$ can be written as
\[
Z_E(T)=\frac{qT^2-t_qT+1}{(1-T)(1-qT)},
\]
where the complex roots of $qT^2-t_qT+1$ have absolute value $q^{-1/2}$.
This implies that we can write $t_q=\alpha+\bar\alpha$ for some $\alpha\in \C$ with $|\alpha|=q^{1/2}$, and we have
$\#E(\Fq)=q+1-(\alpha+\bar\alpha)$.

We now observe that for any integer $r\ge 1$, the set $E(\F_{q^r})$ is the subset of $E(\Fqbar)$ fixed by $\pi_E^r$, which corresponds to the $q^r$-power Frobenius automorphism; it follows that
\[
\#E(\F_{q^r}) = q^r+1-(\alpha^r+\bar\alpha^r),
\]
and therefore $\alpha^r+\bar\alpha^r$ is the trace $t_{q^r}$ of the Frobenius endomorphism of the base change of $E$ to $\F_{q^r}$.

As an application of Corollary~\ref{cor:char}, we now prove the following result, taken from \cite[Prop 2.2]{Fite15}.
Recall that $E/\Fq$ is said to be \emph{ordinary} if $t_q$ is not zero modulo the characteristic of $\Fq$.

\begin{proposition}\label{prop:efqdist}
Let $E/\Fq$ be an ordinary elliptic curve and for integers $r\ge 1$, let $t_{q^r}\coloneqq q^r+1-\#E(\F_{q^r})$ and define
\[
x_r \coloneqq t_{q^r}/q^{r/2}.
\]
The sequence $(x_r)$ is equidistributed in $[-2,2]$ with respect to the measure
\[
\mu \coloneqq \frac{1}{\pi}\frac{dz}{\sqrt{4-z^2}},
\]
where $dz$ is the Lebesgue measure on $[-2,2]$.
\end{proposition}
\begin{proof}
Let $\alpha$ be as above, with $|\alpha|=q^{1/2}$ and $\tr \pi_E=\alpha+\bar\alpha$.
Then $x_r = (\alpha^r+\bar\alpha^r)/q^{r/2}$ for all $r\ge 1$.
Let $\U(1)\coloneqq\{u\in \C^\times:u\bar u=1\}$ be the unitary group.
For $u=e^{i\theta}$, the Haar measure on $\U(1)$ corresponds to the uniform measure on $\theta\in [-\pi,\pi]$, this follows immediately from the translation invariance of the Haar measure.
Let us compute the pushforward of the Haar measure of $\U(1)$ to $[-2,2]$ via the map $u\mapsto z\coloneqq u+\bar u=2\cos\theta$.  We have $dz=2\sin\theta d\theta $, and see that the pushforward is precisely $\mu$.

The nontrivial irreducible characters $\U(1)\to \C^\times$ all have the form $\phi_a(u)=u^a$ for some nonzero $a\in \Z$.
For each such $\phi_a$ we have
\[
\lim_{n\to\infty}\frac{1}{n}\sum_{r=1}^n\phi_a(\alpha^r/q^{r/2}) = \lim_{n\to\infty}\frac{1}{n}\sum_{r=1}^n (\alpha/q^{1/2})^{ra} = \lim_{n\to\infty}\frac{1}{n}\frac{(\alpha/q^{1/2})^{a(n+1)} - (\alpha/q^{1/2})^a}{(\alpha/q^{1/2})^a-1} = 0.
\]
The hypothesis that $E$ is ordinary guarantees that $\alpha/q^{1/2}$ is not a root of unity (see Exercise~\ref{ex:ord}), thus $(\alpha/q^{1/2})^a-1$ is nonzero for all nonzero $a\in \Z$.
Corollary~\ref{cor:char} implies that $(\alpha^r/q^{r/2})$ is equidistributed in $\U(1)$, and therefore $(x_r)$ is $\mu$-equidistributed.
\end{proof}

See \cite{AS12} for a generalization to smooth projective curves $C/\Fq$ of arbitrary genus $g\ge 1$.

\subsection{Equidistribution for \texorpdfstring{$L$}{{\it L}}-functions}\label{sec:eqLfunc}

As above, let $G$ be a compact group and let $X\coloneqq \conj(G)$.
Let~$K$ be a number field, and let $P\coloneqq (\p_1,\p_2,\p_3,\ldots)$ be a sequence consisting of all but finitely many primes~$\p$ of~$K$ ordered by norm; this means that $N(\p_i)\le N(\p_j)$ for all $i\le j$.
Let $(x_\p)$ be a sequence in~$X$ indexed by $P$, and for each irreducible representation $\rho\colon G\to \GL_d(\C)$, define the $L$-function
\[
L(\rho,s)\coloneqq \prod_{\p \in P} \det(1-\rho(x_\p)N(\p)^{-s})^{-1},
\]
for $s\in \C$ with $\Re(s) > 1$.

\begin{theorem}\label{thm:Lfunc}
Let $G$ and $(x_\p)$ be as above, and suppose $L(\rho,s)$ is meromorphic on $\Re(s)\ge 1$ with no zeros or poles except possibly at $s=1$, for every irreducible representation $\rho$ of $G$.
The sequence $(x_\p)$ is equidistributed if and only if for each $\rho\ne 1$, the $L$-function $L(\rho,s)$ extends analytically to a function that is holomorphic and nonvanishing on $\Re(s)\ge 1$.
\end{theorem}
\begin{proof}
See the corollary to \cite[Thm.\ A.2]{Se68}, or see \cite[Thm.\ 2.3]{Fite15}.
\end{proof}

A notable case in which the hypothesis of Theorem~\ref{thm:Lfunc} is known to hold is when $L(\rho,s)$ corresponds to an Artin $L$-function. As in Section \ref{sec:wtzero}, to each prime $\p$ in $K$ we associate an absolute Frobenius element $\Frob_\p\in\Gal(\Kbar/K)$, and for each finite Galois extension $L/K$ we use $\conj_L(\Frob_\p)$ to denote the conjugacy class in $\Gal(L/K)$ of the restriction of $\Frob_\p$ to $L$.

\begin{corollary}\label{cor:artinL}
Let $L/K$ be a finite Galois extension with $G\coloneqq\Gal(L/K)$ and let $P$ be the sequence of unramified primes of~$K$ ordered by norm (break ties arbitrarily).
The sequence $(\conj_L(\Frob_\p))_{\p\in P}$ is equidistributed in $\conj(G)$; in particular, the Chebotarev density theorem (Theorem~\ref{thm:chebotarev}) holds.
\end{corollary}
\begin{proof}
For the trivial representation, the $L$-function $L(1,s)$ agrees with the Dedekind zeta function $\zeta_K(s)$ up to a finite number of holomorphic nonvanishing factors, and, as originally proved by Hecke, $\zeta_K(s)$ is holomorphic and nonvanishing on $\Re(s)\ge 1$ except for a simple pole at $s=1$; see \cite[Cor.\ VII.5.11]{Ne99}, for example.
For every nontrivial irreducible representation $\rho\colon G\to \GL_d(\C)$, the $L$-function $L(\rho, s)$ agrees with the corresponding  Artin $L$-function for $\rho$, up to a finite number of holomorphic nonvanishing factors, and, as originally proved by Artin, $L(\rho, s)$ is holomorphic and nonvanishing on $\Re(s)\ge 1$; see \cite[p.225]{CF10}, for example.
The corollary then follows from Theorem~\ref{thm:Lfunc}.
\end{proof}

\subsection{Sato--Tate for CM elliptic curves}

As a second application of Theorem~\ref{thm:Lfunc}, let us prove an equidistribution result for CM elliptic curves.
To do so we need to introduce Hecke characters, which we will view as (quasi-)characters of the id\`ele class group of a number field.

\begin{definition}
Let $K$ be a number field and let $\I_K$ denote its id\`ele group.
A \emph{Hecke character} is a continuous homomorphism
\[
\psi\colon \I_K\to\C^\times
\]
whose kernel contains $K^\times$.
The \emph{conductor} of $\psi$ is the $\Z_K$-ideal $\mathfrak f\coloneqq \prod_\p \p^{e_\p}$ in which each $e_\p$ is the minimal nonnegative integer for which $1+\hat\p^{e_\p}\subseteq \Z_{K_\p}^\times\hookrightarrow \I_K$ lies in the kernel of $\psi$ (all but finitely many~$e_\p$ are zero because $\psi$ is continuous); here $\hat\p$ denotes the maximal ideal of the valuation ring $\Z_{K_\p}$ of $K_\p$, the completion of $K$ with respect to its $\p$-adic absolute value.
\end{definition}

Each Hecke character $\psi$ has an associated \emph{Hecke $L$-function}
\[
L(\psi,s)\coloneqq \prod_{\p\,\nmid\, \mathfrak f} (1-\psi(\p)N(\p)^{-s})^{-1},
\]
where $\psi(\p)\coloneqq \psi(\pi_{\hat \p})$ for any uniformizer $\pi_{\hat \p}$ of $\hat\p$ (we have omitted the gamma factors at archimedean places).
We now want to consider the sequence of unitarized values
\[
x_\p\coloneqq \frac{\psi(\p)}{|\psi(\p)|}\in \U(1)
\]
indexed by primes $\p\nmid \f$ ordered by norm.

\begin{lemma}\label{lem:Hecke}
The sequence $(x_\p)$ is equidistributed in $\U(1)$.
\end{lemma}
\begin{proof}
As in the proof of Proposition~\ref{prop:efqdist}, the nontrivial irreducible characters of $\U(1)$ are those of the form $\phi_a(z)=z^a$ with $a\in\Z$ nonzero, and in each case the corresponding $L$-function is a Hecke $L$-function (if $\psi$ is a Hecke character, so is $\psi^a$ and its unitarized version).
If $\psi$ is trivial then, as in the proof of Corollary~\ref{cor:artinL}, $L(1,s)$ is holomorphic and nonvanishing on $\Re(s)\ge 1$ except for a simple pole at $s=1$, since the same is true of $\zeta_K(s)$.
Hecke proved \cite{He20} that when $\psi$ is nontrivial $L(\psi,s)$ is holomorphic and nonvanishing on $\Re(s)\ge 1$, and the lemma then follows from Theorem~\ref{thm:Lfunc}.
\end{proof}

As an application of Lemma~\ref{lem:Hecke}, we can now prove the Sato-Tate conjecture for CM elliptic curves.
Les us first consider the case where $K$ is an imaginary quadratic field and $E/K$ is an elliptic curve with CM by $K$ (so $K\simeq \End(E)\otimes_\Z\Q$).
As explained below, the general case (including $K=\Q$) follows easily.

Let $\f$ be the \emph{conductor} of $E$; this is a $\Z_K$-ideal divisible only by the primes of bad reduction for $E$; see \cite[\S IV.10]{Si94} for a definition.
A classical result of Deuring \cite[Thm.\ II.10.5]{Si94} implies the existence of a Hecke character $\psi_E$ of $K$ of conductor $\f$ such that for each prime $\p\nmid\f$ we have $|\psi_E(\p)|=N(\p)^{1/2}$ and
\[
\psi_E(\p)+\overline{\psi_E(\p)} = t_\p,
\]
where $t_\p\coloneqq \tr \pi_E = N(\p)+1-\#E_\p(\F_\p)\in \Z$ is the trace of Frobenius of the reduction of $E$ modulo $\p$.

\begin{proposition}
Let $K$ be an imaginary quadratic field and let $E/K$ be an elliptic curve of conductor~$\f$ with CM by $K$.
Let $P$ be the sequence of primes of $K$ that do not divide $\f$ ordered by norm (break ties arbitrarily), and for $\p\in P$ let $x_\p\coloneqq t_{\p}/N(\p)^{1/2}\in [-2,2]$ be the normalized Frobenius trace of $E_\p$.
The sequence $(x_\p)$ is equidistributed on $[-2,2]$ with respect to the measure
\[
\mu_{\rm cm} \coloneqq \frac{1}{\pi}\frac{dz}{\sqrt{4-z^2}}.
\]
\end{proposition}
\begin{proof}
By the previous lemma, the sequence $(\psi_E(\p)/N(\p)^{1/2})_{\p \in P}$ is equidistributed in $\U(1)$.
As shown in the proof of Proposition~\ref{prop:efqdist}, the measure $\mu_{\rm cm}$ is the pushforward of the Haar measure on $\U(1)$ to $[-2,2]$ under the map $u\mapsto u+\bar u$.
For each $\p\in P$ the image of $\psi_E(\p)/N(\p)^{1/2}$ under this map is
\[
\frac{\psi_E(\p)}{N(\p)^{1/2}} + \frac{\overline{\psi_E(\p)}}{N(\p)^{1/2}} = \frac{t_{p}}{N(\p)^{1/2}}=x_\p.\qedhere
\]
\end{proof}

Figure~\ref{fig:g1cm2} shows a trace histogram for the CM elliptic curve $y^2=x^3+1$ over its CM field $\Q(\sqrt{-3})$.

\begin{figure}[bh!]
\begin{center}
\animategraphics[width=\textwidth,poster=last]{1.5}{g1_D3_a1_}{0}{30}

\caption{Click image to animate (requires Adobe Reader), or visit this \href{http://math.mit.edu/~drew/g1_D2_a1f.gif}{web page}.}\label{fig:g1cm2}
\end{center}
\end{figure}

Let us now consider the case of an elliptic curve $E/\Q$ with CM by $F$.
For primes $p$ of good reduction that are inert in $F$, the endomorphism algebra $\End(E_p)_\Q\coloneqq \End(E_p)\otimes_\Z\Q$ of the reduced curve $E_p$ contains two distinct imaginary quadratic fields, one corresponding to the CM field $F\simeq \End(E)_\Q$ and the other generated by the Frobenius endomorphism (the two cannot coincide because $p$ is inert in $F$ but the Frobenius endomorphism has norm $p$ in $\End(E_p)_\Q$).
It follows that $\End(E_p)_\Q$ must be a quaternion algebra, $E_p$ is supersingular, and for $p>3$ we must have $t_p=0$, since $t_p\equiv 0\bmod p$ and $|t_p|\le 2\sqrt{p}$; see \cite[III,9,V.3]{Si09} for these and other facts about endomorphism rings of elliptic curves.

At split primes $p=\p\bar\p$ the reduced curve $E_p$ will be isomorphic to the reduction modulo $\p$ of its base change to $F$ (both of which are elliptic curves over $\Fp=\F_\p$), and will have the same trace of Frobenius $t_p=t_\p$.
By the Chebotarev density theorem, the split and inert primes both have density $1/2$, and it follows that the sequence of normalized Frobenius traces $x_\p\coloneqq t_p/\sqrt{p}\in  [-2,2]$ is equidistributed with respect to the measure $\frac{1}{2}\delta_0+ \frac{1}{2}\mu_{\rm cm}$, where we use the Dirac measure $\delta_0$ to put half the mass at~$0$ to account for the inert primes.
This can be seen in Figure~\ref{fig:g1cm}, which shows a trace histogram for the CM elliptic curve $y^2=x^3+1$ over $\Q$; the thin spike in the middle of the histogram at zero has area $1/2$ (one can also see that the nontrivial moments are half what they were in Figure~\ref{fig:g1cm2}).

\begin{figure}[htp!]
\begin{center}
\animategraphics[width=\textwidth,poster=last]{1.5}{g1_cm_a1_}{0}{30}

\caption{Click image to animate (requires Adobe Reader), or visit this \href{http://math.mit.edu/~drew/g1_D2_a1f.gif}{web page}.}\label{fig:g1cm}
\end{center}
\end{figure}

A similar argument applies when $E$ is defined over a number field $K$ that does not contain the CM field $F$.
For the sake of proving an equidistribution result we can restrict our attention to the degree-1 primes $\p$ of $K$, those for which $N(\p)=p$ is prime.
Half of these primes $\p$ will split in the compositum~$KF$, and the subsequence of normalized traces $x_\p$ at these primes will be equidistributed with respect to the measure $\mu_{\rm cm}$, and half will be inert in $KF$, in which case $x_\p=t_\p=0$.

\subsection{Sato--Tate for non-CM elliptic curves}

We can now state the Sato-Tate conjecture in the form originally given by Tate, following \cite[\S1A]{Se68}.
Tate's seminal paper \cite{Tate63} describes what is now known as the \emph{Tate conjecture}, which comes in two conjecturally equivalent forms \textbf{T1} and \textbf{T2}, the latter of which is stated in terms of $L$-functions.
The Sato-Tate conjecture is obtained by applying \textbf{T2} to all powers of a fixed elliptic curve $E/\Q$ (as products of abelian varieties); see \cite{Rama07} for an introduction to the Tate conjecture and an explanation of how the Sato-Tate conjecture fits within it.

Let $G$ be the compact group $\SU(2)$ of $2\times 2$ unitary matrices with determinant $1$.
The irreducible representations of $G$ are the $m$th symmetric powers $\rho_m$ of the natural representation $\rho_1$ of degree 2 given by the inclusion $\SU(2)\subseteq\GL_2(\C)$.
Each element of $X\coloneqq \conj(G)$ can be uniquely represented by a matrix of the form
\[
\begin{pmatrix}e^{i\theta} & 0\\ 0 &e^{-i\theta}\end{pmatrix},
\]
where $\theta\in [0,\pi]$ is the eigenangle of the conjugacy class.
It follows that each $f\in C(X)$ can be viewed as a continuous function $f(\theta)$ on the compact set $[0,\pi]$.

The pushforward of the Haar measure of $G$ to $X$ is given by
\begin{equation}\label{eq:su2haar}
\mu = \frac{2}{\pi}\sin^2\theta\,d\theta
\end{equation}
(see Exercise~\ref{ex:su2}), which means that for each $f\in C(X)$ we have
\[
\mu(f)=\frac{2}{\pi}\int_0^\pi f(\theta)\sin^2\theta\,d\theta.
\]

Let $E/\Q$ be an elliptic curve without CM, let $P\coloneqq (p)$ be the sequence of primes that do not divide the conductor $N$ of $E$, in order, and for each $p\in P$ let $x_p\in X$ to be the element of $X$ corresponding to the unique $\theta_p\in [0,\pi]$ for which $2\cos\theta_p\sqrt{p}= t_p\coloneqq p+1-\#E_p(\F_p)$ is the trace of Frobenius of the reduced curve $E_p$.

We are now in the setting of \S\ref{sec:eqLfunc}.  We have a compact group $G\coloneqq\SU(2)$, its space of conjugacy classes $X\coloneqq \conj(G)$, a number field $K=\Q$, a sequence  $P$ containing all but finitely many primes of~$K$ ordered by norm, a sequence $(x_p)$ in $X$ indexed by $P$, and for each integer $m\ge 1$, an irreducible representation $\rho_m\colon G\to \GL_{m+1}(\C)$.
The $L$-function corresponding to $\rho_m$ is given by
\[
L(\rho_m,s)\coloneqq \prod_{p\,\nmid\,N} \det(1-\rho_m(x_p)p^{-s})^{-1} = \prod_{p\,\nmid\,N} \prod_{k=0}^m (1-e^{i(m-2k)\theta_p}p^{-s})^{-1}.
\]
For each $p\nmid N$, let $\alpha_p$ and $\bar\alpha_p$ be the roots of $T^2-t_pT+p$, so that $\alpha_p = e^{i\theta_p}p^{1/2}$.
If we now define
\[
L^1_m(s)\coloneqq\prod_{p\,\nmid\, N}\prod_{r=0}^m(1-\alpha_p^{m-r}{\bar\alpha_p^r}p^{-s})^{-1},
\]
then for $m\ge 1$ we have
\[
L(\rho_m,s)=L_m^1(s-m/2).
\]
Tate conjectured in \cite{Tate63} that $L_m^1(s)$ is holomorphic and nonvanishing on $\Re(s)\ge 1+m/2$, which implies that each $L(\rho_m,s)$ is holomorphic and nonvanishing on $\Re(s)\ge 1$.
Assuming this is true, Theorem~\ref{thm:Lfunc} implies that the sequence $(x_p)$ is $\mu$-equidistributed, which is equivalent to the Sato-Tate conjecture.

We now recall the \emph{modularity theorem} for elliptic curves over $\Q$, which states that there is a one-to-one correspondence between isogeny classes of elliptic curves $E/\Q$ of conductor $N$ and modular forms
\[
f(z) =\sum_{n\ge 1}a_ne^{2\pi in z}\in S_2(\Gamma_0(N))^{\rm new}\qquad (a_n\in\Z \text{ with } a_1=1)
\]
that are eigenforms for the action of the Hecke algebra on the space $S_2(\Gamma_0(N))$ of cuspforms of weight~2 and level~$N$ and \emph{new} at level $N$, meaning not contained in $S_2(\Gamma_0(M))$ for any positive integer $M$ properly dividing~$N$.  Such modular forms $f$ are called (normalized) \emph{newforms}, of weight $2$ and level~$N$, with rational coefficients.
The modularity theorem was proved for squarefree $N$ by Taylor and Wiles \cite{TW95,Wiles95}, and extended to all conductors $N$ by Breuil, Conrad, Diamond, and Taylor \cite{BCDT01}.

The modular form $f$ is a simultaneous eigenform for all the Hecke operators $T_n$, and the normalization $a_1=1$ ensures that for each prime $p\nmid N$, the coefficient $a_p$ is the eigenvalue of $f$ for $T_p$.
Under the correspondence given by the modularity theorem, the eigenvalue~$a_p$ is equal to the trace of Frobenius~$t_p$ of the reduced curve $E_p$, where $E$ is any representative of the corresponding isogeny class.
Here we are using the fact that if $E$ and $E'$ are isogenous elliptic curves over $\Q$ they necessarily have the same conductor $N$ and the same trace of Frobenius $t_p$ at ever $p\nmid N$.

There is an $L$-function $L(f,s)$ associated to the modular form $f$, and the modularity theorem guarantees that it coincides with the $L$-function $L(E,s)$ of $E$.
So not only does $a_p=t_p$ for all $p\nmid N$, the Euler factors at the bad primes $p|N$ also agree.
We need not concern ourselves with Euler factors at these primes, other than to note that they are holomorphic and nonvanishing on $\Re(s)\ge 3/2$.
After removing the Euler factors at bad primes, the $L$-functions $L(E,s)$ and $L(f,s)$ both have the form
\[
\prod_{p\,\nmid\, N}(1-a_p p^{-s} +p^{1-2s})^{-1} = \prod_{p\,\nmid\, N}\prod_{r=0}^1(1-\alpha_p^{1-r}\bar\alpha_p^rp^{-s})^{-1} = L_1^1(s),
\]where $\alpha_p$ and $\bar\alpha_p$ are the roots of $T^2-a_pT+p=T^2-t_pT+p$.

The $L$-function $L(f,s)$ is holomorphic and nonvanishing on $\Re(s)\ge 3/2$; see \cite[Prop.\ 5.9.1]{DS05}.
The modularity theorem tells us that the same is true of $L(E,s)$, and therefore of $L_1^1(s)$.
Thus the modularity theorem proves that Tate's conjecture regarding $L_m^1(s)$ holds when $m=1$.
To prove the Sato-Tate conjecture one needs to show that this holds for all $m\ge 1$.

\begin{theorem}
Let $f(z)\coloneqq\sum_{n\ge 1} a_ne^{2\pi i zn}\in S_2(\Gamma_0(N)^{\rm new}$ be a normalized newform without CM.
For each prime $p\nmid N$ let $\alpha_p,\bar\alpha_p$ be the roots of $T^2-a_pT+p$.  Then
\[
\prod_{p\,\nmid\, N}\prod_{r=0}^m(1-\alpha_p^{m-r}\bar\alpha_p^rp^{-s})^{-1}=L_m^1(s)
\]
is holomorphic and nonvanishing on $\Re(s)\ge 1+m/2$.
\end{theorem}
\begin{proof}
Apply \cite[Theorem B.2]{BLGHT11} with weight $k=2$, trivial nebentypus $\psi=1$, and trivial character $\chi=1$ (as noted in \cite{BLGHT11}, this special case was already addressed in \cite{HSBT10}).
\end{proof}

\begin{corollary}
The Sato-Tate conjecture (Theorem~\ref{thm:ST}) holds.
\end{corollary}

\begin{remark}
The Sato-Tate conjecture is also known to hold for elliptic curves over totally real fields, and over CM fields (imaginary quadratic extensions of totally real fields).  The totally real case was initially proved for elliptic curves with potentially multiplicative reduction at some prime in \cite{HSBT10,Tay08}; it was later shown this technical assumption can be removed (see the introduction of \cite{BLGG11}).  The generalization to CM fields was obtained at a recent IAS workshop \cite{ACCGHLNSTT} and still in the process of being written up in detail.
As a consequence of this result the Sato-Tate conjecture for elliptic curves is now known for all number fields of degree $1$ or $2$ (but not for any higher degrees).
\end{remark}

\subsection{Exercises}
\begin{exercise}\label{ex:measure}
Let $X$ be a compact Hausdorff space.
Show that a set $S\subseteq X$ is $\mu$-quarrable for every measure $\mu$ on $X$ if and only if the set $S$ is both open and closed.
\end{exercise}
\begin{exercise}\label{ex:mueq}
Prove Proposition~\ref{prop:mueq}.
\end{exercise}
\begin{exercise}\label{ex:ord}
Let $E$ an elliptic curve over $\Fq$ and let $\alpha$ be a root of the characteristic polynomial of the Frobenius endomorphism $\pi_E$.
Prove that $\alpha/\sqrt{q}$ is a root of unity if and only if $E$ is supersingular.
\end{exercise}
\begin{exercise}\label{ex:su2}
Show that the set of conjugacy classes of $\SU(2)$ is in bijection with the set of eigenangles $\theta\in [0,\pi]$.
Then prove that the pushforward of the Haar measure of $\SU(2)$ onto $[0,\pi]$ is given by $\mu\coloneqq \frac{2}{\pi}\sin^2\theta\, d\theta$ (hint: show that $\SU(2)$ is isomorphic to the 3-sphere $S^3$ and use this isomorphism together with the translation invariance of the Haar measure to determine $\mu$)
\end{exercise}
\begin{exercise}\label{ex:STg1}
Compute the trace moment sequence for $\SU(2)$ (that is, prove \eqref{eq:cat}).
Embed $\U(1)$ in $\SU(2)$ via the map $u\mapsto \smallmat{u}{0}{0}{\bar u}$ and compute its trace moment sequence (compare to Figure~\ref{fig:g1cm2}).
Now determine the normalizer $N(\U(1))$ of $\U(1)$ in $\SU(2)$ and compute its trace moment sequence (compare to Figure~\ref{fig:g1cm}).
\end{exercise}

\section{Sato-Tate groups}\label{lec:STgroups}

In the previous section we showed that there are three distinct Sato-Tate distributions that arise for elliptic curves $E$ over number fields $K$ (only two of which occur when $K=\Q$).
All three distributions can be associated to the Haar measure of a compact subgroup $G\subseteq \SU(2)$, in which we embed $\U(1)$ via the map $u\mapsto \smallmat{u}{0}{0}{\bar u}$.
We are interested in the pushforward $\mu$ of the Haar measure onto $\conj(G)$, which can be expressed in terms of the eigenangle $\theta\in [0,\pi]$.  The three possibilities for $G$ are listed below.
\begin{itemize}
\item $\U(1)$: we have $\mu(\theta)=\frac{1}{\pi}d\theta$ and trace moments: $(1,0,2,0,6,0,20,0,70,0,252,\ldots)$.\\
This case arises for CM elliptic curves defined over a field that contains the CM field.
\smallskip
\item $N(\U(1))$: we have $\mu(\theta)=\frac{1}{2\pi}d\theta+\frac{1}{2}\delta_{\pi/2}$ and trace moments: $(1,0,1,0,3,0,10,0,35,0,126,\ldots)$.\\
This case arises for CM elliptic curves defined over a field that does not contain the CM field.
\smallskip
\item $\SU(2)$: we have $\mu(\theta)=\frac{2}{\pi}\sin^2\theta\,d\theta$ and trace moments: $(1,0,1,0,2,0,5,0,14,0,42,\ldots)$.\\
This case arises for all non-CM elliptic curves (conjecturally so when $K$ not totally real or CM).
\end{itemize}
We have written $\mu$ in terms of $\theta$, but we may view it as a linear function on the Banach space $C(X)$, where we identify $X\coloneqq \conj(G)$ with $[0,\pi]$, by defining $\mu(f)\coloneqq\int_0^\pi f(\theta) \mu(\theta)$, as in \S\ref{sec:measure}.
In particular,~$\mu$ assigns a value to the trace function $\tr\colon X\to [-2,2]$, where $\tr(\theta)=2\cos\theta$, and to its powers $\tr^n$, which allows us to compute the trace moment sequence $(\mu(\tr^n))_{n\ge 0}$.

Our goal in this section is to define the compact group $G$ as an invariant of the elliptic curve $E$, the \emph{Sato-Tate group} of~$E$, and to then generalize this definition to abelian varieties of arbitrary dimension.  This will allow us to state the Sato-Tate conjecture for abelian varieties as an equidistribution statement with respect to the Haar measure of the Sato-Tate group.

\subsection{The Sato-Tate group of an elliptic curve}

Thus far the link between the elliptic curve $E$ and the compact group $G$ whose Haar measure is claimed (and in many cases proved) to govern the distribution of Frobenius traces has been made via the measure $\mu$.
That is, we have an equidistribution claim for the sequence $(x_p)$ of normalized Frobenius traces associated to $E$ that is phrased in terms of a measure~$\mu$ that happens to be induced by the Haar measure of a compact group $G$. 
We now want to establish a direct relationship between $E$ and $G$ that defines $G$ as an arithmetic invariant of $E$, \emph{without assuming the Sato-Tate conjecture}.

In Section \ref{sec:wtzero} we considered the Galois representation $\rho_f\colon \Gal(\Qbar/\Q)\to \GL_d(\C)$ defined by the action of $\Gal(\Qbar/\Q)$ on the roots of a squarefree polynomial $f\in\Z[x]$.
We thereby obtained a compact group~$G_f$ and a map that sends each prime $p$ of good reduction for $f$ to an element of $\conj(G_f)$ (namely, the map $p\mapsto \rho_f(\Frob_p)$).
We were then able to relate the image of $p$ under this map to the quantity $N_f(p)$ of interest, via \eqref{eq:Nfp}.
This construction did not involve any discussion of equidistribution, but we could then prove, via the Chebotarev density theorem, that the conjugacy classes $\rho_f(p)$ are equidistributed with respect to the pushforward of the Haar measure to $\conj(G_f)$. 

We take a similar approach here.  To each elliptic curve $E$ over a number field $K$ we will associate a compact group $G$ that is constructed via a Galois representation attached to $E$, equipped with a map that sends each prime $\p$ of good reduction for $E$ to an element $x_\p$ of $\conj(G)$ that we can directly relate to the quantity $N_E(\p)\coloneqq p+1-t_\p$ whose distribution we wish to study.
We may then conjecture (and prove, when $E$ has CM or~$K$ is a totally real or CM field), that the sequence $(x_\p)$ is equidistributed in $X\coloneqq \conj(G)$ (with respect to the pushforward of the Haar measure of $G$).

The group $G$ is the \emph{Sato--Tate} group of $E$, and will be denoted $\ST(E)$.
It is a compact subgroup of $\SU(2)$, and our construction will make it easy to show that $\ST(E)$ is always one of the three groups $\U(1)$, $N(\U(1))$, $\SU(2)$ listed above, depending on whether $E$ has CM or not, and if so, whether the CM field is contained in the ground field or not.
None of this depends on any equidistribution results.
This construction will be our prototype for the definition of the Sato-Tate group of an abelian variety of arbitrary dimension~$g$, so we will work out the $g=1$ case in some detail.

In order to associate a Galois representation to $E/K$, we need a set on which $\Gal(\Kbar/K)$ can act.
For each integer $n\ge 1$, let $E[n]:=E(\Kbar)[n]$ denote the $n$-torsion subgroup of $E(\Kbar)$, a free $\Z/n\Z$-module of rank $2$ (see \cite[Cor.\ III.6.4]{Si09}).
The group $\Gal(\Kbar/K)$ acts on points in $E(\Kbar)$ coordinate-wise, and $E[n]$ is invariant under this action because it is the kernel of the multiplication-by-$n$ map $[n]$, an endomorphism of $E$ that is defined over~$K$; one can concretely define $E[n]$ as the zero locus of the $n$-division polynomials, which have coefficients in~$K$.
The action of $\Gal(\Kbar/K)$ on $E[n]$ induces the \emph{mod-$n$ Galois representation}
\[
\Gal(\Kbar/K)\to \Aut(E[n])\simeq \GL_2(\Z/n\Z).
\]
This Galois representation is insufficient for our purposes, because the image $M_\p$ of $\Frob_\p$ in $\GL_2(\Z/n\Z)$ does not determine $t_\p$, we only have $t_\p\equiv \tr M_\p\bmod n$; we need to let $\Gal(\Kbar/K)$ act on a bigger set.

So let us fix a prime $\ell$ (any prime will do), and consider the inverse system
\[
\cdots\overset{[\ell]}{\longrightarrow} E[\ell^3]\overset{[\ell]}{\longrightarrow}E[\ell^2]\overset{[\ell]}\longrightarrow E[\ell].
\]
The inverse limit
\[
T_\ell\coloneqq \varprojlim_n E[\ell^n]
\]
is the $\ell$-adic \emph{Tate-module} of $E$; it is a free $\Z_\ell$-module of rank 2.
The group $\Gal(\Kbar/K)$ acts on $T_\ell$ via its action on the groups $E[\ell^n]$, and this action is compatible with the multiplication-by-$\ell$ map $[\ell]$ because this map is defined over $K$ (it can be written as a rational map with coefficients in $K$).
This yields the \emph{$\ell$-adic Galois representation}
\[
\rho_{E,\ell}\colon \Gal(\Kbar/K)\to \Aut(T_\ell)\simeq \GL_2(\Z_\ell).
\]
The representation $\rho_{E,\ell}$ enjoys the following property: for every prime $\p\nmid \ell$ of good reduction for~$E$ the image of $\Frob_\p$ is a matrix $M_\p\in \GL_2(\Z_\ell)$ that has the same characteristic polynomial as the Frobenius endomorphism of $E_\p$, namely, $T^2-t_\p T+N(\p)$, where $t_\p\coloneqq \tr \pi_{E_\p}$.
Note that the matrix $M_\p$ is determined only up to conjugacy; there is ambiguity both in our choice of $\Frob_\p$ (see \S\ref{sec:wtzero}) and in our choice of a basis for $T_\ell$, which fixes the isomorphism $\Aut(T_\ell)\simeq \GL_2(\Z_\ell)$.
We should thus think of $\rho_{E,\ell}(\Frob_\p)$ as representing a conjugacy class in $\GL_2(\Z_\ell)$.

We prefer to work over the field $\Q_\ell$, rather than its ring of integers $\Z_\ell$, so let us define the \emph{rational Tate module}
\[
V_\ell\coloneqq T_\ell\otimes_\Z\Q,
\]
which is a 2-dimensional $\Q_\ell$-vector space equipped with an action of $\Gal(\Kbar/K)$.
This allows us to view the Galois representation $\rho_{E,\ell}$ as having image $G_{\ell}\subseteq \GL_2(\Q_\ell)$.
We also prefer to work with an algebraic group, so let us define $G_{\ell}^{\rm zar}$ to be the $\Q_\ell$-algebraic group obtained by taking the Zariski closure of $G_{\ell}$ in $\GL_2(\Q_\ell)$. This means that $G_{\ell}^{\rm zar}$ is the affine variety defined by the ideal of $\Q_\ell$-polynomials that vanish on the set $G_\ell$; it is a subvariety of $\GL_2/\Q_\ell$ that is closed under the group operation and thus an algebraic group over $\Q_\ell$.  The algebraic group $G_{\ell}^{\rm zar}$ is the \emph{$\ell$-adic monodromy group} of $E$ (it is also denoted $G_{\ell}^{\rm alg}$).

\begin{background}[Algebraic groups]
An affine (or linear) \emph{algebraic group} over a field $k$ is a group object in the category of (not necessarily irreducible) affine varieties over $k$.
The only projective algebraic groups we shall consider are smooth and connected, hence abelian varieties, so when we use the term algebraic group without qualification, we mean an affine algebraic group.\footnote{There are interesting algebraic groups (group schemes of finite type over a field) that are neither affine nor projective (even if we restrict our attention to those that are smooth and connected), but we shall not consider them here.}
The canonical example is $\GL_n$, which can be defined as an affine variety in $\A^{n^2+1}$ (over any field) by the equation $t \det M=1$ (here $\det M$ denotes the determinant polynomial in $n^2$ variables $M_{ij}$), with morphisms $m\colon \GL_n\times\GL_n\to\GL_n$ and $i\colon \GL_n\to\GL_n$ defined by polynomial maps corresponding to matrix multiplication and inversion (one uses $t$ as the inverse of $\det A$ when defining~$i$).
The classical groups $\SL_n$, $\Sp_{2n}$,$\U_n$, $\SU_n$, $\O_n$, $\SO_n$  are all affine algebraic groups (assume $\mathrm{char}(k)\ne 2$ for $\O_n$ and $\SO_n$), as are the groups $\USp_{2n}\coloneqq \Sp_{2n}\cap \U_{2n}$ and $\GSp_{2n}$ that are of particular interest to us; the $\R$ and $\C$ points of these groups are \emph{Lie groups} (differentiable manifolds with a group structure).
If $G$ is an affine algebraic group over $k$ and $L/k$ is a field extension, the Zariski closure of any subgroup $H\subseteq G(L)$ of the $L$-points of~$G$ is equal to the set of rational points of an affine variety defined over $L$ that is also an algebraic group via the morphisms $m$ and~$i$ defining~$G$.  Thus every subgroup $H\subseteq G(L)$ uniquely determines an algebraic group over $L$ whose rational points coincide with the Zariski closure of $H$; as an abuse of terminology we may refer to this algebraic group as the Zariski closure of $H$ in $G(L)$ (or in $G_L$, the base change of $G$ to $L$).
The connected and irreducible components of an algebraic group $G$ coincide, and are necessarily finite in number.
The connected component $G^0$ of the identity is itself an algebraic group, a normal subgroup of $G$ compatible with base change.
For more on algebraic groups see any of the classic texts \cite{Borel91,Hu75,Spr98}, or see \cite{Milne:iAG} for a more modern treatment.
\end{background}

Having defined the $\Q_\ell$-algebraic group $G_{\ell}^{\rm zar}$, we now restrict our attention to the subgroup $G_{\ell}^{1,\rm zar}$ obtained by imposing the symplectic constraint
\[
M^t\Omega M=\Omega, \qquad\qquad \Omega\coloneqq \smallmat{0}{-1}{1}{0},
\]
which corresponds to putting a \emph{symplectic form} (a nondegenerate bilinear alternating pairing) on the vector space $V_\ell$ (we could of course choose any $\Omega$ that defines such a form).
This condition can clearly be expressed by a polynomial (a quadratic form in fact), thus $G_{\ell}^{1,\rm zar}$ is an algebraic group over $\Q_\ell$ contained in $\Sp_2$.
We remark that $\Sp_2=\SL_2$, so we could have just required $\det M=1$, but this is an accident of low dimension: the inclusion $\Sp_{2n}\subseteq \SL_{2n}$ is strict for all $n>1$.

Finally, let us choose an embedding $\iota\colon \Q_\ell\to \C$, and let $G_{\ell,\iota}^{1,\rm zar}$ be the $\C$-algebraic group obtained from $G_{\ell}^{1,\rm zar}$ by base change to $\C$ (via $\iota$).
The group $G_{\ell,\iota}^{1,\rm zar}(\C)$ is a subgroup of $\Sp_2(\C)$ that we may view as a Lie group with finitely many connected components.
It therefore contains a maximal compact subgroup that is unique up to conjugacy \cite[Thm.\ IV.3.5]{OV94}, and we take this as the \emph{Sato--Tate group} $\ST(E)$ of $E$ (which is thus defined only up to conjugacy).
It is a compact subgroup of $\USp(2)=\SU(2)$ (this equality is another accident of low dimension).

For each prime $\p\nmid \ell$ of good reduction for $E$, let $M_\p$ denote the image of $\Frob_\p$ under the maps
\[
\Gal(\Kbar/K)\overset{\rho_{E,\ell}}\longrightarrow G_{\ell}\hookrightarrow G_{\ell}^{\rm zar}(\Q_\ell)\hookrightarrow G_{\ell,\iota}^{\rm zar}(\C),
\]
where the map in the middle is inclusion and we use the embedding $\iota\colon\Q_\ell\to\C$ to obtain the last map.
We now consider the normalized Frobenius image
\[
\bar M_\p\coloneqq N(\p)^{-1/2}M_\p;
\]
it is a matrix with trace $t_\p/N(\p)^{-1/2}\in [-2,2]$ and determinant $1$, and its eigenvalues $e^{\pm i\theta_\p}$ lie on the unit circle.\footnote{Note that we embed $G_{\ell}^{\rm zar}(\Q_\ell)$ in $G_{\ell,\iota}^{\rm zar}(\C)$ \emph{before} normalizing by $N(\p)^{-1/2}$; as pointed out by Serre \cite[p. 131]{Se12}, we want to take the square root in $\C$ where it is unambiguously defined.}
The eigenangle $\theta_\p$ determines a unique conjugacy class in $\ST(E)$, which we take as $x_\p$.
The characteristic polynomial of $x_\p$ is the normalized $L$-polynomial
$\bar L_\p(T)\coloneqq L_p(N(\p)^{-1/2}T)$, where $L_\p(T)$ is the numerator of the zeta function of $E_\p$, and $L_\p(N(\p)^{-s})$ is the Euler factor at $\p$ in the $L$-series $L(E,s)$.

The Sato--Tate conjecture then amounts to the statement that the sequence $(x_\p)$ in $X\coloneqq \conj(\ST(E))$ is equidistributed.
Notice that the statement is the same in both the CM and non-CM cases, but the measure on $X$ is different, because $\ST(E)$ is different.  Indeed, there are three possibilities for $\ST(E)$, corresponding to the three distributions that we noted at the beginning of this section.

\begin{theorem}\label{thm:STg1}
Let $E$ be an elliptic curve over a number field $K$.
Up to conjugacy in $\SU(2)$ we have
\[
\ST(E)=\begin{cases}
\U(1) &\text{ if $E$ has CM defined over $K$},\\
N(\U(1)) &\text{ if $E$ has CM not defined over $K$},\\
\SU(2) &\text{ if $E$ does not have CM},
\end{cases}
\]
where $\U(1)$ is embedded in $\SU(2)$ via $u\mapsto \smallmat{u}{0}{0}{\bar u}$.
\end{theorem}
\begin{proof}
If $E$ has CM defined over $K$ then $G_\ell$ is abelian, because the action of $\Gal(\Kbar/K)$ on $V_\ell$ factors through the abelian group $\Gal(L/K)$, where $L\coloneqq K(E[\ell^\infty])$ is obtained by adjoining the coordinates of the $\ell$-power torsion points of $E$; this follows from \cite[Thm.\ II.2.3]{Si94}.
Therefore $G_\ell$ lies in a \emph{Cartan subgroup} of $\GL_2(\Q_\ell)$ (a maximal abelian subgroup), which necessarily splits when we pass to $G_{\ell,\iota}^{\rm zar}(\C)$, where it is conjugate to the group of diagonal matrices.
This implies that $\ST(E)$ is conjugate to $\U(1)$, the subgroup of diagonal matrices in $\SU(2)$.

If $E$ has CM not defined over $K$, then $G_\ell$ lies in the normalizer of a Cartan subgroup of $\GL_2(\Q_\ell)$, but not in the Cartan itself, and $\ST(E)$ is conjugate to the normalizer $N(\U(1))$ of $\U(1)$ in $\SU(2)$; the argument is as above, but now the action of $\Gal(\Kbar/K)$ factors through $\Gal(FL/K)$, where $F$ is the CM field and $\Gal(FL/K)$ contains the abelian subgroup $\Gal(FL/FK)$ with index~2.

If $E$ does not have CM then Serre's open image theorem (see \cite[\S IV.3]{Se68} and \cite{Se72}) implies that $G_\ell$ is a finite index subgroup of $\GL_2(\Z_\ell)$; we therefore have $G_\ell^{1,\rm zar}=\SL_2$, which implies $\ST(E)=\SU(2)$.
\end{proof}

It follows from Theorem~\ref{thm:STg1} that (up to conjugacy), the Sato--Tate group $\ST(E)$ does not depend on our choice of the prime $\ell$ or the embedding $\iota\colon \Q_\ell\to\C$ that we used.
We should also note that $\ST(E)$ depends only on the isogeny class of $E$; this follows from the fact that we used the rational Tate module~$V_\ell$ to define it (indeed, two abelian varieties over a number field are isogenous if and only if their rational Tate modules are isomorphic as Galois modules, by Faltings' isogeny theorem \cite{Falt83}, but we are only using the easy direction of this equivalence here).

\subsection{The Sato--Tate group of an abelian variety}

We now wish to generalize our definition of the Sato--Tate group of an elliptic curve to abelian varieties.
Recall that an \emph{abelian variety} is a smooth connected projective variety that is also an algebraic group, where the group operations are now given by morphisms of projective varieties; on any affine patch they can be defined by a polynomial map.
Remarkably, the fact that abelian varieties are commutative algebraic groups is not part of the definition, it is a consequence; see \cite[Cor.\,1.4]{Milne:AV}.
We also recall that an \emph{isogeny} of abelian varieties is simply an isogeny of algebraic groups, a surjective morphism with finite kernel.

Abelian varieties of dimension $g$ may arise as the Jacobian $\Jac(C)$ of a smooth projective curve $C/k$ of genus $g$.
If $C$ has a $k$-rational point (as when $C$ is an elliptic curve), one can functorially identify $\Jac(C)$ with the \emph{divisor class group} $\Pic^0(C)$, the group of degree-zero divisors modulo principal divisors, but one can unambiguously define the abelian variety $\Jac(C)$ in any case; see \cite[Ch. III]{Milne:AV} for details.

If $C$ is a smooth projective curve over a number field $K$ and $A\coloneqq \Jac(C)$ is its Jacobian, then for every prime $\p$ of good reduction for $C$, the abelian variety $A$ also has good reduction at $\p$,\footnote{For $g>1$ the converse does not hold (in general); this impacts only finitely many primes $\p$ and will not concern us.} and the $L$-polynomial $L_\p(T)$ appearing in the numerator of the zeta function $Z_{C_p}(T)$ is reciprocal to the characteristic polynomial $\chi_p(T)$ of the Frobenius endomorphism $\pi_{A_\p}$ of $A_\p$, which acts on points of $A$ via the $N(\p)$-power Frobenius automorphism (coordinate-wise).
In particular, we have the identity
\begin{equation}\label{eq:LpAV}
L_\p(T) = T^{2g}\chi_\p(T^{-1}),
\end{equation}
in which both sides are integer polynomials of degree $2g$ whose complex roots have absolute value $N(\p)^{-1/2}$.
As with elliptic curves, one can consider the $L$-function $L(A,s)$ attached to $A$, which is defined as an Euler product with factors $L_\p(N(\p)^{-s})$ at each prime $\p$ where $A$ has good reduction.\footnote{Explicitly determining the Euler factors at bad primes is difficult when $\dim A >1$. Practical methods are known only in special cases, such as when $A$ is the Jacobian of a hyperelliptic curve (even in this case there is still room for improvement).}
Studying the distribution of the normalized $L$-polynomials $\bar L_\p(T)$ associated to $C$ is thus equivalent to studying the distribution of the normalized characteristic polynomials of $\pi_{A_\p}$, and also equivalent to studying the distribution of the normalized Euler factors of $L(A,s)$.

\begin{remark}
Each of these three perspectives is successively more general than the previous, the last vastly so.
There are abelian varieties over $K$ that are not the Jacobian of any curve defined over~$K$, and $L$-functions that can be written as Euler products over primes of $K$ that are not the $L$-function of any abelian variety.
One can more generally consider the distribution of normalized Euler factors of \emph{motivic $L$-functions}, which we also expect to be governed by the Haar measure of a Sato-Tate group associated to the underlying motive, as defined in \cite{Se94,Se12}; see \cite{FKS16} for some concrete examples in weight 3.
\end{remark}

The recipe for defining the Sato-Tate group $\ST(A)$ of an abelian variety $A/K$ of genus $g$ is a direct generalization of the $g=1$ case.
We proceed as follows:
\begin{enumerate}[1.]
\setlength{\itemsep}{2pt}
\item Pick a prime $\ell$, define the Tate module $T_\ell\coloneqq \varprojlim_n A[\ell^n]$, a free $\Z_\ell$-module of rank $2g$, and the rational Tate module $V_\ell\coloneqq T_\ell\otimes_\Z \Q$, a $\Q_\ell$-vector space of dimension $2g$.
\item Use the Galois representation $\rho_{A,\ell}\colon \Gal(\Kbar/K)\to\Aut(V_\ell)\simeq \GL_{2g}(\Q_\ell)$ to define $G_\ell\coloneqq\im\rho_{A,\ell}$.
\item Let $G_\ell^{\rm zar}$ be the Zariski closure of $G_\ell$ in $\GL_{2g}(\Q_\ell)$ (as an algebraic group), and define $G_\ell^{1,\rm zar}$ by adding the symplectic constraint $M^t\Omega M=\Omega$, so that $G_\ell^{1,\rm zar}$ is a $\Q_\ell$-algebraic subgroup of $\Sp_{2g}$.
\item Pick an embedding $\iota\colon \Q_\ell\to \C$ and use it to define $G_{\ell,\iota}^{1,\rm zar}$ as the base-change of $G_\ell^{1,\rm zar}$ to $\C$.
\item Define $\ST(A)\subseteq \USp(2g)$ as a maximal compact subgroup of $G_{\ell,\iota}^{1,\rm zar}(\C)$, unique up to conjugacy.
\item For each good prime $\p\nmid \ell$, let $M_\p$ be the image of $\Frob_\p$ in $G_{\ell,\iota}^{\rm zar}(\C)$ and define $x_\p\in \conj(\ST(A))$ to be the conjugacy class of $\overline M_\p\coloneqq N(\p)^{-1/2}M_\p,$ in $\ST(A)$.
\end{enumerate}
Step 6 requires some justification; it is not obvious why $\overline M_\p$ should necessarily be conjugate to an element of $\ST(A)$.
Here we are relying on two key facts.

First, the image $G_\ell$ of $\rho_{A,\ell}$ in $\GL_{2g}(\Q_\ell)$ actually lies in $\GSp_{2g}(\Q_\ell)$, the group of \emph{symplectic similitudes}.
The algebraic group $\GSp_{2g}$ is defined by imposing the constraint
\[
M^t\Omega M=\lambda\Omega,\qquad\Omega\coloneqq \smallmat{0}{-I_g}{I_g}{0},
\]
where $\lambda$ is necessarily an element of the multiplicative group $\Gm\coloneqq \GL_1$, since $M$ is invertible.
The morphism $\GSp_{2g}\to \Gm$ defined by $\lambda$ is the \emph{similitude character}, and we have an exact sequence of algebraic groups
\[
1\to \Sp_{2g}\hookrightarrow \GSp_{2g}\overset{\lambda}{\longrightarrow}\Gm\to 1.
\]
The action of $\Gal(\Kbar/K)$ on the Tate module is compatible with the Weil pairing, and this forces the image $G_\ell$ of $\rho_{E,\ell}$ to lie in $\GSp_{2g}(\Q_\ell)$; see Exercise~\ref{ex:gsp}.
By fixing a symplectic basis for $V_\ell$ in step 1 we can view $\rho_{A,\ell}$ as a continuous homomorphism
\[
\rho_{A,\ell}\colon \Gal(\Kbar/K)\to \GSp_{2g}(\Q_\ell)\subseteq \GL_{2g}(\Q_\ell)
\]
For $g=1$ we have $\GL_2=\GSp_2$, but for $g>1$ the algebraic group $\GSp_{2g}$ is properly contained in $\GL_{2g}$.

Second, we are relying on the fact that $M_\p$, and therefore $\overline M_\p$, is \emph{semisimple} (diagonalizable, since we are working over $\C$).
This follows from Tate's proof of the Tate conjecture for abelian varieties over finite fields (combine the main theorem and part (a) of Theorem 2 in \cite{Tate66}).
The matrix $\overline M_\p$ is thus diagonalizable and has eigenvalues of absolute value 1; it therefore lies in a compact subgroup of $G_{\ell,\iota}^{1,\rm zar}(\C)$ (take the closure of the group it generates).  This compact group is necessarily conjugate to a subgroup of the maximal~compact subgroup $\ST(A)$, which must contain an element conjugate to $\overline M_\p$.

\begin{remark}
When defining the Sato-Tate group in more general settings one instead uses the semisimple component of the (multiplicative) Jordan decomposition (see \cite[Thm.\ I.4.4]{Borel91}) of $\overline M_\p$ to define $x_\p$, as in \cite[\S 8.3.3]{Se12}.
This avoids the need to assume the conjectured \emph{semisimplicity of Frobenius}, which is known for abelian varieties but not in general.
\end{remark}

\begin{background}[Weil pairing]
If $A$ is an abelian variety over a field $k$ and $A^\vee$ is its dual abelian variety (see \cite[\S I.8]{Milne:AV}), then for each $n\ge 1$ prime to the characteristic of $k$, the \emph{Weil pairing} is a nondegenerate bilinear map
\[
A[n]\times A^\vee[n] \to \mu_n(\kbar)
\]
that commutes with the action of $\Gal(\kbar/k)$; here $\mu_n$ denotes the group of $n$th roots of unity (the algebraic group defined by $x^n=1$).
Letting $n$ vary over powers of a prime $\ell\ne{\rm char}(k)$ and taking inverse limits yields a bilinear map on the corresponding Tate modules:
\[
e_{\ell}\colon T_\ell\times T_\ell^\vee \to \mu_{\ell^\infty}(\kbar) \coloneqq \varprojlim_n \mu_{\ell^n}(\kbar).
\]
Given a \emph{polarization}, an isogeny $\phi\colon A\to A^\vee$, we can use it to define a bilinear pairing
\begin{align*}
e_{\ell}^\phi\colon T_\ell\times T_\ell&\to \mu_{\ell^\infty}(\kbar)\\
(x,y)&\mapsto e_\ell(x,\phi(y))
\end{align*}
that is also compatible with the action of $\Gal(\kbar/k)$.
One can always choose a polarization $\phi$ so that the pairing $e_{\ell}^\phi$ is nondegenerate and skew symmetric, meaning that $e_\ell^\phi(a,b)=e_\ell^\phi(b,a)^{-1}$ for all $a,b\in T_\ell$; see \cite[Prop.\ I.13.2]{Milne:AV}.
When~$A$ is the Jacobian of a curve it is naturally equipped with a \emph{principal polarization} $\phi$, an isomorphism $A\overset{\sim}{\rightarrow} A^\vee$, for which this automatically holds; in this situation it is common to simply identify $e_\ell$ with $e_\ell^\phi$ without mentioning $\phi$ explicitly.
\end{background}

We should note that our definition of the Sato-Tate group $\ST(A)$ required us to choose a prime $\ell$ and an embedding $\iota:\Q_\ell\to \C$.
Up to conjugacy in $\USp(2g)$ one expects the Sato-Tate group to be independent of these choices; this is known for $g\le 3$ (see \cite{BK15}), but open in general.
We shall nevertheless refer to $\ST(A)$ as ``the'' Sato-Tate group of $A$, with the understanding that we are fixing once and for all a prime~$\ell$ and an embedding  $\iota:\Q_\ell\to \C$ (note that these choices do not depend on $A$ or even its dimension $g$).

\subsection{The Sato-Tate conjecture for abelian varieties}
Having defined the Sato-Tate group of an abelian variety over a number field we can now state the Sato-Tate conjecture for abelian varieties.

\begin{conjecture}
Let $A$ be an abelian variety over a number field $K$, let $\ST(A)$ denote its Sato-Tate group, and let $(x_\p)$ be the sequence of conjugacy classes of normalized images of Frobenius elements in $\ST(A)$ at primes~$\p$ of good reduction for $A$, ordered by norm (break ties arbitrarily).
Then the sequence $(x_\p)$ is equidistributed (with respect to the pushforward of the Haar measure of $\ST(A)$ to its space of conjugacy classes).
\end{conjecture}

\subsection{The identity component of the Sato-Tate group}\label{sec:stid}

There are two algebraic groups that one can associate to an abelian variety $A$ over a number field $K$ that are closely related to its Sato--Tate group, the \emph{Mumford--Tate group} and the \emph{Hodge group}, both of which conjecturally determine the identity component of the Sato--Tate group (provably so whenever the Mumford--Tate conjecture is known, which includes all abelian varieties of dimension $g\le 3$, as shown in \cite{BK15}).
In order to define these groups we need to recall some facts about complex abelian varieties and their associated Hodge structures.

\begin{background}[complex abelian varieties]
Let $A$ be an abelian variety of dimension $g$ over $\C$.
Then $A(\C)$ is a connected compact Lie group and therefore isomorphic to a torus $V/\Lambda$, where $V\simeq \C^g$ is a complex vector space of dimension $g$ and $\Lambda\simeq \Z^{2g}$ is a full lattice in $V$ that we view as a free $\Z$-module;
one can obtain $\Lambda$ as the kernel of the exponential map $\exp\colon T_0(A(\C))\to A(\C)$, where $T_0(A(\C))$ denotes the tangent space at the identity.
While every complex abelian variety corresponds to a complex torus, the converse is true only when $g=1$.
The complex tori $X\coloneqq V/\Lambda$ that correspond to abelian varieties are those that admit a \emph{polarization} (or \emph{Riemann form}), 
a positive definite Hermitian form $H\colon V\times V\to\C$ with $\Im H(\Lambda,\Lambda)=\Z$ (here $\Im$ means imaginary part).
Given a polarization $H$ on $X$, the map $v\mapsto H(v,\cdot)$ defines an isogeny to the \emph{dual torus} $X^\vee\coloneqq V^*/\Lambda^*$, where
\[
V^*\coloneqq \{f\colon V\to \C:f(\alpha v)=\bar\alpha f(v)\text{ and } f(v_1+v_2)=f(v_1)+f(v_2)\},
\]
and $\Lambda^*\coloneqq \{f\in V^*:\Im f(\Lambda)\subseteq \Z\}$.
This isogeny is a polarization of $X$ as an abelian variety; conversely, any polarization on $A$ (one always exists) can be used to define a polarization on the complex torus $A(\C)$.
One can then show that the map $A\mapsto A(\C)$ defines an equivalence of categories between complex abelian varieties and polarizable complex tori. For more background on complex abelian varieties, see the overviews in \cite[\S 1]{Milne:AV} or \cite [\S 1]{Mu74}, or see \cite{BL04} for a comprehensive treatment.
\end{background}

Now let $A$ be an abelian variety over a number field $K$, fix an embedding $K\hookrightarrow \C$, and let $\C^g/\Lambda$ be the complex torus corresponding to $A(\C)$.
We may identify $\Lambda$ with the singular homology group $H_1(A(\C),\Z)$, and we similarly have $\Lambda_R\coloneqq \Lambda\otimes_\Z R \simeq H_1(A(\C),R)$ for any ring $R$.

The isomorphisms $A(\C)\simeq \C^g/\Lambda$ and $A(\C)\simeq\R^{2g}/\Lambda$ of complex and real Lie groups allow us to view
\[
\Lambda_\R \simeq H_1(A(\C),\R)
\]
as a real vector space of dimension $2g$ equipped with a \emph{complex structure}, by which we mean an $\R$-algebra homomorphism $h\colon \C\to \End(\Lambda_\R)$.
In the language of Hodge theory, this amounts to the statement that $(\Lambda,h)$ is an \emph{integral Hodge structure} (pure of weight~ $-1$).

We can also view $h$ as morphism of $\R$-algebraic groups $h\colon \SS\to \GL_{\Lambda_\R}$.
Here $\SS$ denotes the \emph{Deligne torus} (also known as the \emph{Serre torus}), obtained by viewing $\C^\times$ as an $\R$-algebraic group (this amounts to taking the restriction of scalars of $\Gm\coloneqq \GL_1$ from $\C$ to $\R$; see Exercise ~\ref{ex:DeligneTorus}).  The morphism $h$ can be defined over $\R$ because $\C^{g}/\Lambda$ is a polarizable torus, since it comes from an abelian variety (in general this need not hold).
The real Lie group $\SS(\R)\simeq \C^\times$ is generated by $\R^\times$ and $\U(1)=\{z\in\C^\times:z\bar z=1\}$, which intersect in $\{\pm 1\}$; taking Zariski closures yields $\R$-algebraic subgroups $\Gm$ and $\U_1$ of $\SS$ that intersect in~$\mu_2$.
Restricting $h$ to $\U_1\subseteq \SS$ yields a morphism $\U_1\to\GL_{\Lambda_\R}$ with the following property: the image of each $u\in \U_1(\R)= \U(1)$ has eigenvalues $u,u^{-1}$ with multiplicity $g$; see \cite[Prop.\ 17.1.1]{BL04}.
The image of such a map is known as a \emph{Hodge circle}.

The \emph{rational Hodge structure} $(\Lambda_\Q,h)$ is obtained by replacing $\Lambda$ with $\Lambda_\Q:=\Lambda\otimes_\Z\Q$ and can be used to define the Mumford-Tate group.

\begin{definition}
The \emph{Mumford--Tate group} $\MT(A)$ is the smallest $\Q$-algebraic group $G$ in $\GL_{\Lambda_\Q}$ for which $h(\SS)\subseteq G(\R)$; equivalently, it is the $\Q$-Zariski closure of $h(\SS(\R))$ in $\GL_{\Lambda_\R}$.
The \emph{Hodge group} $\Hg(A)$ is similarly defined as the $\Q$-Zariski closure of $h(\U(1))$ in $\GL_{\Lambda_\R}$.
\end{definition}

As defined above, the Mumford--Tate group $\MT(A)$ is a $\Q$-algebraic subgroup of $\GL_{2g}$.
But the complex torus $\C^{g}/\Lambda$ is polarizable, which means that we can put a symplectic form on $\Lambda_R$ that is compatible with~$h$, and this implies that in fact $\MT(A)$ is a $\Q$-algebraic subgroup of $\GSp_{2g}$.
Similarly, the Hodge group $\Hg(A)$ is a $\Q$-algebraic subgroup of $\Sp_{2g}$, and in fact $\Hg(A)=\MT(A)\cap \Sp_{2g}$; this is sometimes used as an alternative definition of $\Hg(A)$.
Much of the interest in the Hodge group arises from the fact that it gives us an isomorphism of $\Q$-algebras
\[
\End(A_\C)_\Q \simeq \End(\Lambda_\Q)^{\Hg(A)},
\]
where $\End(A_\C)_\Q\coloneqq \End(A_\C)\otimes_\Z\Q$ and $\Hg(A)$ acts on $\End(\Lambda_\Q)$ by conjugation; see \cite[Prop.\ 17.3.4]{BL04}.
To see why this isomorphism is useful, let us note one application.

\begin{theorem}\label{thm:cm}
For an abelian variety $A$ of dimension $g$ over a number field $K$, the Hodge group $\Hg(A)$ is commutative if and only if the endomorphism algebra $\End(A_{\Kbar})_\Q$ contains a commutative semisimple $\Q$-algebra of dimension $2g$.
\end{theorem}
\begin{proof}
See \cite[Prop.\ 17.3.5]{BL04}.
\end{proof}
\noindent
For $g=1$ the abelian varieties $A$ that satisfy the two equivalent properties of Theorem~\ref{thm:cm} are CM elliptic curves.  More generally, such abelian varieties are said to be of \emph{CM-type}.
For abelian varieties of general type one has the opposite extreme: $\End(A_{\Kbar})_\Q=\Q$ and $\Hg(A)=\Sp_{2g}$; see \cite[Prop.\ 17.4.2]{BL04}.

In the previous section we defined two $\Q_\ell$-algebraic groups $G_\ell^{\rm zar}\subseteq \GSp_{2g}$ and $G_\ell^{1,\rm zar}\subseteq \Sp_{2g}$ associated to $A$.
It is reasonable to ask how they are related to the $\Q$-algebraic groups $\MT(A)$ and $\Hg(A)$.
Unlike the groups $G_\ell^{\rm zar}$ and $G_\ell^{1,\rm zar}$, the algebraic groups $\MT(A)$ and $\Hg(A)$ are necessarily connected (by construction).\footnote{This is true more generally for all motives of odd weight.  For motives of even weight the situation is more delicate; complications arise from the fact that we are then working with orthogonal groups rather than symplectic groups; see \cite{BK15,BK16}.}
Deligne proved that the identity component of $G_\ell^{\rm zar}$ is always a subgroup of $\MT(A)\otimes_\Q \Q_\ell$, equivalently, that the identity component of $G_\ell^{1,\rm zar}$ is a subgroup of $\Hg(A)\otimes_\Q \Q_\ell$); see \cite{Del82}.
It is conjectured that these inclusions are in fact equalities.

\begin{conjecture}[\textsc{Mumford--Tate Conjecture}]
The identity component of $G_\ell^{\rm zar}$ is equal to $\MT(A)\otimes_\Q \Q_\ell$; equivalently, the identity component of $G_\ell^{1,\rm zar}$ is equal to $\Hg(A)\otimes_\Q \Q_\ell$.
\end{conjecture}

This conjecture is known to hold for abelian varieties of dimension $g\le 3$; see \cite[Th. 6.11]{BK15} where it is shown that this follows from \cite{MZ99}.
When it holds, the Mumford--Tate group (and the Hodge group) uniquely determines the identity component of the Sato--Tate group, up to conjugation in $\USp(2g)$; see \cite[Lemma 2.8]{FKRS12}.
Neither the Mumford--Tate group nor the Hodge group tell us anything about the component groups of $G_\ell^{\rm zar}$, $G_\ell^{1,\rm zar}$, $\ST(A)$ (the three are isomorphic; see \cite[\S 8.3.4]{Se12}), but there is a closely related $\Q$-algebraic group that conjecturally does.

\begin{conjecture}[\textsc{Algebraic Sato--Tate Conjecture}]
There exists a $\Q$-algebraic subgroup $\AST(A)$ of $\Sp_{2g}$ such that $G_\ell^{1,\rm zar}=\AST(A)\otimes_\Q \Q_\ell$.
\end{conjecture}

Banaszak and Kedlaya \cite{BK15} have shown that this conjecture holds for $g\le 3$ via an explicit description of $\AST(A)$ using \emph{twisted Lefschetz groups}.

\subsection{The component group of the Sato-Tate group}

We have seen that the Mumford--Tate group conjecturally determines the identity component $\ST(A)^0$ of the Sato--Tate group $\ST(A)$ of an abelian variety~$A$ over a number field $K$ (provably so in dimension $g\le 3$).
The identity component $\ST(A)^0$ is a normal finite index subgroup of $\ST(A)$, and we now want to consider the component group $\ST(A)/\ST(A)^0$.
As above, for any field extension $L/K$, we use $A_L$ to denote the base change of $A$ to $L$.

\begin{theorem}\label{thm:component}
Let $A$ be an abelian variety over a number field $K$.
There is a unique finite Galois extension $L/K$ with the property that $\ST(A_L)$ is connected and $\Gal(L/K)\simeq \ST(A)/\ST(A)^0$.
The extension $L/K$ is unramified outside the primes of bad reduction for $A$, and for every subextension $F/K$ of $L/K$ we have $\Gal(L/F)\simeq \ST(A_F)/\ST(A_F)^0$.
\end{theorem}
\begin{proof}
As explained in \cite[\S 8.3.4]{Se12}, the component groups of $G_\ell^{\rm zar}$ and $\ST(A)$ are isomorphic.
Let $\Gamma$ be the Galois group of the maximal subextension $K_{S_\ell}$ of $\Gal(\Kbar/K)$ that is unramified away from the set $S_\ell$ consisting of the primes of bad reduction for $A$ and the primes of $K$ lying above $\ell$.
The $\ell$-adic Galois representation $\rho_{A,\ell}\colon \Gal(\Kbar/K)\to \Aut(V_\ell)$ induces a continuous surjective homomorphism
\[
\Gamma\to G_\ell^{\rm zar}/(G_\ell^{\rm zar})^0,
\]
whose kernel is a normal open subgroup $\Gamma_0$ of $\Gamma$.
The corresponding fixed field $L$ is a finite Galois extension of $K$, and it is the minimal Galois extension of $K$ for which $\ST(A_L)$ is connected.  It is clearly uniquely determined and unramified outside $S_\ell$, and we have isomorphisms
\[
\Gal(L/K)\simeq \Gamma/\Gamma_0\simeq G_\ell^{\rm zar}/(G_\ell^{\rm zar})^0\simeq \ST(A)/\ST(A)^0.
\]
As shown by Serre \cite{Se91}, the component group of $G_\ell^{\rm zar}$, and therefore of $\ST(A)$, is independent of $\ell$, and the above argument applies to any choice of $\ell$. Thus $L/K$ can be ramified only at primes of bad reduction for~$A$.
For any subextension $F/K$ of $L/K$, replacing~$A$ by $A_F$ in the argument above yields the same field~$L$, with $\Gal(L/F)\simeq \ST(A_F)/\ST(A_F)^0$.
\end{proof}

\subsection{Exercises}

\begin{exercise}\label{ex:gsp}
Let $A$ be an abelian variety of dimension $g$ over a number field $K$.
Show that one can choose a basis for $V_\ell=T_\ell\otimes_\Z\Q$ so that the matrix $M$ describing the action of any $\sigma\in \Gal(\Kbar/K)$ on $V_\ell$ satisfies $M^t\Omega M = \lambda\Omega$ for some $\lambda\in \Q_\ell^\times$, where $\Omega\coloneqq \bigl(\begin{smallmatrix}0&-I\\I&0\end{smallmatrix}\bigr)$.
Conclude that the image of the corresponding Galois representation lies in $\GSp_{2g}(\Q_\ell)$ and describe the map $\Gal(\Kbar/K)\to \Q_\ell^\times$ induced by the similitude character $\lambda$.
\end{exercise}

\begin{exercise}\label{ex:DeligneTorus}
Define the Deligne torus $\SS$ as an $\R$-algebraic group in $\A^4$ (give equations that define it as an affine variety and polynomial maps for the group operations), and then express the $\R$-algebraic groups $\Gm$ and $\U_1$ as subgroups of $\SS$ that intersect in $\mu_2$.
Prove that $\SS(\R)$ and $\C^\times$ are isomorphic as real Lie groups (give explicit maps in both directions).
\end{exercise}

\begin{exercise}\label{ex:WeilRestriction}
Let $L/K$ be a finite separable extension of degree $d$, with $L=K(\alpha)$.
Given an affine $L$-variety $Y$ defined by polynomials $P_k\in L[y_1,\ldots,y_n]$, we can construct an affine $K$-variety $\Res_{L/K}(Y)$ by writing each $y_i=\sum_{j=0}^{d-1} x_{ij}\alpha^j$ in terms of the $K$-basis $\{1,\alpha,\ldots,\alpha^{d-1}\}$ for $L$ and using the minimal polynomial of $\alpha$ to replace each $P_k(y_1,\ldots,y_n)$ by a polynomial in $K[x_{11},\ldots,x_{1d},\ldots,x_{n1}\ldots,x_{nd}]$.
The $K$-variety $\Res_{L/K}(Y)$ is the \emph{Weil restriction} (or \emph{restriction of scalars}) of $Y$.
Prove that the $\R$-algebraic group $\SS$ (the Deligne torus) is the Weil restriction of the $\C$-algebraic group $\Gm$, that is, $\SS = \Res_{\C/\R}(\Gm)$.
\end{exercise}

\section{Sato--Tate axioms and Galois endomorphism types}\label{lec:STaxioms}

In this section we present the Sato-Tate axioms and consider the problem of classifying Sato-Tate groups of abelian varieties of a given dimension $g$.  We then compute trace moment sequences of all connected Sato-Tate groups of abelian varieties of dimension $g\le 3$ and present formulas for the trace moment sequence of $\USp(2g)$ (the generic case) that apply to all $g$,

\subsection{Sato--Tate axioms}

In \cite[\S 8.2]{Se12} Serre gives a set of axioms that any Sato--Tate group is expected to satisfy.
Serre considers Sato--Tate groups in a more general context than we do here, so we will state the axioms as they apply to Sato--Tate groups of abelian varieties.
As in \S\ref{sec:stid}, for a Lie group $G$ we define a \emph{Hodge circle} to be a subgroup $H$ of $G$ that is the image of a continuous homomorphism $\theta\colon \U(1)\to G^0$ whose elements $\theta(u)$ have eigenvalues $u$ and $u^{-1}$ with multiplicity $g$ (note that $H$ necessarily lies in the identity component $G^0$ of $G$).

\begin{definition}\label{def:STaxioms}
A group $G$ satisfies the \emph{Sato--Tate axioms} (for abelian varieties of dimension $g\ge 1$) if and only if the following hold:
\begin{enumerate}
\setlength{\itemsep}{2pt}
\item[(ST1)] (Lie condition) $G$ is a closed subgroup of $\USp(2g)$.
\item[(ST2)] (Hodge condition) The Hodge circles in $G$ generate a dense non-trivial subgroup of $G^0$.\footnote{The statement of (ST2) in \cite{FKRS12} inadvertently omits the requirement that the Hodge circles generate a dense subgroup.}
\item[(ST3)] (rationality condition) For each component $H$ of $G$ and irreducible character $\chi$ of $\GL_{2g}(\C)$, we have $\int_H\chi\mu\in \Z$, where $\mu$ is the Haar measure on $G$ normalized so that $\mu(\ifunc_H)=1$.
\end{enumerate}
\end{definition}

\begin{remark}\label{rem:STaxgen}
Definition \ref{def:STaxioms} generalizes easily to self-dual motives with rational coefficients.
Given an integer weight $w\ge 0$ and Hodge numbers $h^{p,q}\in \Z_{\ge 0}$ indexed by $p,q\in\Z_{\ge 0}$ with $p+q=w$ such that $h^{p,q}=h^{q,p}$ when $w$ is odd, let $d\coloneqq\sum h^{p,q}$.
For abelian varieties we have $w=1$ and $h^{1,0}=h^{0,1}=g$.
In axiom (ST1) we require $G$ to be a closed subgroup of $\USp(d)$ (resp. $\O(d)$) when $w$ is odd (resp. even), and in axiom (ST2) we require elements $\theta(u)$ of a Hodge circle to have eigenvalues $u^{p-q}$ with multiplicity $h^{p,q}$; axiom (ST3) is unchanged.
\end{remark}

Axiom (ST1) implies that $G$ is a compact Lie group, and (ST2) rules out finite groups, since~$G$ must contain at least one Hodge circle and therefore contains a subgroup isomorphic to $\U(1)$.
When~$G$ is connected, (ST3) holds automatically and only (ST1) and (ST2) need to be checked; this is an easy application of representation theory, see \cite[Prop.\ 2]{KS09}.
Axiom (ST3) plays no role when $g=1$ (see the proof of Proposition~\ref{prop:staxg1} below), but for $g>1$ it is crucial.
When $g=2$, for example, for every integer $n\ge 1$ we can diagonally embed $\U(1)\times \U(1)[n]$ in $\USp(4)$ to get infinitely many non-conjugate closed groups $G\subseteq \USp(4)$ whose identity component is a Hodge circle.
All of these groups satisfy (ST1) and (ST2),  but only finitely many satisfy (ST3).
Indeed, if we take $\chi$ and let $C$ be a component on which the projection to $\U(1)[n]$ has order $n$, we have
\[
\int_C\chi\mu=\zeta_n+\bar\zeta_n\in \Z
\]
only for $n\in \{2,3,4,6\}$.
More generally, we have the following theorem.

\begin{theorem}\label{thm:staxfinite}
Up to conjugacy, for any fixed dimension $g\ge 1$ the number of subgroups of $\USp(2g)$ that satisfy the Sato--Tate axioms is finite.
\end{theorem}
\begin{proof}
See \cite[Rem. 3.3]{FKRS12}
\end{proof}

Theorem~\ref{thm:staxfinite} motivates the following \emph{classification problem}: given an integer $g\ge 1$, determine the subgroups of $\USp(2g)$ that satisfy the Sato--Tate axioms.
The case $g=1$ is easy.

\begin{proposition}\label{prop:staxg1}
For $g=1$ the three groups $\U(1)$, $N(\U(1)$ and $\SU(2)$ listed in Theorem~\ref{thm:STg1} are the only groups that satisfy the Sato--Tate axioms (up to conjugacy).
\end{proposition}
\begin{proof}
Suppose $G$ satisfies the Sato--Tate axioms.
Then $G^0$ contains a conjugate of $\U(1)$ embedded in $\USp(2)$ via $u\mapsto\smallmat{u}{0}{0}{\bar u}$, as in Theorem~\ref{thm:STg1}, and it must be a compact connected Lie group.
The only nontrivial compact connected Lie groups in $\USp(2)=\SU(2)$ are $\U(1)$ and $\SU(2)$ itself (this follows from the classification of compact connected Lie groups but is easy to see directly).
Thus either $G^0=\SU(2)$, in which case $G=\SU(2)$, or $G^0$ is conjugate to $\U(1)$ and must be a normal subgroup of $G$ (the identity component of a compact Lie group is always a normal subgroup of finite index).
The group $\U(1)$ has index~2 in its normalizer, so $\U(1)$ and $N(\U(1))$ are the only possibilities for $G$ when $G^0=\U(1)$.
\end{proof}

\begin{corollary}\label{cor:staxg1}
For $g=1$ a group $G$ satisfies the Sato--Tate axioms if and only if it is the Sato--Tate group of an elliptic curve over a number field.
\end{corollary}

The classification problem for $g=2$ is more difficult, but it has been solved.

\begin{theorem}\label{thm:staxg2}
Up to conjugacy in $\USp(4)$ there are $55$ groups that satisfy the Sato--Tate axioms for $g=2$.
Of these $55$, the following $6$ are connected:
\[
\U(1)_2,\qquad \SU(2)_2,\qquad \U(1)\times\U(1), \qquad \U(1)\times \SU(2),\qquad \SU(2)\times \SU(2),\qquad \USp(4),
\]
were $\U(1)_2$ denotes $\U(1)=\bigl\{\smallmat{u}{0}{0}{\bar u}:u\in \C^\times\bigr\}$ diagonally embedded in $\USp(4)$, and similarly for $\SU(2)_2$.
\end{theorem}
\begin{proof}
See \cite[Thm.\ 3.4]{FKRS12}, which gives an explicit description of the 55 groups.
\end{proof}

\begin{remark}
Those familiar with the classification of connected compact Lie groups may notice that the group $\U(2)$, which can be embedded in $\USp(4)$, is missing from Theorem~\ref{thm:staxg2}.
This is because it fails to satisfy the Hodge condition (ST2); it contains subgroups isomorphic to $\U(1)$, but there is no way to embed $\U(1)\hookrightarrow \U(2)\hookrightarrow\USp(4)$ and get eigenvalues $u$ and $u^{-1}$ with multiplicity $2$; see \cite[Rem. 2.3]{FKS16}.
However, for motives of weight $3$ and Hodge numbers $h^{3,0}=h^{2,1}=h^{1,2}=h^{0,3}=1$ the modified Hodge condition noted in Remark~\ref{rem:STaxgen} is satisfied by a subgroup of $\USp(4)$ isomorphic to $\U(2)$; see \cite{FKS16} for details, including two examples of weight 3 motives with Sato-Tate group $\U(2)$.
\end{remark}

Corollary~\ref{cor:staxg1} does not hold for $g=2$.

\begin{theorem}\label{thm:stg2}
Of the $55$ groups appearing in Theorem~\ref{thm:staxg2}, only $52$ arise as the Sato--Tate group of an abelian surface over a number field.  Of these, $34$ arise for abelian surfaces defined over $\Q$.
\end{theorem}
\begin{proof}
See \cite[Thm.\ 1.5]{FKRS12}.
\end{proof}

The three subgroups of $\USp(4)$ that satisfy the Sato--Tate axioms but are not the Sato--Tate group of any abelian surface over a number field are the normalizer of $\U(1)\times\U(1)$ in $\USp(4)$, whose component group is the dihedral group of order 8, and two of its subgroups, one of index 2 and one of index 4.
The proof that these three groups do not occur is obtained by first establishing a bijection between \emph{Galois endomorphism types} (see Definition~\ref{def:galtype} below) and Sato--Tate groups, and then showing that there are only 52 Galois endomorphism types of abelian surfaces.
Explicit examples of genus 2 curves whose Jacobians realize these 52 possibilities can be found in \cite[Table 11]{FKRS12}, and animated histograms of their Sato--Tate distributions are available at
\begin{center}
\url{http://math.mit.edu/~drew/g2SatoTateDistributions.html}
\end{center}

The classification problem for $g=3$ remains open, but the connected cases have been determined (see Table~\ref{tab:st0g3} in the next section).
Before leaving our discussion of the Sato--Tate axioms, it is reasonable to ask whether Sato--Tate groups necessarily satisfy them.
Of course we expect this to be the case, but it is difficult to prove in general.
However, it can be proved to hold in all cases where the Mumford--Tate conjecture is known, including all cases with $g\le 3$.

\begin{proposition}
Let $A$ be an abelian variety of dimension $g$ over a number field $K$ for which the Mumford--Tate conjecture holds.
Then $\ST(A)$ satisfies the Sato--Tate axioms.
\end{proposition}
\begin{proof}
See \cite[Prop.\ 3.2]{FKRS12}.
\end{proof}

\subsection{Galois endomorphism types}


We will work in the abstract category $\mathcal C$ whose objects are pairs $(G,E)$ of a finite group $G$ and an $\R$-algebra $E$ equipped with an $\R$-linear action of $G$, and whose morphisms $\Phi\colon (G,E)\to (G',E')$ are pairs $(\phi_G,\phi_E)$, where $\phi_G\colon G\to G'$ is a morphism of groups, and $\phi_E\colon E\to E'$ is an equivariant morphism of $\R$-algebras, meaning that
\begin{equation}\label{eq:equiv}
\phi_E(e^g)=\phi_E(e)^{\phi_G(g)}\qquad\text{for all }g\in G \text{ and } e\in E.
\end{equation}
To each abelian variety $A/K$ we now associate an isomorphism class $[G,E]$ in $\mathcal C$ as follows.
The minimal extension $L/K$ for which $\End(A_L)=\End(A_{\Kbar})$ is a finite Galois extension of $K$; we shall take $G$ to be $\Gal(L/K)$ and $E$ to be the real endomorphism algebra $\End(A_L)_\R\coloneqq \End(A_L)\otimes_\Z\R$.
The Galois group $\Gal(L/K)$ acts on $\End(A_L)$ via its action on the coefficients of the rational maps defining each element of $\End(A_K)$; this induces an $\R$-linear action of $\Gal(L/K)$ on $\End(A_L)_\R$ via composition with the natural map $\End(A_L)\to \End(A_L)_\R$.
The pair $(\Gal(L/K),\End(A_L)_\R)$ is thus an object of $\mathcal C$.

\begin{definition}\label{def:galtype}
The \emph{Galois endomorphism type} $\GT(A)$ of an abelian variety $A/K$ is the isomorphism class of the pair $(\Gal(L/K),\, \End(A_L)_\R)$ in the category $\mathcal C$, where $L$ is the minimal extension of $K$ for which $\End(A_L)=\End(A_{\Kbar})$.
\end{definition}

\begin{example}\label{ex:g1end}
Let $E$ be an elliptic curve over a number field $K$.
If $E$ does not have CM, or if it has CM defined over $K$, then its endomorphisms are all defined over $L=K$; otherwise, its endomorphisms are all defined over its CM field $L$, an imaginary quadratic extension of $K$.
The real endomorphism algebra $\End(E_L)_\R$ is isomorphic to $\R$ when $E$ does not have CM, and isomorphic to $\C$ when $E$ does have CM.
We therefore have
\[
\GT(E)=\begin{cases}
[\cyc{1},\C]&\text{if $E$ has CM defined over $K$}\\
[\cyc{2},\C]&\text{if $E$ has CM not defined over $K$}\\
[\cyc{1},\R]&\text{if $E$ does not have CM}
\end{cases}
\]
Here $\cyc{n}$ denotes the cyclic group of order $n$; in the case $[\cyc{2},\C]$ the action of $\cyc{2}$ on $\C$ corresponds to complex conjugation.
\end{example}

The three Galois endomorphism types listed in Example~\ref{ex:g1end} correspond to the three Sato-Tate groups listed in Theorem~\ref{thm:STg1}.
Under this correspondence, the real endomorphism algebra $\End(E_L)_\R$ determines the identity component $\ST(E)^0$ (up to conjugacy), and the Galois group $\Gal(L/K)$ is isomorphic to the component group $\ST(E)/\ST(E)^0$.
Moreover, the field~$L$ is precisely the field $L$ given by Theorem~\ref{thm:component}.

\begin{theorem}\label{thm:gt}
Let $A$ be an abelian variety $A$ of dimension $g\le 3$ defined over a number field $K$ and let~$L$ be the minimal field for which $\End(A_L)=\End(A_{\Kbar})$.
The conjugacy class of the Sato-Tate group $\ST(A)$ determines the Galois endomorphism type $\GT(A)$; moreover, the conjugacy class of the identity component $\ST(A)^0$ determines the isomorphism class of $\End(A_L)_\R$ and $\ST(A)/\ST(A)^0\simeq \Gal(L/K)$.
For $g\le 2$ the converse holds: the Galois endomorphism type $\GT(A)$ determines the Sato--Tate group $\ST(A)$ up to conjugacy.
\end{theorem}
\begin{proof}
See Proposition 2.19 and Theorem 1.4 in \cite{FKRS12}.
\end{proof}

It is expected that in fact the Sato--Tate group always determines the Galois endomorphism type, and that the converse holds for $g\le 3$.
For $g=3$ we at least know that the real endomorphism algebra $\End(A_L)_\R$ determines the identity component $\ST(A)^0$ and that $\Gal(L/K)\simeq \ST(A)/\ST(A)^0$.
At first glance it might seem that this should determine $\ST(A)$, but it does not, even when $g=2$.
One needs to also understand how $\Gal(L/K)$ acts on $\End(A_L)_\R$ and relate this to the Sato-Tate group $\ST(A)$.
In \cite{FKRS12} this is accomplished for $g=2$ by looking at the lattice of $\R$-subalgebras of $\End(A_L)_\R$ fixed by subgroups of $\Gal(L/K)$ and showing that this is enough to uniquely determine $\ST(A)$; see \cite[Thm.\ 4.3]{FKRS12}.
To apply the same approach when $g=3$ we need a more detailed classification of the possible Galois endomorphism types and Sato--Tate groups for $g=3$ than is currently available.

For $g=4$ the Galois endomorphism type does not always determine the Sato--Tate group.
This is due to an exceptional counterexample constructed by Mumford in \cite{Mu69}, in which he proves the existence of an abelian four-fold $A$ for which $\End(A_{\Kbar})=\Z$ but $\MT(A)\ne \GSp_8$.
The fact that $\MT(A)$ is properly contained in $\GSp_8$ implies that $\ST(A)$ must be properly contained in $\USp(8)$ (this does not depend on the Mumford--Tate conjecture, here we are only using the inclusion proved by Deligne).
On the other hand, for an abelian variety of general type one has $\End(A_{\Kbar})=\Z$ and $\ST(A)=\USp(2g)$; see \cite{Ha11, Za00} for an explicit criterion that applies to almost all Jacobians of hyperelliptic curves.

For $g>4$ one can construct exceptional examples as a product of an abelian variety with one of Mumford's exceptional four-folds, so in general the Galois endomorphism type cannot determine the Sato--Tate group for any $g\ge 4$.
However, such examples will not be simple and will have $\End(A)\ne\Z$.
In \cite{Se86b} Serre proves an analog of his open image theorem for elliptic curves that applies to abelian varieties of dimension $g=2,6$ and $g$ odd.
For these values of $g$, if $\End(A_{\Kbar})=\Z$ then $\ST(A)=\USp(2g)$ and no direct analog of Mumford's construction exists.

\begin{remark}
For $g\le 3$, the field $L$ in Theorem~\ref{thm:component} (the minimal $L$ for which $\ST(A_L)$ is connected) is the same as the field $L$ in Theorem~\ref{thm:gt} (the minimal $L$ for which $\End(A_L)=\End(A_{\Kbar})$).
In any case, the former always contains the latter: if $\ST(A_L)$ is connected then we necessarily have $\End(A_{\Kbar})=\End(A_L)$.
This can be seen as a consequence of Bogomolov's theorem \cite{Bo80}, which states that $G_\ell$ is open in $G_\ell^{\rm zar}(\Q_\ell)$,  and Faltings` theorem \cite{Falt83} that $\End(A)_{\Q_\ell}\simeq  \End(V_\ell(A))^{G_\ell}$.
If $\ST(A)$ (and therefore $G_\ell^{\rm zar}$) is connected, then $\End(A)$ is invariant under base change (now apply this to $A=A_L$).
\end{remark}

Tables~\ref{tab:st0g2} and~\ref{tab:st0g3} below list the real endomorphism algebras and corresponding identity components of Sato-Tate groups that arise in dimensions $g=2,3$.
A complete list of the 52 Galois endomorphism types and corresponding Sato-Tate groups for $g=2$ can be found in \cite[Thm.\ 4.3]{FKRS12} and \cite[Table 9]{FKRS12}. 

\clearpage
\begin{table}
\setlength{\extrarowheight}{1.7pt}
\begin{tabular}{|l|l|l|}\hline
\textbf{geometric type of abelian surface} & ${\End(A_{\Kbar})_\R}$ & ${\ST(A)^0}$\\\hline
square of CM elliptic curve & ${\rm M}_2(\C)$ & $\U(1)_2$\\\hline
{QM abelian surface} & ${\rm M}_2(\R)$ & $\SU(2)_2$\\
square of non-CM elliptic curve &  & \\\hline
{CM abelian surface} & $\C\times\C$& $\U(1)\times\U(1)$\\
product of CM elliptic curves & &\\\hline
product of CM and non-CM elliptic curves & $\C\times\R$ & $\U(1)\times\SU(2)$\\\hline
{RM abelian surface} & $\R\times\R$ & $\SU(2)\times\SU(2)$\\
product of non-CM elliptic curves & & \\\hline
{abelian surface of general type} & $\R$ & $\USp(4)$\\\hline
\end{tabular}
\bigskip

\caption{Real endomorphism algebras and Sato--Tate identity components for abelian surfaces}\label{tab:st0g2}
\bigskip
\medskip
\end{table}
\vspace{-40pt}

\begin{table}
\setlength{\extrarowheight}{1.7pt}
\begin{tabular}{|l|l|l|}\hline
\textbf{geometric type of abelian three-fold} & ${\End(A_K)_\R}$ & ${\ST(A)^0}$\\\hline
cube of a CM EC & ${\rm M}_3(\C)$ & $\U(1)_3$\\\hline
cube of a non-CM EC & ${\rm M}_3(\R)$ & $\SU(2)_3$\\\hline
product of CM EC and square of CM EC & $\C\times{\rm M}_2(\C)$ & $\U(1)\times\U(1)_2$\\\hline
product of CM EC and QM abelian surface & $\C\times{\rm M}_2(\R)$ & $\U(1)\times\SU(2)_2$\\
product of CM EC and square of non-CM EC & &\\\hline
product of non-CM EC and square of CM EC & $\R\times{\rm M}_2(\C)$ & $\SU(2)\times\U(1)_2$\\\hline
product of non-CM EC and QM abelian surface & $\R\times{\rm M}_2(\R)$ & $\SU(2)\times\SU(2)_2$\\
product of non-CM EC and square of non-CM EC & &\\\hline
{CM abelian threefold} & $\C\times\C\times\C$ & $\U(1)\times\U(1)\times\U(1)$\\
product of CM EC and CM abelian surface & &\\
product of three CM ECs & &\\\hline
product of non-CM EC and CM abelian surface & $\C\times\C\times\R$ & $\U(1)\times\U(1)\times\SU(2)$\\
product of non-CM EC and two CM ECs & &\\\hline
product of CM EC and RM abelian surface & $\C\times\R\times\R$&$\U(1)\times\SU(2)\times\SU(2)$\\
product of CM EC and two non-CM ECs & &\\\hline
{RM abelian threefold} & $\R\times\R\times\R$ & $\SU(2)\times\SU(2)\times\SU(2)$\\
product of non-CM EC and RM abelian surface & &\\
product of 3 non-CM ECs & &\\\hline
product of CM EC and abelian surface & $\C\times\R$ & $\U(1)\times\USp(4)$\\\hline
product of non-CM EC and abelian surface & $\R\times\R$ & $\SU(2)\times\USp(4)$\\\hline
{quadratic CM abelian threefold} & $\C$ & $\U(3)$\\\hline
{generic abelian threefold} & $\R$ & $\USp(6)$\\\hline
\end{tabular}
\bigskip

\caption{Real endomorphism algebras and Sato--Tate identity components for abelian threefolds}\label{tab:st0g3}
\end{table}
\clearpage

As can be seen in the two tables above, the Sato--Tate group is in some respects a rather coarse invariant; for example, it cannot distinguish a product of non-CM elliptic curves from a geometrically simple abelian surface with real multiplication (RM).
On the other hand, the Haar measures of the 52 Sato--Tate groups of abelian surfaces over number fields all give rise to distinct distributions of characteristic polynomials, which, under the Sato--Tate conjecture, match the distribution of normalized $L$-polynomials, and there are some rather fine distinctions among these distributions that the Sato--Tate group detects.
For example, there are only 37 distinct trace distributions among the 52 groups, one needs to look at both the linear and quadratic coefficients of the characteristic polynomials in order to distinguish them.

It is possible for two non-conjugate Sato--Tate groups to be isomorphic as abstract groups yet give rise to distinct trace distributions.
For example, the connected Sato-Tate groups $\SU(2)\times \U(1)_2$ and $\U(1)\times \SU(2)_2$ that appear in Table~\ref{tab:st0g3} are both abstractly isomorphic to the real Lie group $\U(1)\times \SU(2)$, but these two embeddings of $\U(1)\times \SU(2)$  in $\USp(6)$ have different trace distributions.

As shown by the example below, this phenomenon can also occur for disconnected Sato-Tate groups with the same identity component.

\begin{example}
Consider the hyperelliptic curves
\begin{align*}
C_1\colon y^2&=x^6 + 3x^5 + 15x^4 - 20x^3 + 60x^2 - 60x + 28,\\
C_2\colon y^2&=x^6 + 6x^5 - 15x^4 + 20x^3 - 15x^2 + 6x - 1,
\end{align*}
and let $A_1\coloneqq \Jac(C_1)$ and $A_2\coloneqq\Jac(C_2)$ denote their Jacobians.
Over $\Qbar$ both $A_1$ and $A_2$ are isogenous to the square of the elliptic curve $y^2=x^3+1$, which has CM by $\Q(\sqrt{-3})$.
We necessarily have $\ST(A_1)^0=\ST(A_2)^0=\U(1)_2$, and the component groups are both isomorphic to the dihedral group of order 12.
However, their Sato--Tate groups are different: in terms of the labels used in \cite{FKRS12}, we have $\ST(A_1)=D_{6,1}$, while $\ST(A_2)=D_{6,2}$ (see \cite[\S 3.4]{FKRS12} for explicit descriptions of these groups in terms of generators), and their normalized trace distributions are quite different.
For $C_1$ the density of zero traces is $3/4$, whereas for $C_2$ it is $7/12$ (these ratios represent the proportion of Sato--Tate group components on which the trace is identically zero), and their normalized trace moment sequences are $(1,0,1,0,9,0,110,0,1505,0,21546,\ldots)$ and $(1,0,2,0,18,0,200,0,2450,0,31752,\ldots)$, respectively.
The Sato-Tate conjecture for these two curves was proved in \cite{FS14}, so this difference in Sato-Tate groups provably impacts the normalized trace distributions of $A_1$ and $A_2$.
\end{example}

\subsection{Sato--Tate measures}

Once we know the Sato--Tate group $\ST(A)$ of an abelian variety $A$, we are in a position to compute various statistic related to the distribution of its conjugacy classes, such as the moments of characteristic polynomial coefficients (or any other conjugacy class invariant).
We can then test the Sato--Tate conjecture by comparing these to corresponding statistics obtained by computing normalized $L$-polynomials $\bar L_\p(T)$ for all primes $\p$ of good reduction for $A$ up to some norm bound $B$.

The first step is to determine the Haar measure on $\ST(A)^0$.
For $g=1$ there are only two possibilities: either $\ST(A)^0=\U(1)$ or $\ST(A)^0=\SU(2)$, where, as usual we embed $\U(1)$ in $\SU(2)$ via $u\mapsto \smallmat{u}{0}{0}{\bar u}$.
In terms of the eigenangle $\theta$, the pushforward measure on $\conj(\ST(A)^0)$ is one of
\begin{align*}
\mu_{\U(1)} &\coloneqq \tfrac{1}{\pi}d\theta,\\
\mu_{\SU(2)} &\coloneqq \tfrac{2}{\pi}\sin^2\theta\,d\theta,
\end{align*}
with $0\le \theta\le \pi$.
This also addresses two of the possibilities for $\ST(A)^0$ that arise when $g=2$, the groups $\U(1)_2$ and $\SU(1)_2$ listed in the first two rows of Table~\ref{tab:st0g2}; these denote two identical copies of $\U(1)$ and $\SU(2)$ diagonally embedded in $\USp(4)$.
When expressed in terms of the eigenangle $\theta$, the measure $\mu_{\U(1)_2}$ is exactly the same as $\mu_{\U(1)}$ (and similarly for $\mu_{\SU(2)_2}$), but note that we will get a different distribution on characteristic polynomials (which now have degree 4 rather than degree 2), because each eigenvalue now occurs with multiplicity $2$; in particular, the trace becomes $4\cos\theta$ rather than $2\cos\theta$.

For the groups $\ST(A)^0$ that appear in the next three rows of Table~\ref{tab:st0g2}, the measure on $\conj(\ST(A)^0)$ is a product of measures that we already know:
\begin{align*}
\mu_{\U(1)\times\U(1)} &\coloneqq \tfrac{1}{\pi^2}d\theta_1\, d\theta_2,\\
\mu_{\U(1)\times\SU(2)} &\coloneqq \tfrac{2}{\pi^2}\sin^2\theta_2\, d\theta_1 \,d\theta_2,\\
\mu_{\SU(2)\times\SU(2)} &\coloneqq \tfrac{4}{\pi^2}\sin^2\theta_1\sin^2\theta_2\, d\theta_1 \,d\theta_2.
\end{align*}
To obtain the measure for the generic case $\ST(A)=\ST(A)^0=\USp(4)$, we use the Weyl integration formula for $\USp(2g)$ (which includes the case $\USp(2)=\SU(2)$ that we already know):
\begin{equation}\label{haarUSp}
\mu_{\USp(2g)}\coloneqq\frac{1}{g!}\left(\prod_{1\le j<k\le g}\left(2\cos\theta_j-2\cos\theta_k\right)^2\right)\prod_{1\le j\le g}\left(\tfrac{2}{\pi}\sin^2\theta_j\, d\theta_j\right),
\end{equation}
with $0\le \theta_j\le \pi$, see \cite[Thm.\ 7.8B]{Weyl46} or \cite[\S 5.0.4]{KaSa99}.

This covers all the Sato-Tate groups listed in Table~\ref{tab:st0g2} for $g=2$.  By taking appropriate products of measures we know and applying the Weyl integration formula with $g=3$, we obtain all the $g=3$ cases listed in Table~\ref{tab:st0g3} except for $\U(3)$, where we need the Weyl integration formula for $\U(g)$:
\begin{equation}\label{haarU}
\mu_{\U(g)}\coloneqq \frac{1}{g!}\left(\prod_{1\le j<k\le g}\left|e^{i\theta_j}-e^{i\theta_k}\right|\right)\prod_{1\le j\le g} \tfrac{1}{2\pi}d\theta_j,
\end{equation}
with $0\le \theta_j\le 2\pi$ (note the $2\pi$); see \cite[Thm.\ 7.4B]{Weyl46} or \cite[\S 5.0.3]{KaSa99}.

With the measure $\mu_{\ST(A)^0}$ in hand, for any continuous class function $f$ on $\ST(A)$, we can compute
\[
\mu_{\ST(A)}(f)\coloneqq \int_{\ST(A)} f(x)\mu_{\ST(A)}(x) = \sum_x \int_{\ST(A)^0} f(xy)\mu_{\ST(A)^0}(y),
\]
as a finite sum over left coset representatives $x\ST(A)^0$ of $\ST(A)/\ST(A)^0$; see \cite[\S 5.1.1]{FKRS12} for details and explicit results in the case $g=2$.

\subsection{Trace moment sequences}

Having determined Haar measures for various Sato--Tate groups $\ST(A)$, let us now consider the problem of computing the trace moment sequence of a connected Sato--Tate group; so assume $\ST(A)=\ST(A)^0$.
For each integer $n\ge 0$ we wish to compute the $n$th moment
\[
\Exp_{\ST(A)}[\tr^n] = \int_0^\pi\cdots\int_0^\pi \left(\sum_{j=1}^g 2\cos\theta_j\right)^n\mu_{\ST(A)}(\theta_1,\ldots,\theta_g).
\]

We have already done this computation for the groups $\U(1)$ and $\SU(2)$ that arise in dimension $g=1$.  For $\U(1)$ we have
\[
\Exp_{\U(1)}[\tr^n] = \frac{1}{\pi}\int_0^\pi(2\cos\theta)^n\,d\theta = b_n \coloneqq \binom{n}{\nicefrac{n}{2}},
\]
where we adopt the convention that $\binom{n}{n/2}=0$ when $n$ is odd, and for $\SU(2)$ we have
\[
\Exp_{\SU(2)}[\tr^n] = \frac{2}{\pi}\int_0^\pi(2\cos\theta)^n\sin^2\theta\,d\theta = c_n \coloneqq \frac{2}{n+2}\binom{n}{\nicefrac{n}{2}}.
\]
We thus obtain the moment sequences
\begin{align*}
\Mom_{\U(1)}[\tr] &= (1,\, 0,\, 2,\, 0,\, 6,\, 0,\, 20,\, 0,\, 70,\, 0,\, 252, \ldots),\\
\Mom_{\SU(2)}[\tr] &= (1,\, 0,\, 1,\, 0,\, 2,\, 0,\, 5,\, 0,\, 14,\, 0,\, 42, \ldots).
\end{align*}

For $g=2$, observe that for 5 of the 6 connected Sato--Tate groups listed in Table~\ref{tab:st0g2} we can compute their trace moment sequences directly from the trace moment sequences for $\U(1)$ and $\SU(2)$; no integration is required.  For $\U(1)_2$ and $\SU(2)_2$ we simply have
\begin{align*}
E_{\U(1)_2}[\tr^n] &= E_{\U(1)}[2^n\tr^n] = 2^nb_n,\\
E_{\SU(2)_2}[\tr^n] &= E_{\SU(2)}[2^n\tr^n] = 2^nc_n,
\end{align*}
and for $\U(1)\times\U(1)$, $\U(1)\times\SU(2)$, $\SU(2)\times\SU(2)$ we take binomial convolutions to obtain\footnote{It is at this point we see the utility of starting our moment sequences at $\Mom_0$.}
\begin{align}
E_{\U(1)\times\U(1)}[\tr^n] &= \sum_{r=0}^n \binom{n}{r}E_{\U(1)}[\tr^r]E_{\U(1)}[\tr^{n-r}]=\sum_{r=0}^n \binom{n}{r}b_rb_{n-r}=b_n^2,\label{eq:u1u1}\\
E_{\U(1)\times\SU(2)}[\tr^n] &= \sum_{r=0}^n \binom{n}{r}E_{\U(1)}[\tr^r]E_{\SU(2)}[\tr^{n-r}]=\sum_{r=0}^n \binom{n}{r}b_rc_{n-r}=\tfrac{1}{2}c_nb_{n+2},\label{eq:u1su2}\\
E_{\SU(2)\times\SU(2)}[\tr^n] &= \sum_{r=0}^n \binom{n}{r}E_{\SU(2)}[\tr^r]E_{\SU(2)}[\tr^{n-r}]=\sum_{r=0}^n \binom{n}{r}c_rc_{n-r}=c_nc_{n+2}.\label{eq:su2su2}
\end{align}
For the generic case $\USp(4)$ we apply \eqref{haarUSp} with $g=2$ to obtain
\[
E_{\USp(4)}[\tr^n] = \tfrac{2^{n+3}}{\pi^2}\int_0^\pi\int_0^\pi(\cos\theta_1+\cos\theta_2)^n(\cos\theta_1-\cos\theta_2)^2\sin^2\theta_1\sin^2\theta_2\, d\theta_1 d\theta_2 = c_nc_{n+4} - c_{n+2}^2.
\]
Here we have applied the general determinantal formula from \cite[Thm.\ 1]{KS09} that allows one to compute the moment generating function of the $k$th eigenvalue power-sum in $\USp(2g)$.
Recall that the \emph{moment generating function} of a moment sequence $(m_0,m_1,m_2,\ldots)$ is the exponential generating function
\[
\mathcal M(z)\coloneqq \sum_{n=0}^\infty m_n\frac{z^n}{n!}.
\]
One uses exponential generating functions so that products of moment generating functions correspond to binomial convolutions of moment sequences; this means that if $\mathcal M_1(z)$ and $\mathcal M_2(z)$ are the moment generating functions of two independent random variable $X_1$ and $X_2$, then the moment generating function of $X_1+X_2$ is simply $\mathcal M_1(z)\mathcal M_2(z)$.

The determinantal formula for the first eigenvalue power-sum (the trace) is simply
\[
\mathcal M_{\USp(2g)}[\tr] = \det_{g\times g}\left (\mathcal C^{i+j-2}\right)_{ij},
\]
where $\mathcal C^m$ is the moment generating function defined by
\[
\mathcal C^m(z) \coloneqq \sum_{r=0}^m\binom{n}{r}\left(\mathcal B_{2r-n}-\mathcal B_{2r-n+2}\right),\qquad \mathcal B_s(z)\coloneqq \sum_{n=0}^\infty \frac{z^{2n+s}}{s!(n+s)!}.
\]
The function $\mathcal B_s(z)$ is related to a hyperbolic Bessel function of the first kind; see \cite[p. 13]{KS09} for details.

For the connected Sato--Tate groups that arise in dimension $g=2$ we obtain the moment sequences
\begin{align*}
\Mom_{\U(1)_2}[\tr] &= (1,\, 0,\, 8,\, 0,\, 96,\, 0,\, 1280,\, 0,\, 17920,\, 0,\, 258048,\, \ldots),\\
\Mom_{\SU(2)_2}[\tr] &= (1,\, 0,\, 4,\, 0,\, 32,\, 0,\, 320,\, 0,\, 3584,\, 0,\, 43008,\, \ldots),\\
\Mom_{\U(1)\times\U(1)}[\tr] &= (1,\, 0,\, 4,\, 0,\, 36,\, 0,\, 400,\, 0,\, 4900,\, 0,\, 63504,\, \ldots),\\
\Mom_{\U(1)\times\SU(2)}[\tr] &= (1,\, 0,\, 3,\, 0,\, 20,\, 0,\, 175,\, 0,\, 1764,\, 0,\, 19404,\, \ldots),\\
\Mom_{\SU(2)\times\SU(2)}[\tr] &= (1,\, 0,\, 2,\, 0,\, 10,\, 0,\, 70,\, 0,\, 588,\, 0,\, 5544,\, \ldots),\\
\Mom_{\USp(4)}[\tr] &= (1,\, 0,\, 1,\, 0,\, 3,\, 0,\, 14,\, 0,\, 84,\, 0,\, 594,\, \ldots),
\end{align*}
and for $g=3$ we have
\begin{align*}
\Mom_{\U(1)_3}[\tr] &= (1,\, 0,\, 18,\, 0,\, 486,\, 0,\, 14580,\, 0,\, 459270,\, 0,\, 14880348,\, \ldots),\\
\Mom_{\SU(2)_3}[\tr] &= (1,\, 0,\, 9,\, 0,\, 162,\, 0,\, 3645,\, 0,\, 91854,\, 0,\, 2480058,\, \ldots),\\
\Mom_{\U(1)\times\U(1)_2}[\tr] &= (1,\, 0,\, 10,\, 0,\, 198,\, 0,\, 4900,\, 0,\, 134470,\, 0,\, 3912300,\, \ldots),\\
\Mom_{\U(1)\times\SU(2)_2}[\tr] &= (1,\, 0,\, 6,\, 0,\, 86,\, 0,\, 1660,\, 0,\, 37254,\, 0,\, 916020,\, \ldots),\\
\Mom_{\SU(2)\times\U(1)_2}[\tr] &= (1,\, 0,\, 9,\, 0,\, 146,\, 0,\, 2965,\, 0,\, 68334,\, 0,\, 1707930,\, \ldots),\\
\Mom_{\SU(2)\times\SU(2)_2}[\tr] &= (1,\, 0,\, 5,\, 0,\, 58,\, 0,\, 925,\, 0,\, 17598,\, 0,\, 374850,\,\, \ldots),\\
\Mom_{\U(1)\times\U(1)\times\U(1)}[\tr] &= (1,\, 0,\, 6,\, 0,\, 90,\, 0,\, 1860,\, 0,\, 44730,\, 0,\, 1172556,\,\, \ldots),\\
\Mom_{\U(1)\times\U(1)\times\SU(2)}[\tr] &= (1,\, 0,\, 5,\, 0,\, 62,\, 0,\, 1065,\, 0,\, 21714,\, 0,\, 492366,\,\, \ldots),\\
\Mom_{\U(1)\times\SU(2)\times\SU(2)}[\tr] &= (1,\, 0,\, 4,\, 0,\, 40,\, 0,\, 570,\, 0,\, 9898,\, 0,\, 19521,\,\, \ldots),\\
\Mom_{\SU(2)\times\SU(2)\times\SU(2)}[\tr] &= (1,\, 0,\, 3,\, 0,\, 24,\, 0,\, 285,\, 0,\, 4242,\, 0,\, 73206,\,\, \ldots),\\
\Mom_{\U(1)\times\USp(4)}[\tr] &= (1,\, 0,\, 3,\, 0,\, 21,\, 0,\, 214,\, 0,\, 2758,\, 0,\, 41796,\,\, \ldots),\\
\Mom_{\SU(2)\times\USp(4)}[\tr] &= (1,\, 0,\, 2,\, 0,\, 11,\, 0,\, 94,\, 0,\, 1050,\, 0,\, 14076,\,\, \ldots),\\
\Mom_{\U(3)}[\tr] &= (1,\, 0,\, 2,\, 0,\, 12,\, 0,\, 120,\, 0,\, 1610,\, 0,\, 25956,\,\, \ldots),\\
\Mom_{\USp(6)}[\tr] &= (1,\, 0,\, 1,\, 0,\, 3,\, 0,\, 15,\, 0,\, 104,\, 0,\, 909,\, \ldots),
\end{align*}

Recall that for $g=1$ the trace moment sequence $(1,0,1,0,2,0,5,0,14,0,42,\ldots)$ of the generic Sato--Tate group $\SU(2)$ corresponds to the sequence of Catalan numbers with 0's inserted at the odd moments.
There is a standard combinatorial interpretation of this sequence: the $n$th moment counts the number of returning walks of length $n$ on a 1-dimensional integer lattice that stay to the right of the origin (there are no such walks when $n$ is odd, hence the odd moments are zero).

This combinatorial interpretation generalizes to higher genus.
For $g=2$ the trace moment sequence for the generic Sato--Tate group $\USp(4)$ counts returning walks on a $2$-dimensional integer lattice that satisfy $x_1\ge x_2 \ge 0$ (so now there are 3 walks of length 4, not just 2).
In general, for any $g\ge 1$ the trace moment sequence for the generic Sato--Tate group $\USp(2g)$ counts returning walks on a $g$-dimensional integer lattice that satisfy $x_1\ge \ldots \ge x_g\ge 0$; this follows from a general result of Grabiner and Magyar \cite{GM93} that relates the decomposition of tensor powers of certain representations of classical Lie groups to lattice paths that are constrained to lie in the closure of the fundamental \emph{Weyl chamber} of the corresponding Lie algebra (which can be defined as an intersection of hyperplanes orthogonal to elements of a basis for the root system).

This combinatorial feature has an interesting asymptotic consequence.  For any integers $g'\ge g>0$, the moment sequences $\Mom_{\USp(2g')}[\tr]$ and $\Mom_{\USp(2g)}[\tr]$ must agree up to the $2g$th moment; see Exercise~\ref{ex:walk}.
Thus the moments sequences $\Mom_{\USp(2g)}[\tr]$ converge to a limiting sequence as $g\to \infty$:

\begin{align*}
\Mom_{\USp(2)}[\tr] &= (1,\, 0,\, 1,\, 0,\, 2,\, 0,\, 5,\, 0,\, 14,\, 0,\, 42,\, \ldots),\\
\Mom_{\USp(4)}[\tr] &= (1,\, 0,\, 1,\, 0,\, 3,\, 0,\, 14,\, 0,\, 84,\, 0,\, 594,\, \ldots),\\
\Mom_{\USp(6)}[\tr] &= (1,\, 0,\, 1,\, 0,\, 3,\, 0,\, 15,\, 0,\, 104,\, 0,\, 909,\, \ldots),\\
\Mom_{\USp(8)}[\tr] &= (1,\, 0,\, 1,\, 0,\, 3,\, 0,\, 15,\, 0,\, 105,\, 0,\, 944\, \ldots).\\
&\ \ \vdots\\
\Mom_{\USp(\infty)}[\tr] &= (1,\, 0,\, 1,\, 0,\, 3,\, 0,\, 15,\, 0,\, 105,\, 0,\, 945,\, \ldots).
\end{align*}
\smallskip

The limiting sequence $\Mom_{\USp(\infty)}[\tr]$ is precisely the moment sequence of the standard normal distribution (mean $0$ and variance $1$); the $n$th moment is zero if $n$ is odd, and for even $n$ it is given by
\[
(n-1)!!\coloneqq n(n-2)(n-4)\cdots 3\cdot 1.
\]
Figure \ref{fig:g1234generic} shows the $a_1$-distributions for $g=1,2,3,4$, normalized to the same scale, which illustrates convergence to the standard normal distribution.

\subsection{Exercises}

\begin{exercise}
Give combinatorial proofs of the identities used in \eqref{eq:u1u1}, \eqref{eq:u1su2}, \eqref{eq:su2su2}.
\end{exercise}

\begin{exercise}
Using the combinatorial interpretation of the trace moment sequence $\Mom_{\USp(2g)}[\tr]$, prove that for $g' > g$ the moment sequences $\Mom_{\USp(2g')}[\tr]$ and $\Mom_{\USp(2g)}[\tr]$ agree up to the $2g$th moment but disagree at the $(2g+2)$th moment.  Then show that the limiting trace moment sequence $\Mom_{\USp(\infty)}[\tr]$ is equal to the moment sequence of the standard normal distribution.
\end{exercise}

\begin{exercise}\label{ex:walk}
Characterize each of the 6 trace moment sequences that arise for connected Sato--Tate groups in dimension $g=2$ by showing that each sequence counts returning walks on an 2-dimensional integer lattice that are constrained to a certain region of the plane.
\end{exercise}

\begin{exercise}
Similarly characterize the~14 trace moment sequences that arise for connected Sato--Tate groups in dimension $g=3$ in terms of returning walks on a 3-dimensional integer lattice.
\end{exercise}

\begin{exercise}
For each of the 5 non-generic connected Sato--Tate groups that arise in dimension $g=2$ compute the moment sequence for $a_2$, the quadratic coefficient of the characteristic polynomial.
\end{exercise}

\begin{figure}
\includegraphics[scale=0.24]{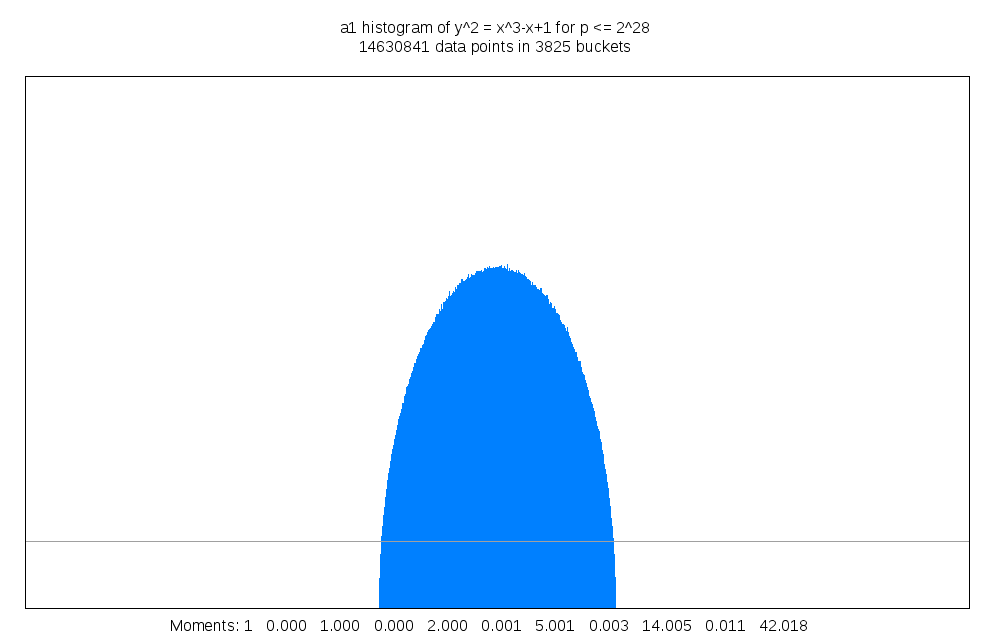}\\
\includegraphics[scale=0.24]{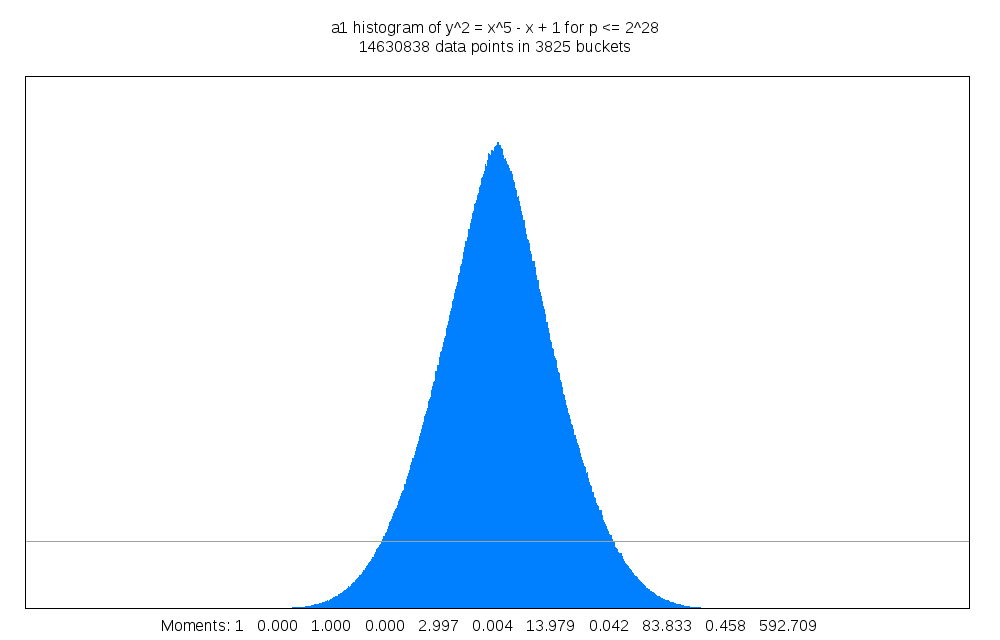}\\
\includegraphics[scale=0.24]{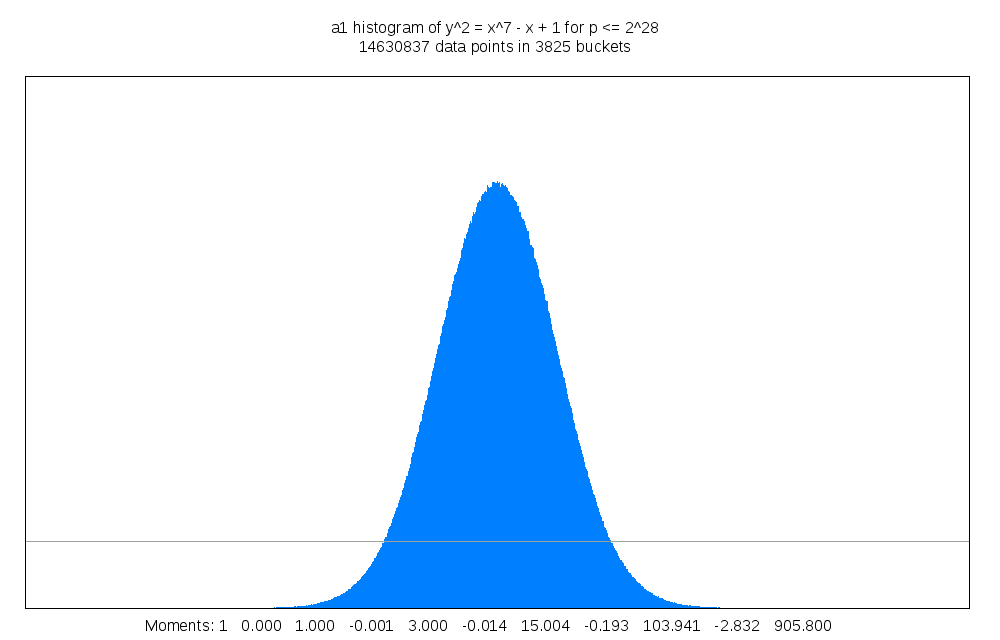}\\
\includegraphics[scale=0.24]{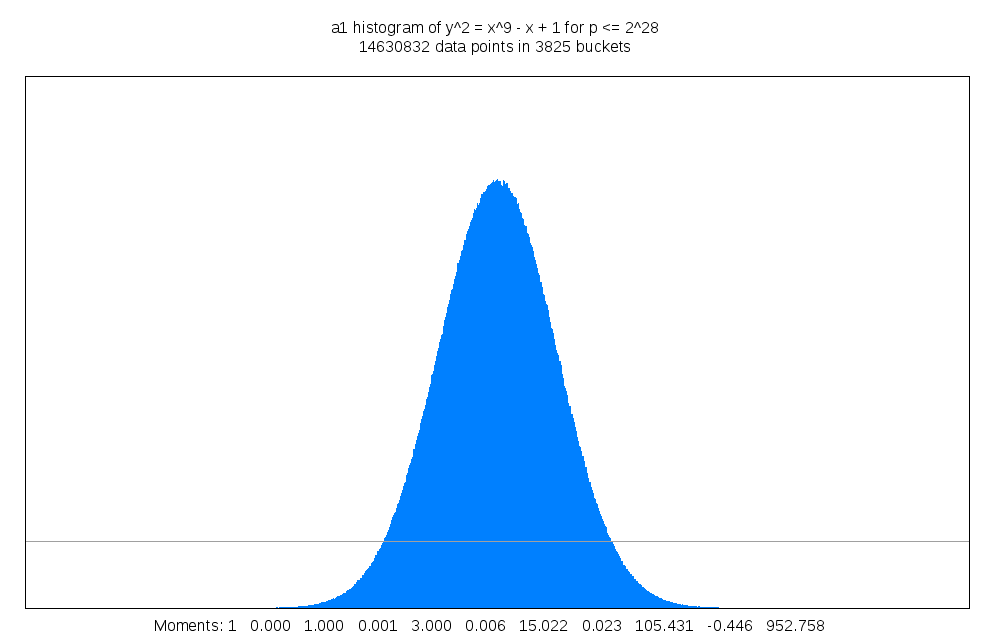}
\caption{Generic trace distributions for $g=1,2,3,4$ (shown with the same vertical scale).}\label{fig:g1234generic}
\end{figure}


\begin{thebibliography}{9}
\bibitem{AH03}
J. Achter and J. Holden, \href{http://www.numdam.org/item/JTNB_2003__15_3_627_0}{\textit{Notes on an analogue of the Fontaine-Mazur conjecture}}, Journal de Th\'eorie des Nombres de Bordeaux \textbf{15} (2003), 627--637.

\bibitem{AS12}
O. Ahmadi and I.\,E. Shparlinski, \href{http://dx.doi.org/10.4310/MRL.2010.v17.n4.a9}{\textit{On the distribution of the number of points on algebraic curves in extensions of finite fields}}, Mathematical Research Letters \textbf{17} (2012), 689--699.

\bibitem{ACCGHLNSTT}
P. Allen, F. Calegari, A. Caraiani, T. Gee, D. Helm, B. Le-Hung, J. Newton, P. Scholze, R. Taylor, J. Thorne, \href{http://www.math.ias.edu/files/Emerging1016report.pdf}{\textit{Applications to modularity of recent progress on the cohomology of Shimura Varieties}}, IAS working group report, 2016 (paper to appear).

\bibitem{BK15}
G. Banaszak and K.\,S. Kedlaya, \href{http://www.iumj.indiana.edu/IUMJ/fulltext.php?artid=5438&year=2015&volume=64}{\textit{An algebraic Sato-Tate group and Sato-Tate conjecture}}, Indiana University Mathematics Journal \textbf{64} (2015), 245--274.

\bibitem{BK16}
G. Banaszak and K.\,S. Kedlaya, \href{http://dx.doi.org/10.1090/conm/663/13348}{\textit{Motivic Serre group, algebraic Sato-Tate group and Sato-Tate conjecture}}, in \textit{Frobenius Distributions: Lang--Trotter and Sato--Tate Conjectures}, Contemporary Mathematics \textbf{663} (2016), AMS, 11--44.

\bibitem{BLGG11}
T. Barnet-Lamb, D. Geraghty, and T. Gee, \href{http://www.ams.org/journals/jams/2011-24-02/S0894-0347-2010-00689-3/}{\textit{The Sato-Tate conjecture for Hilbert modular forms}}, Journal of the American Mathematical Society \textbf{24} (2011), 411--469.

\bibitem{BLGHT11}
T. Barnet-Lamb, D. Geraghty, M. Harris, and R. Taylor, \href{http://www.math.ias.edu/~rtaylor/cy2fin.pdf}{\textit{A family of Calabi-Yau varieties and potential automorphy II}}, Publications of the Research Institute for Mathematical Sciences \textbf{47} (2011), 29--98.

\bibitem{BL04}
C. Birkenhake and H. Lange, \href{http://link.springer.com/book/10.1007/978-3-662-06307-1}{\textit{Complex abelian varieties}}, Springer, 2004.

\bibitem{Bo80}
F.A. Bogomolov, \href{http://www.ams.org/mathscinet-getitem?mr=574307}{\textit{Sur l'alg\'ebricit\'e des repr\'esentations l-adiques}}, Comptes Rendus Acad. Sci. Paris \textbf{290} (1980), 701--703.

\bibitem{Borel91}
A. Borel, \href{http://link.springer.com/book/10.1007/978-1-4612-0941-6}{\textit{Linear algebraic groups}}, second edition, 1991.

\bibitem{Magma}
Wieb Bosma, John Cannon, and Catherine Playoust, \href{http://www.sciencedirect.com/science/article/pii/S074771719690125X}{\textit{The Magma algebra system I: The user language}},  Journal of Symbolic Computation \textbf{24} (1997), 235--265.

\bibitem{BCDT01}
C. Breuil, B. Conrad, F. Diamond, and R. Taylor, \href{http://www.ams.org/journals/jams/2001-14-04/S0894-0347-01-00370-8/home.html}{\textit{On the modularity of elliptic curves over $\Q$: wild $3$-adic exercises}}, Journal of the American Mathematical Society \textbf{14} (2001), 843--939.

\bibitem{BuK16}
A. Bucur and K.\,S. Kedlaya, \href{http://dx.doi.org/10.1090/conm/663/13349}{\textit{An application of the effective Sato-Tate conjecture}}, in \textit{Frobenius Distributions: Lang--Trotter and Sato--Tate Conjectures}, Contemporary Mathematics \textbf{663} (2016), AMS, 45--56.

\bibitem{CF10}
J.\,W.\,S. Cassels and A. Fr\"ohlich, \href{https://www.lms.ac.uk/publications/algebraic-number-theory}{\textit{Algebraic number theory}}, second edition, London Mathematical Society, 2010.

\bibitem{CFHS12}
W. Castryck, A. Folsom, H. Hubrechts, and A.\,V. Sutherland, \href{https://doi.org/10.1112/plms/pdr063}{\textit{The probability that the number of points on the Jacobian of a genus 2 curve is prime}}, Proceedings of the London Mathematical Society \textbf{104} (2012), 1235--1270.

\bibitem{CHT08}
L. Clozel, M. Harris, and R. Taylor, \href{http://link.springer.com/journal/10240}{\textit{Automorphy for some $\ell$-adic lifts of automorphic mod-$\ell$ Galois representations}}, Publ. Math. IHES \textbf{108} (2008), 1--181.

\bibitem{CDSS15} A.\,C. Cojocaru, R. Davis, A. Silverberg, and K.\.E. Stange, \href{https://doi.org/10.1093/imrn/rnw0582}{\textit{Arithmetic properties of the Frobenius traces defined by a rational abelian variety}}, with two appendices by J-P. Serre, International Mathematics Research Notices \textbf{12} (2017), 3557--3602.

\bibitem{Del74}
P. Deligne, \href{http://www.numdam.org/item?id=PMIHES_1974__43__273_0}{\textit{La conjecture de Weil:\,I}}, Publ. Math. IHES \textbf{43} (1974), 273--307.

\bibitem{Del80}
P. Deligne, \href{http://www.numdam.org/item?id=PMIHES_1980__52__137_0}{\textit{La conjecture de Weil:\,II}}, Publ. Math. IHES \textbf{52} (1980), 173--252.

\bibitem{Del82}
P. Deligne, \href{http://www.jmilne.org/math/Documents/Deligne82.pdf}{\textit{Hodge cycles on abelian varieties (notes by J.S. Milne)}}, Lecture Notes in Mathematics \textbf{900} (1982), 9--100.

\bibitem{DS05}
F. Diamond and J. Shurman, \href{http://www.springer.com/us/book/9780387232294}{\textit{A first course in modular forms}}, Springer, 2005.

\bibitem{DS14}
J. Diestel and A. Spalsbury, \href{http://bookstore.ams.org/gsm-150}{\textit{The joys of Haar measure}}, Graduate Studies in Mathematics \textbf{150}, AMS, 2014.

\bibitem{Falt83}
G. Faltings, \href{http://link.springer.com/article/10.1007/BF01388432}{\textit{Endlichkeitss\"atze f\"ur abelsche Variet\"aten \"uber Zahlk\"orpern}}, Inventiones Mathematicae \textbf{73} (1983), 349--366.

\bibitem{Fite15}
F. Fit\'e, \href{http://dx.doi.org/10.1090/conm/649/13020}{\textit{Equidistribution, $L$-functions, and Sato--Tate groups}}, Contemporary Mathematics \textbf{649} (2015), 63--88.

\bibitem{FKRS12}
F. Fit\'e, K.\,S. Kedlaya, V. Rotger, and A.\,V. Sutherland, \href{https://doi.org/10.1112/S0010437X12000279}{\textit{Sato-Tate distributions and Galois endomorphism modules in genus $2$}}, Compositio Mathematica \textbf{148} (2012), 1390--1442.

\bibitem{FKS16}
F. Fit\'e, K.\,S. Kedlaya, and A.\,V. Sutherland, \href{http://dx.doi.org/10.1090/conm/663/13350}{\textit{Sato-Tate groups of some weight $3$ motives}}, in \textit{Frobenius Distributions: Lang--Trotter and Sato--Tate Conjectures}, Contemporary Mathematics \textbf{663}, AMS, 57--102.

\bibitem{FS14}
F. Fit\'e and A.\,V. Sutherland, \href{https://msp.org/ant/2014/8-3/p02.xhtml}{\textit{Sato-Tate distributions of twists of $y^2=x^5-x$ and $y^2=x^6+1$}}, Algebra and Number Theory~\textbf{8} (2014), 543--585.

\bibitem{FS16}
F. Fit\'e and A.\,V. Sutherland, \href{http://dx.doi.org/10.1090/conm/663/13351}{\textit{Sato-Tate groups of $y^2=x^8+c$ and $y^2=x^7-cx$}}, in \textit{Frobenius Distributions: Lang--Trotter and Sato--Tate Conjectures}, Contemporary Mathematics \textbf{663}, AMS, 103--126.

\bibitem{GG13}
Joachim von zur Gathen and J\"urgen Gerhard, \href{http://ebooks.cambridge.org/ebook.jsf?bid=CBO9781139856065}{\textit{Modern computer algebra}}, third edition, Cambridge University Press, 2013.

\bibitem{GM93}
D.\,J. Grabiner and P. Magyar, \href{http://users.math.msu.edu/users/magyar/papers/RandomWalk.pdf}{\textit{Random walks in Weyl chambers and the decomposition of tensor powers}}, Journal of Algebraic Combinatorics \textbf{2} (1993), 239--260.

\bibitem{Ha11}
C. Hall, \href{http://blms.oxfordjournals.org/content/43/4/703.abstract}{\textit{An open--image theorem for a general class of abelian varieties, with an appendix by E. Kowalski}}, Bulletin of the London Mathematical Society \textbf{43} (2011), 703--711.

\bibitem{HSBT10}
M. Harris, N. Shepherd-Barron, and R. Taylor, \href{http://annals.math.princeton.edu/2010/171-2/p04}{\textit{A family of Calabi-Yau varieties and potential automorphy}}, Annals of Mathematics \textbf{171} (2010), 779--813.

\bibitem{Har10}
W.\,B. Hart, \href{http://link.springer.com/chapter/10.1007/978-3-642-15582-6_18}{\textit{Fast Library for Number Theory: An introduction}}, in \href{http://link.springer.com/book/10.1007/978-3-642-15582-6}{\textit{Proceedings of the Third International Congress on Mathematical Software (ICMS 2010)}}, LNCS 6327, Springer, 2010, 88--91.

\bibitem{FLINT}
W.\,B. Hart, F. Johansson and S. Pancratz, \href{http://www.flintlib.org/}{\textit{{F}ast {L}ibrary for {N}umber {T}heory}}, version 2.5.2, \url{http://flintlib.org}, 2015.

\bibitem{H07}
D. Harvey, \href{https://doi.org/10.1093/imrn/rnm095}{\textit{Kedlaya's algorithm in larger characteristic}}, International Mathematics Research Notices \textbf{2007}.

\bibitem{H14}
D. Harvey, \href{http://annals.math.princeton.edu/2014/179-2/p07}{\textit{Counting points on hyperelliptic curves in average polynomial time}}, Annals of Mathematics \textbf{179} (2014), 783--803.

\bibitem{H16}
D. Harvey, \href{https://doi.org/10.1112/plms/pdv056}{\textit{Computing zeta functions of arithmetic schemes}}, Proceedings of the London Mathematical Society \textbf{111} (2015), 1379--1401.

\bibitem{HS14}
D. Harvey and A.\,V. Sutherland, \href{https://doi.org/10.1112/S1461157014000187}{\textit{Computing Hasse-Witt matrices of hyperelliptic curves in average polynomial time}}, in \textit{Algorithmic Number Theory 11th International Symposium (ANTS XI)}, LMS Journal of Computation and Mathematics \textbf{17} (2014), 257--273.

\bibitem{HS16}
D. Harvey and A.\,V. Sutherland, \href{http://dx.doi.org/10.1090/conm/663/13352}{\textit{Computing Hasse-Witt matrices of hyperelliptic curves in average polynomial time II}}, in \textit{Frobenius Distributions: Lang--Trotter and Sato--Tate Conjectures}, Contemporary Mathematics \textbf{663}, AMS, 127--148.

\bibitem{He20}
E. Hecke, \href{http://link.springer.com/article/10.1007/BF01202991}{\textit{Eine neue Art von Zetafunktionen und ihre Beziehungen zur Verteilung der Primzahlen. Zweite Mitteilung}}, Mathematische Zeitschrift \textbf{6} (1920), 11--51.

\bibitem{He77}
H. Heyer, \href{http://link.springer.com/book/10.1007/978-3-642-66706-0}{\textit{Probability measures on locally compact groups}}, Springer, 1977.

\bibitem{Hu75}
J.\,E. Humphreys, \href{http://link.springer.com/book/10.1007/978-1-4684-9443-3}{\textit{Linear algebraic groups}},
Springer, 1975.

\bibitem{J13}
C. Johansson, \href{http://dx.doi.org/10.1090/tran/6847}{\textit{On the Sato-Tate conjecture for non-generic abelian surfaces}}, with an appendix by Francesc Fit\'e, Transactions of the AMS \textbf{369} (2017), 6303--6325.

\bibitem{Kap47}
I. Kaplansky, \href{http://projecteuclid.org/euclid.bams/1183510807}{\textit{Lattices of continuous functions}}, Bulletin of the AMS \textbf{6} (1947), 617--623.

\bibitem{KaSa99}
N.\,M. Katz and P. Sarnak, \href{https://web.math.princeton.edu/~nmk/RMFEM.pdf}{\textit{Random matrices, Frobenius eigenvalues, and monodromy}}, Colloquium Publications \textbf{45}, AMS, 1999.

\bibitem{K01}
K.\,S. Kedlaya, \href{https://doi.org/10.1016/j.ffa.2005.01.003}{\textit{Counting points on hyperelliptic curves using Monsky-Washnitzer cohomology}}, Journal of the Ramanujan Mathematical Society \textbf{16} (2001), 323--338.

\bibitem{K04}
K.\,S. Kedlaya, \href{https://link.springer.com/chapter/10.1007/978-3-540-24847-7_1}{\textit{Computing zeta functions via $p$-adic cohomology}}, in \textit{Algorithmic Number Theory 6th International Symposium (ANTS VI)}, Lecture Notes in Computer Science \textbf{3076}, Springer 2004, 1--17.

\bibitem{KS08}
K.\,S. Kedlaya and A.\,V. Sutherland, \href{https://link.springer.com/chapter/10.1007/978-3-540-79456-1_21}{\textit{Computing $L$-series of hyperelliptic curves}}, in \textit{Algorithmic Number Theory 8th International Symposium (ANTS VIII)}, Lecture Notes in Computer Science \textbf{5011}, Springer, 2008, 312--326.

\bibitem{KS09}
K.\,S. Kedlaya and A.\,V. Sutherland, \href{http://dx.doi.org/10.1090/conm/487/09529}{\textit{Hyperelliptic curves, $L$-polynomials, and random matrices}}, in \textit{Arithmetic Geometry, Cryptography, and Coding Theory (AGCCT-11)}, Contemporary Mathematics \textbf{487}, American Mathematical Society, 2000, 119--162.

\bibitem{KU11}
K.\,S. Kedlaya and C. Umans, \href{https://doi.org/10.1137/08073408X}{\textit{Fast polynomial factorization and modular composition}}, SIAM Journal of Computing \textbf{40} (2011), 1767--1802.

\bibitem{KM01}
J. Kl\"uners and G. Malle, \href{http://dx.doi.org/10.1112/S1461157000000851}{\textit{A database for field extensions of the rationals}}, LMS Journal of Computation and Mathematics \textbf{4} (2001), 182--196.

\bibitem{Koo98}
P. Koosis, \href{http://dx.doi.org/10.1017/CBO9780511566196}{\textit{The logarithmic integral I}}, Cambridge University Press, 1998.

\bibitem{LT76}
S. Lang and H. Trotter, \href{http://link.springer.com/book/10.1007/BFb0082087}{\textit{Frobenius distributions in $\GL_2$-extensions}}, Lecture Notes in Mathematics \textbf{504} (1976), Springer.

\bibitem{Milne:AV}
J.\,S. Milne, \href{http://www.jmilne.org/math/CourseNotes/av.html}{\textit{Abelian varieties}}, v2.00, 2008.

\bibitem{Milne:iAG}
J.\,S. Milne, \href{http://www.jmilne.org/math/CourseNotes/iAG200.pdf}{\textit{Algebraic groups: An introduction to the theory of algebraic group schemes over fields}}, v2.00, 2015.

\bibitem{Mo04}
B. Moonen, \href{http://www.math.ru.nl/~bmoonen/Lecturenotes/MTGps.pdf}{\textit{An introduction to Mumford--Tate groups}}, Monte Verita lecture notes, available at \url{http://www.math.ru.nl/~bmoonen/Lecturenotes/MTGps.pdf}, 2004.

\bibitem{MZ99}
B. Moonen and Yu.G. Zarhin, \href{http://link.springer.com/article/10.1007/s002080050333}{\textit{Hodge classes on abelian varieties of low dimension}}, Mathematische Annalen \textbf{315} (1999), 711--733.

\bibitem{Mu69}
D. Mumford, \href{http://link.springer.com/article/10.1007/BF01350672}{\textit{A note on Shimura's paper ``Discontinuous subgroups and abelian varieties"}}, Mathematische Annalen \textbf{181} (1969), 345--351.

\bibitem{Mu74}
D. Mumford, \href{http://bookstore.ams.org/tifr-13}{\textit{Abelian varieties}}, second edition, Tata Institute of Fundamental Research Studies in Mathematics, 1974.

\bibitem{MM11}
M.\,R. Murty and V.\, K. Murty, \href{http://link.springer.com/book/10.1007/978-3-0348-0274-1}{\textit{Non-vanishing of $L$-functions and applications}}, Modern Birkh\"auser Classics, 1997.

\bibitem{Murty85}
V.\,K. Murty, \href{http://projecteuclid.org/download/pdf_1/euclid.rmjm/1250127227}{\textit{Explicit formulae and the Lang-Trotter conjecture}}, Rocky Mountain Journal of Mathematics \textbf{15} (1985), 535--551.

\bibitem{Ne99}
J. Neukirch, \href{http://link.springer.com/book/10.1007/978-3-662-03983-0}{\textit{Algebraic number theory}}, Springer, 1999.

\bibitem{OV94}
A.\,L. Onishchik and A.\,L. Vinberg (eds.), \href{http://www.springer.com/us/book/9783540546832}{\textit{Lie groups and Lie algebras III: Structure of Lie groups and Lie algebras}}, Springer, 1994.

\bibitem{OEIS}
OEIS Foundation Inc., \href{http://oeis.org}{\textit{The On-Line Encyclopedia of Integer Sequences}}, online database at \url{http://oeis.org}, 2016.

\bibitem{Pi90}
J. Pila, \href{http://www.ams.org/journals/mcom/1990-55-192/S0025-5718-1990-1035941-X/}{\textit{Frobenius maps of abelian varieties and finding roots of unity in finite fields}}, Mathematics of Computation \textbf{55} (1990), 745--763.


\bibitem{Rama07}
D. Ramakrishnan, \href{http://www.aimath.org/WWN/tateconjecture/tateconjecture.pdf}{\textit{Remarks on the Tate Conjecture for beginners}}, notes from the AIM Tate Conjecture Workshop, available at \url{http://www.aimath.org/WWN/tateconjecture/tateconjecture.pdf}, 2007.

\bibitem{Sage}
The Sage Developers, \href{http://www.sagemath.org}{\textit{Sage Mathematics Software}}, Version 7.0, available at \url{http://www.sagemath.org}, 2016.

\bibitem{Sa15}
W.F. Sawin, \href{http://www.sciencedirect.com/science/article/pii/S1631073X16000376}{\textit{Ordinary primes for abelian surfaces}}, Comptes Rendus Mathematique \textbf{354} (2016), 566--568.

\bibitem{Sc85}
R. Schoof, \href{http://www.ams.org/journals/mcom/1985-44-170/S0025-5718-1985-0777280-6/}{\textit{Elliptic curves over finite fields and the computation of square roots mod $p$}}, Mathematics of Computation \textbf{44} (1995), 483--494.

\bibitem{Sc95}
R. Schoof, \href{http://jtnb.cedram.org/jtnb-bin/item?id=JTNB_1995__7_1_219_0}{\textit{Counting points on elliptic curves over finite fields}}, Journal de Th\'eorie des Nombres de Bordeaux \textbf{7} (1995), 219--254.

\bibitem{Se68}
J.-P. Serre, \href{http://www.ams.org/mathscinet-getitem?mr=1484415}{\textit{Abelian $\ell$-adic representations and elliptic curves}}, Research Notes in Mathematics \textbf{7}, A.K. Peters, 1998.

\bibitem{Se72}
J.-P. Serre, \href{http://link.springer.com/article/10.1007/BF01405086}{\textit{Propri\'et\'es galoisiennes des points d'ordre fini des courbes elliptiques}}, Inventiones Mathematicae \textbf{15} (1972), 259--331.

\bibitem{Se86a}
J.-P. Serre, \href{https://books.google.com/books?id=62pe2GBLyKUC&lpg=PA325&dq=Propri%C2%B4et%C2%B4es%20conjecturales%20des%20groupes%20de%20Galois%20motiviques%20et%20des%20repr%C2%B4esentations%20%E2%84%93-adiques&pg=PA33#v=onepage&q=%22des%20cours%20de%201985%22&f=false}{\textit{R\'esum\'e des cours de 1985-1986, Annuaire du Coll\`ege de France, 1986, 95--99}}; in \textit{Oeuvres -- Collected Papers, Volume IV}, Springer, 2003, 33--37.

\bibitem{Se86b}
J.-P. Serre, \href{http://books.google.com/books?id=62pe2GBLyKUC&lpg=PA325&dq=Propri%C2%B4et%C2%B4es%20conjecturales%20des%20groupes%20de%20Galois%20motiviques%20et%20des%20repr%C2%B4esentations%20%E2%84%93-adiques&pg=PA38#v=onepage&q=Lettre%20a%20Marie-France&f=false}{\textit{Lettre \`a Marie-France Vign\'eras du 10/2/1986}}, pages 38--55 in \href{http://www.springer.com/us/book/9783642398391}{\textit{Oeuvres -- Collected Papers, Volume IV}}, Springer, 2000.

\bibitem{Se91}
J.-P. Serre, \href{http://books.google.com/books?id=62pe2GBLyKUC&lpg=PA325&dq=Propri%C2%B4et%C2%B4es%20conjecturales%20des%20groupes%20de%20Galois%20motiviques%20et%20des%20repr%C2%B4esentations%20%E2%84%93-adiques&pg=PA1#v=onepage&q=Ken%20Ribet%20du%201/1&f=false}{\textit{Lettres \`a Ken Ribet du 1/1/1981 et du 29/1/1981}}, pages 1--20 in \href{http://www.springer.com/us/book/9783642398391}{\textit{Oeuvres -- Collected Papers, Volume IV}}, Springer, 2000.

\bibitem{Se94}
J.-P. Serre, \href{http://books.google.com/books?id=62pe2GBLyKUC&lpg=PA325&dq=Propri%C2%B4et%C2%B4es%20conjecturales%20des%20groupes%20de%20Galois%20motiviques%20et%20des%20repr%C2%B4esentations%20%E2%84%93-adiques&pg=PA325#v=onepage&q=Propri%C2%B4et%C2%B4es%20conjecturales%20des%20groupes%20de%20Galois%20motiviques%20et%20des%20repr%C2%B4esentations%20%E2%84%93-adiques&f=false}{\textit{Propri\'et\'es conjecturales des groups de Galois motiviques et des repr\'esentations $\ell$-adiques}}, pages 377-400 in \href{http://bookstore.ams.org/pspum-55-1/}{\textit{Motives: Part I}}, U. Jansen, S. Kleiman, and J.-P. Serre eds., AMS Proceedings of Symposia in Pure Mathematics \textbf{55}, 1994.

\bibitem{Se12}
J.-P. Serre, \href{http://www.crcnetbase.com/isbn/9781466501935}{\textit{Lectures on $N_X(p)$}}, Research Notes in Mathematics \textbf{11}, CRC Press, 2012.

\bibitem{Sh15}
Y.-D. Shieh, \href{http://iml.univ-mrs.fr/~kohel/phd/thesis_shieh.pdf}{\textit{Arithmetic aspects of point counting and Frobenius distributions}}, Ph.D. thesis, Universit\'e d'Aix-Marseille, 2015.

\bibitem{Sh16}
Y.-D. Shieh, \href{http://doi.org/10.1112/S1461157016000279}{\textit{Character theory approach to Sato-Tate groups}}, in \textit{Algorithmic Number Theory 12th International Symposium (ANTS XII)}, LMS Journal of Computation and Mathematics \textbf{19} (2016), 301--314.

\bibitem{NTL}
V. Shoup, \href{http://www.shoup.net/ntl/}{\textit{NTL: A Library for doing Number Theory}}, version 9.6.4, available at \url{http://www.shoup.net/ntl/}, 2016.

\bibitem{Si94}
J.\,H. Silverman, \href{http://link.springer.com/book/10.1007/978-1-4612-0851-8}{\textit{Advanced topics in the arithmetic of elliptic curves}}, Springer, 1994.

\bibitem{Si09}
J.\,H. Silverman, \href{http://www.springerlink.com/content/978-0-387-09493-9}{\textit{The arithmetic of elliptic curves}}, second edition, Springer, 2009. 

\bibitem{Spr98}
T.\,A. Springer, \href{http://link.springer.com/book/10.1007/978-0-8176-4840-4}{\textit{Linear algebraic groups}}, second edition, Modern Birkh\"auser Classics, 1998.

\bibitem{Stan15}
R. Stanley, \href{http://www.cambridge.org/us/academic/subjects/mathematics/discrete-mathematics-information-theory-and-coding/catalan-numbers}{\textit{Catalan numbers}}, Cambridge University Press, 2015.

\bibitem{smalljac}
A.\,V. Sutherland, \href{http://math.mit.edu/~drew}{\texttt{smalljac}}, version 5.0, available at \url{http://math.mit.edu/~drew}, 2017.

\bibitem{S07}
A.\,V. Sutherland, \href{http://groups.csail.mit.edu/cis/theses/sutherland-phd.pdf}{\textit{Order computations in generic groups}}, PhD thesis, Massachusetts Institute of Technology, 2007.

\bibitem{S11}
A.\,V. Sutherland, \href{http://www.ams.org/journals/mcom/2011-80-273/S0025-5718-10-02356-2/}{\textit{Structure computation and discrete logarithms in finite abelian $p$-groups}}, Mathematics of Computation \textbf{80} (2011), 477--500.

\bibitem{Tate63}
J. Tate, \href{http://books.google.com/books?id=3T27DQAAQBAJ&lpg=PA195&dq=Algebraic%20cycles%20and%20poles%20of%20zeta%20functions&pg=PA195#v=onepage&q=Algebraic%20cycles%20and%20poles%20of%20zeta%20functions&f=false}{\textit{Algebraic cycles and poles of zeta functions}}, pp. 224--241 in \href{http://bookstore.ams.org/cworks-24-1/}{\textit{Collected works of John Tate: Part I (1951--1975)}}, B. Mazur and J.-P. Serre eds., American Mathematical Society, 2016.

\bibitem{Tate66}
J. Tate, \href{http://link.springer.com/article/10.1007/BF01404549}{\textit{Endomorphisms of abelian varieties over finite fields}}, Inventiones Mathematicae \textbf{2} (1966), 134--144.

\bibitem{Tay08}
R. Taylor, \href{http://link.springer.com/article/10.1007/s10240-008-0015-2}{\textit{Automorphy for some $\ell$-adic lifts of automorphic mod $\ell$ Galois representations II}}, Publ. Math. IHES \textbf{108} (2008) 183--239.

\bibitem{TW95}
R.\,Taylor and A. Wiles, \href{http://www.jstor.org/stable/2118560}{\textit{Ring-theoretic properties of certain Hecke algebras}}, Annals of Mathematics \textbf{141} (1995), 553--572.

\bibitem{Th15}
J. Thorner, \href{http://link.springer.com/article/10.1007/s00013-014-0673-x}{\textit{The error term in the Sato-Tate conjecture}}, Archiv der Mathematik \textbf{103} (2014), 147--156.

\bibitem{vW97}
P. van Wamelen, \href{http://www.sciencedirect.com/science/article/pii/S0022314X97920898}{\textit{On the CM character of the curves $y^2=x^q-1$}}, Journal of Number Theory \textbf{64} (1997), 59--83.

\bibitem{Weil45}
A. Weil, \href{http://www.ams.org/mathscinet-getitem?mr=27151}{\textit{Sur les courbes alg\'ebriques et les vari\'et\'es qui s'en d\'eduisent}}, Publ. Inst. Math. Univ. Strasbourg \textbf{7} (1945).

\bibitem{Weil46}
A. Weil, \href{http://www.ams.org/mathscinet-getitem?mr=29522}{\textit{Vari\'et\'es ab\'eliennes et courbes alg\'ebriques}}, Publ. Inst. Math. Univ. Strasbourg \textbf{8} (1946).

\bibitem{Weil49}
A. Weil, \href{http://www.ams.org/journals/bull/1949-55-05/S0002-9904-1949-09219-4/S0002-9904-1949-09219-4.pdf}{\textit{Numbers of solutions of equations in finite fields}}, Bulletin of the AMS \textbf{55} (1949), 497--508.

\bibitem{Weyl46}
H. Weyl, \href{http://press.princeton.edu/titles/2169.html}{\textit{The classical groups: their invariants and representations}}, Princeton University Press, 1966.

\bibitem{Wiles95}
A. Wiles, \href{http://www.jstor.org/stable/2118559}{\textit{Modular elliptic curves and Fermat's last theorem}}, Annals of Mathematics \textbf{141} (1995), 443--551.

\bibitem{Wulf61}
A. Wulfsohn, \href{http://dx.doi.org/10.1017/S030500410003663X}{\textit{A note on the vague topology for measures}}, Mathematical Proceedings of the Cambridge. Philosophical Society \textbf{58} (1962), 421--422.

\bibitem{Za00}
Yu. G. Zarhin, \href{http://dx.doi.org/10.4310/MRL.2000.v7.n1.a11}{\textit{Hyperelliptic Jacobians without complex multiplication}}, Mathematical Research Letters \textbf{7} (2000), 123--132.

\end{thebibliography}
\end{document}